\documentclass[12pt]{amsart}
\usepackage{amscd,verbatim}
\usepackage{amssymb}
\usepackage[all]{xy}
\usepackage[inline]{enumitem}
\usepackage{hyphenat}
\usepackage[colorlinks,linkcolor=blue,citecolor=blue,urlcolor=red]{hyperref}

\newcommand{\eq}[2]{\begin{equation}\label{#1}#2 \end{equation}}
\newcommand{\ml}[2]{\begin{multline}\label{#1}#2 \end{multline}}
\newcommand{\mlnl}[1]{\begin{multline*}#1 \end{multline*}}

\newcommand{\arir}{\ar@{^{(}->}}
\newcommand{\aril}{\ar@{_{(}->}}
\newcommand{\are}{\ar@{>>}}

\newcommand{\xr}[1] {\xrightarrow{#1}}

\newtheorem{lemma}{Lemma}[section]

\newtheorem{thm}[lemma]{Theorem}
\newtheorem{theorem}{Theorem}
\newtheorem{cor-intro}{Corollary}
\newtheorem{prop}[lemma]{Proposition}
\newtheorem{proposition}[lemma]{Proposition}
\newtheorem{cor}[lemma]{Corollary}
\newtheorem{corollary}[lemma]{Corollary}

\theoremstyle{definition}
\newtheorem{defn}[lemma]{Definition}

\newtheorem{definition}[lemma]{Definition}
\newtheorem{nota}[lemma]{Notation}
\newtheorem{para}[lemma]{}

\newtheorem{conventions}[lemma]{Conventions}
\newtheorem*{ack}{Acknowledgments}

\theoremstyle{remark}

\newtheorem{remark}[lemma]{Remark}

\newtheorem{remarks}[lemma]{Remarks}

\newtheorem{claim}{Claim}[lemma]
\newtheorem*{claim*}{Claim}

\newcounter{zaehler} 
\numberwithin{equation}{lemma}

\newcommand{\N}{\mathbb{N}}

\newcommand{\Q}{\mathbb{Q}}
\newcommand{\Z}{\mathbb{Z}}

\renewcommand{\P}{\mathbf{P}}
\newcommand{\A}{\mathbf{A}}
\newcommand{\F}{\mathbf{F}}
\newcommand{\G}{\mathbf{G}}

\newcommand{\sH}{\mathcal{H}}

\newcommand{\sO}{\mathcal{O}}

\newcommand{\sU}{\mathcal{U}}

\newcommand{\sX}{\mathcal{X}}
\newcommand{\sY}{\mathcal{Y}}

\newcommand{\bZ}{\mathbb{Z}}
\newcommand{\fm}{\mathfrak{m}}

\newcommand{\fp}{\mathfrak{p}}

\newcommand{\Xb}{{\overline{X}}}

\newcommand{\Cor}{\operatorname{\mathbf{Cor}}}
\newcommand{\ProCor}{\operatorname{\mathbf{ProCor}}}
\newcommand{\ProulMCor}{\operatorname{\mathbf{Pro\underline{M}Cor}}}

\newcommand{\HI}{\operatorname{\mathbf{HI}}}

\newcommand{\RSC}{{\operatorname{\mathbf{RSC}}}}
\newcommand{\RSCNis}{{\operatorname{\mathbf{RSC}}}_{\Nis}}

\newcommand{\ul}[1]{{\underline{#1}}}

\newcommand{\PST}{{\operatorname{\mathbf{PST}}}}

\newcommand{\NST}{\operatorname{\mathbf{NST}}}

\newcommand{\Hom}{\operatorname{Hom}}

\newcommand{\Ker}{\operatorname{Ker}}

\renewcommand{\Im}{\operatorname{Im}}
\newcommand{\Coker}{\operatorname{Coker}}
\newcommand{\Tr}{\operatorname{Tr}}
\newcommand{\Nm}{\operatorname{Nm}}
\newcommand{\Div}{\operatorname{Div}}
\newcommand{\Pic}{\operatorname{Pic}}

\newcommand{\Spec}{\operatorname{Spec}}

\newcommand{\Art}{\operatorname{Art}}
\newcommand{\RoSe}{\operatorname{RoSe}}
\newcommand{\wR}{\operatorname{R}}
\newcommand{\wV}{\operatorname{V}}
\newcommand{\wF}{\operatorname{F}}
\newcommand{\td}{\operatorname{trdeg}}

\newcommand{\ab}{{\rm ab}}
\newcommand{\dlog}{\operatorname{dlog}}
\newcommand{\Gal}{\operatorname{Gal}}
\newcommand{\Qlb}{{\overline{\mathbb{Q}}_{\ell}}}

\newcommand{\Sm}{\operatorname{\mathbf{Sm}}}
\newcommand{\ProSm}{\operatorname{\mathbf{ProSm}}}

\newcommand{\Ab}{\operatorname{\mathbf{Ab}}}

\newcommand{\tr}{{\operatorname{tr}}}

\newcommand{\red}{{\operatorname{red}}}
\newcommand{\cont}{{\operatorname{cont}}}
\newcommand{\Zar}{{\operatorname{Zar}}}
\newcommand{\Nis}{{\operatorname{Nis}}}
\newcommand{\fppf}{{\operatorname{fppf}}}
\newcommand{\dR}{{\operatorname{dR}}}
\newcommand{\et}{{\operatorname{\acute{e}t}}}
\newcommand{\irr}{{\operatorname{irr}}}

\newcommand{\inj}{\hookrightarrow}

\newcommand{\surj}{\rightarrow\!\!\!\!\!\rightarrow}

\newcommand{\Res}{\operatorname{Res}}

\newcommand{\id}{{\operatorname{id}}}

\newcommand{\divi}{{\operatorname{div}}}

\newcommand{\ch}{{\operatorname{ch}}}
\newcommand{\Sym}{{\operatorname{Sym}}}

\newcommand{\Frac}{{\operatorname{Frac}}}
\newcommand{\fillog}[1]{{\operatorname{fil}^{\rm log}_{#1}}}
\newcommand{\fil}[1]{{\operatorname{fil}_{#1}}}
\newcommand{\filF}[1]{{\operatorname{fil}^{F}_{#1}}}
\newcommand{\gr}{{\operatorname{gr}}}

\newcommand{\colim}{\operatornamewithlimits{\varinjlim}}

\newcommand{\ol}{\overline}

\renewcommand{\phi}{\varphi}
\renewcommand{\epsilon}{\varepsilon}

\newcommand{\MNST}{\operatorname{\mathbf{MNST}}}

\newcommand{\MCor}{\operatorname{\mathbf{MCor}}}

\newcommand{\MSm}{\operatorname{\mathbf{MSm}}}

\newcommand{\MPST}{\operatorname{\mathbf{MPST}}}
\newcommand{\CI}{\operatorname{\mathbf{CI}}}

\newcommand{\bcube}{{\ol{\square}}}

\newcommand{\ulMPST}{\operatorname{\mathbf{\underline{M}PST}}}
\newcommand{\ulMNST}{\operatorname{\mathbf{\underline{M}NST}}}

\newcommand{\ulMCor}{\operatorname{\mathbf{\underline{M}Cor}}}

\newcommand{\Comp}{\operatorname{\mathbf{Comp}}}

\def\bZ{\mathbb{Z}}
\def\Ztr{\bZ_\tr}

\def\tF{\widetilde{F}}
\def\tG{\widetilde{G}}

\def\Xinf{X_{\infty}}

\def\Func{\mathrm{Func}}
\def\FunPhiF{\Func_\Phi(F,\N)}
\def\FunPhiFn{\Func_\Phi(F,\N)_{\leq n}}
\def\Phin{\Phi_{\leq n}}
\def\CondscF{\mathrm{Cond}(F)^{sc}}
\def\CondscFn{\mathrm{Cond}(F)^{sc}_{\leq n}}
\def\tF{\tilde{F}}

\def\hF{\hat{F}}
\def\hFc{\hat{F}_c}

\begin{document}
\title{Reciprocity sheaves and their ramification filtrations}

\begin{abstract}
We define a motivic conductor for any presheaf  with transfers $F$ using the 
categorical framework developed for the theory of motives with modulus by Kahn-Miyazaki-Saito-Yamazaki.
If $F$ is a reciprocity sheaf this conductor yields an increasing and exhaustive filtration
on $F(L)$, where $L$ is any henselian discrete valuation field of geometric type over the perfect ground field.
We show if $F$ is a smooth group scheme, then the motivic conductor extends the Rosenlicht-Serre conductor;
if $F$ assigns to $X$ the group of finite characters on the abelianized \'etale fundamental group of $X$, then the motivic conductor
agrees with the Artin conductor defined by Kato-Matsuda; if $F$ assigns to $X$ the group of integrable rank one connections
(in characteristic zero), then it agrees with the irregularity. We also show that this machinery gives rise to a conductor
for torsors under finite flat group schemes over the base field, which we believe to be new.
We introduce a general notion of conductors on
presheaves with transfers and show that on a reciprocity sheaf the motivic conductor is minimal and
any conductor which is defined only for henselian discrete valuation fields of geometric type 
with {\em perfect} residue field can be uniquely extended to all such fields without any restriction on the residue field.
For example the Kato-Matsuda Artin conductor is characterized as the canonical extension 
of the classical Artin conductor defined in the perfect residue field case.
  \end{abstract}

\author{Kay R\"ulling and Shuji Saito}

\address{Bergische Universit\"at Wuppertal, Gau\ss{}str 20, 42119 Wuppertal, and}
\address{Technische Universit\"at M\"unchen,Boltzmannstr. 3, 85748 Garching}
\email{ruelling@uni-wuppertal.de}
\address{Graduate School of Mathematical Sciences, University of Tokyo}
\email{sshuji@msb.biglobe.ne.jp}

\thanks{The first  author is supported by the DFG Heisenberg Grant RU 1412/2-2.
Part of the work was done while he was a visiting professor at the TU M{\"u}nchen.
He thanks Eva Viehmann for the invitation and the support.
The second author is supported by JSPS KAKENHI Grant (15H03606) and the DFG SFB/CRC 1085 ``Higher Invariants''.}

\maketitle

\tableofcontents

\section{Introduction}

Fix a perfect field $k$ and let $\Sm$ be the category of separated smooth $k$-schemes. Let $\Cor$ be the category 
of finite correspondences:
$\Cor$ has the same objects as $\Sm$ and morphisms in $\Cor$ are finite correspondences (see \ref{para:Cor}
 for a precise definition). 
Let $\PST$ be the category of additive presheaves of abelian groups on $\Cor$, called presheaves with transfers. 
In this note we give a construction which associates to each $F\in \PST$ a collection of functions
\begin{equation*}
c^F =\{c^F_L: F(L) \to \N\cup \{\infty\}\}_{L\in \Phi},
\end{equation*}
where $\N$ is the set of non-negative integers, $\Phi$ is the collection of henselian discrete valuation fields which are the fraction 
fields of the henselization $\sO_{X,x}^h$ of $X\in \Sm$ at points  $x$ of codimension one in $X$, and 
\[ F(L) = \colim_{V} F(V-D_x),\]
where $V\to X$ ranges over \'etale neighborhoods of $x$ and $D_x$ is the closure of $x$ in $V$. 
We call $c^F$ \emph{the motivic conductor for $F$}. 
Our main aim is to convince the reader that our construction deserves such a pretentious terminology.
Indeed, it gives a unified way to understand different conductors such as 
the Artin conductor of a character of the abelian fundamental group
$\pi_1^{\ab}(X)$ with $X\in \Sm$ along a boundary of $X$,
the Rosenlicht-Serre conductor of a morphism from a curve to a commutative algebraic $k$-group, 
and the irregularity of a line bundle with connections on
$X\in \Sm$ along a boundary of $X$ . 
It also gives rise to a new conductor for  $G$-torsors with $G$ a finite flat 
$k$-group scheme. The latter conductor specializes to the classical Artin conductor in case $G$ is constant.

Our construction of the motivic conductors is rather simple once we have the new categorical framework introduced 
in \cite{KMSY-MotModI}, \cite{KMSY-MotModII} at our disposal (see \eqref{eq2;cF} below).
The main aim of {\em loc. cit.} is to develop a theory of \emph{motives with modulus} 
generalizing Voevodsky's theory of motives in order to capture non-$\A^1$-invariant phenomena and objects.
The basic principle is that the category $\Cor$ should be replaced by the larger category of \emph{modulus pairs}, 
$\ulMCor$: Objects are pairs 
$\sX=(\ol{X}, \Xinf)$ consisting of a separated $k$-scheme of finite type $\ol{X}$ and an effective (possibly empty) 
Cartier divisor $\Xinf$ on it such that the complement 
$\ol{X}\setminus \Xinf$ is smooth. Morphisms are given by finite correspondences between the smooth 
complements satisfying certain admissibility conditions (see \S\ref{sec:modulus} for the precise definition). 
Let $\MCor\subset \ulMCor$ be the full subcategory consisting of objects $(\Xb,\Xinf)$
with $\Xb$ proper over $k$. 
We then define $\ulMPST$ (resp. $\MPST$) as the category of additive presheaves of abelian groups 
on $\ulMCor$ (resp. $\MCor$). We have a functor 
\[\omega:\ulMCor \to \Cor, \quad (\Xb,\Xinf) \mapsto \Xb - |\Xinf|,\]
and two pairs of adjunctions
\begin{equation*}
\MPST\begin{smallmatrix} \tau^*\\ \longleftarrow\\ \tau_!\\ \longrightarrow\\
\end{smallmatrix}\ulMPST, 
\quad
\MPST
\begin{smallmatrix} \omega^*\\ \longleftarrow\\ \omega_!\\ \longrightarrow\\
\end{smallmatrix}\PST,
\end{equation*}
where $\tau^*$ is induced by the inclusion $\tau:\MCor\to \ulMCor$ and
$\tau_!$ is its left Kan extension, and $\omega^*$ is induced by $\omega$ and $\omega_!$ is its left Kan extension 
(see \ref{para:modulusII} for more concrete descriptions of these functors). 
A basic notion is the $\bcube$-invariance, where $\bcube=(\P^1,\infty)\in \MCor$: $F\in \MPST$ is called $\bcube$-invariant 
if $F(\sX)\simeq F(\sX\otimes\bcube)$ for all $\sX\in \MCor$  
(see \ref{para:modulus} for the tensor product $\otimes$ in $\MCor$). 
It is an analogue of the $\A^1$-invariance\footnote{Recall $F\in \PST$ is $\A^1$-invariant if $F(X)\simeq F(X\times\A^1)$ 
for all $X\in \Sm$.} exploited by Voevodsky in his theory of motives.
We write $\CI$ for the full subcategory of $\MPST$ consisting of $\bcube$-invariant objects. 
We know (\cite[Lem 2.1.7]{KSY-RecII}) that the inclusion $\CI\to \MPST$ admits 
a right adjoint $h_\bcube^0$ which associates to $F\in \MPST$ the maximal $\bcube$-invariant subobject of $F$.
We define the functor
\[ \omega^{\CI} : \PST \xr {\omega^*} \MPST \xr{h_\bcube^0} \CI,\]
and write $\tilde{F}=\tau_!\omega^{\CI} F\in \ulMPST$, for $F\in \PST$.
Then the motivic conductor $c^F$ for $F\in \PST$ is defined by
\begin{equation}\label{eq2;cF} 
c^F_L(a) =\min\{n|\; a\in \tilde{F}(\sO_L,\fm_L^{-n})\}, \quad\text{for } a\in F(L).
\end{equation}
Here, for $G\in \ulMPST$,  $L=\Frac(\sO_{X,x}^h)\in \Phi$, and $n\in \Z_{\geq 1}$, we put 
\[ G(\sO_L,\fm_L^{-n}) = \colim_{V} G(V,nD_x),\]
where $V\to X$ ranges over \'etale neighborhoods of $x$ and $D_x$ is the closure of $x$ in $V$ and 
$n D_x$ is its $n$-th thickening in $V$. 
By convention, 
\[G(\sO_L,\fm_L^{-n}) = G(\sO_L) = \colim_{V} G(V),\quad \text{for } n=0.\]
For $G=\tF$ there are natural inclusions $\tF(\sO_L,\fm^{-n}_L)\inj F(L)$, which are used to define \eqref{eq2;cF}.
It turns out that 
$\{\tilde{F}(\sO_L,\fm_L^{-n})\}_{n\in \Z_{\geq 0}}$
induces an increasing filtration on $F(L)$ which is exhaustive if $F\in \RSC$.  
Here $\RSC$ is the full subcategory of $\PST$ consisting of the objects belonging to the essential image of $\CI$ under $\omega_!$.
Objects of $\RSC$ are called \emph{reciprocity presheaves} and play a key role in this note. 
We know (see \cite[Cor 2.3.4]{KSY-RecII}) that 
$\RSC$ contains all $\A^1$-invariant objects in $\PST$. 
Moreover it contains many interesting objects $F$ which are not $\A^1$-invariant. 
In this note we consider in particular the following examples
($X$ runs over objects of $\Sm$):
\begin{enumerate}[label = (\roman*)]
\item\label{intro:ex1}
$F(X)=\Hom_{\Sm}(X,\Gamma)$, where $\Gamma$ is a smooth commutative algebraic $k$-group 
which may have non-trivial unipotent part (for example $\Gamma=\G_a$). 
\item\label{intro:ex2}
$F(X)=H^1_{\et}(X,\Q/\Z)=\Hom_{\cont}(\pi_1(X)^{\ab},\Q/\Z)$.
\item\label{intro:ex3}
$F(X)={\rm Conn}^1(X)$ (resp. ${\rm Conn}_{\rm int}^1(X)$) the group of isomorphism classes of (resp. integrable) 
rank $1$ connections on $X$.
Here we assume $\ch(k)=0$.
\item\label{intro:ex4}
$F(X)=H^1_{\fppf}(X,\Gamma)$,  where $\Gamma$ is a finite flat $k$-group.
\end{enumerate}
We prove the following (see Theorems \ref{thm:rose-cond}, \ref{thm:filF-motivic},  \ref{thm:Art-motivic}, and 
\ref{thm:irr-mot} for the precise statements).

\begin{theorem}\label{intro:thm1}
\begin{enumerate}[label = (\arabic*)]
\item
In case \ref{intro:ex1}, $c^F_L$ agrees with the conductor of Rosenlicht\hyp{}Serre (\cite{SerreGACC})
if $L$ has perfect residue field.
If $\ch(k)=p$ is positive and $F=W_n$ is the group scheme of $p$-typical Witt vectors of length $n$,
 then $c^F_L$ agrees  with a conductor
defined by Kato-Russell in \cite{Kato-Russell} for any $L$.
\item
In case \ref{intro:ex2}, $c^F_L$ agrees with the Artin conductor $\Art_L$ of Kato-Matsuda 
(see \S7.1)\footnote{It coincides with the classical Artin conductor if
$L$ has perfect residue field.}. 
\item
In case \ref{intro:ex3}, $c^F$ agrees with the irregularity of connections. 
\end{enumerate}
\end{theorem}

As far as we know, the motivic conductor $c^F$ in the case \ref{intro:ex4} is new
and we give an explicit description only in case the infinitesimal unipotent part of $G$ is $\alpha_p$, where $p=\ch(k)$ 
(see  Theorem \ref{thm-H1G-RSC}).  
\medbreak

An amusing application of the motivic conductor $c^F$ is to give an explicit description of the maximal $\A^1$-invariant part of $F$:
 Let $\HI\subset \PST$ be the full subcategory of $\A^1$-invariant objects. 
The inclusion $\HI\to \PST$ admits a right adjoint $h^0_{\A^1}$ which associates to $F\in\PST$ the 
maximal $\A^1$-invariant subobject of $F$ (see \ref{para:HIsub} for an explicit description of $h^0_{\A^1}$). 
Let $\NST\subset \PST$ be the full subcategory of Nisnevich sheaves, i.e.,
 those objects $F\in \PST$ whose restrictions to $\Sm\subset \Cor$ are sheaves with respect to the Nisnevich topology.

\begin{theorem}\label{intro:thm2}
For $F\in \RSC\cap \NST$ and $X\in \Sm$, we have
\[h^0_{\A^1}(F)(X)  = \underset{\rho}{\bigcap} \; \{a\in F(X)|\; c^F(\rho^*a)\leq 1\},\]
where $\rho$ ranges over all morphisms $\Spec L \to X$ with $L\in \Phi$.
\end{theorem}
 
In case $F=H^1_{\et}(-,\Q/\Z)$ from \ref{intro:ex2} (resp. $F={\rm Conn}^1_{\rm int}$ from \ref{intro:ex3}),
Theorem \ref{intro:thm2} asserts that the maximal $\A^1$-invariant part of $F$ is 
precisely the subsheaf of tame characters (resp. regular singular connections).

\medbreak

In what follows we fix $F\in \RSC\cap \NST$ and introduce a class of collections of functions
\begin{equation*}
c =\{c_L: F(L) \to \N\}_{L\in \Phi}
\end{equation*}
which may be called {\em conductors for $F$}.
Let $\FunPhiF$ be the partially ordered set consisting of collections of functions with
partial order given by  $c\leq c'$, if $c_L(a)\leq c'_L(a)$ for all $L\in \Phi$ and $a\in F(L)$. 
Let $\CI(F)$ be the partially ordered set consisting of subobjects
$G$ of $\omega^{\CI} F$ such that the induced maps
$\omega_! G \to \omega_! \omega^{\CI} F$ 
are isomorphisms and with partial order given by inclusion.
Then every $G\in \CI(F)$ gives rise to an exhaustive increasing filtration 
$\{\tau_!G(\sO_L,\fm_L^{-n})\}_{n\geq 0}$ on $F(L)$ and we define $c^G\in \FunPhiF$ by
\begin{equation*}
c^G_L(a) =\min\{n\mid a\in \tau_!G(\sO_L,\fm_L^{-n})\},\quad \text{for } a\in F(L).
\end{equation*}
By definition the motivic conductor $c^F$ of $F$ is $c^{\omega^{\CI} F}$ and $c^F\leq c^G$, for all $G\in \CI(F)$. 
Now a question is whether there is a simple characterization of the poset
$\{c^G|\; G\in \CI(F)\}$ in $\FunPhiF$. We answer it in the following refined form.
Let $n$ be a positive integer or $\infty$.
Let $\Phin \subset \Phi$ be the collection of such $L$ that $\td_k(L)\leq n$.
(Note that in positive characteristic $\Phi_{\leq 1}$ consists precisely of those $L\in \Phi$ that have a perfect residue field.)
Let $\FunPhiFn$ be the poset consiting of collections of functions 
\begin{equation*}
c =\{c_L: F(L) \to \N\}_{L\in \Phin}
\end{equation*} 
with partial order defined in the same manner as $\FunPhiF$.
There is an obvious restriction functor
\begin{equation}\label{restrictionFunPhiF}
\FunPhiF \to \FunPhiFn, \quad c \mapsto c^{\leq n}.
\end{equation}
We then introduce six axioms \ref{c1} through \ref{c6} for $\FunPhiFn$
(cf. Definitions \ref{defn:cond} and \ref{defn:c6}) and call those elements satisfying the axioms 
\emph{semi-continuous conductors of level $n$}. 
Let $\CondscFn$ be the  sub-poset of $\FunPhiFn$ consisting of such objects.
Write $\CondscF$ for $\CondscFn$ with $n=\infty$.
\footnote{There is one axiom \ref{c4} which is not preserved by the functor \eqref{restrictionFunPhiF}. 
So it does not induce  $\CondscF \to \CondscFn$.}
For example,  for $F=H^1_{\et}(-,\Q/\Z)$ from \ref{intro:ex2}, the classical Artin conductor $\{\Art_L\}_{L\in \Phi_{\leq 1}}$ 
is an element of $\CondscF_{\leq 1}$ and
the Kato-Matsuda conductor $\{\Art_L\}_{L\in \Phi}$ is an element of $\CondscF$.
We show the following (see Theorem \ref{thm:summary}).
\begin{theorem}\label{intro:thm3}
\begin{enumerate}[label= (\arabic*)]
\item
$c^G\in \CondscF$ for every $G\in \CI(F)$.
\item
There exists an order reversing map
\[ \CondscFn\to \CI(F),\quad c \mapsto \hFc \]
such that $c=(c^{\hFc})^{\leq n}$. 
For $\sX=(\Xb,\Xinf)\in \MCor$ with $X=\Xb-|\Xinf|$ we have
\[ \hFc(\sX) = \{a\in F(X)\;| \; c_{\Xb}(a)\leq \Xinf\},\]
where $c_{\Xb}(a)\leq \Xinf$ means that for any $L\in \Phin$ and any morphism
$\rho: \Spec \sO_L \to \Xb$ such that $\rho(\Spec L)\in X$, 
$c_L(\rho^*a)$ is not more than the multiplicity of 
the pullback of $\Xinf$ along $\rho$. 
\end{enumerate}
\end{theorem}

As a consequence, we obtain the following (see Theorem \ref{thm:summary}\ref{thm:summary4}).

\begin{cor-intro}\label{canonicalextension}
There exists a unique map 
\[ \CondscFn \to \CondscF,\quad  c \mapsto c^\infty,\]
such that $\hFc=\hF_{c^\infty}$ and that $c=(c^\infty)^{\leq n}$.
\end{cor-intro}

We call $c^\infty$ the \emph{canonical extension of $c$}.
For example, the Kato-Matsuda Artin conductor is the canonical extension of the classical Artin conductor. 
We say $F$ has level $n$, if $(c^F)^{\leq n}\in \CondscFn$;
in this case $c^F$ is the canonical extension of $(c^{F})^{\le n}$, by Theorem \ref{thm:summary}\ref{thm:summary5}.
We show  that $F= H^1_{\et}(-,\Q/\Z)$ in  \ref{intro:ex2} is of level $1$ (see Theorem \ref{thm:Art-motivic}), 
$F={\rm Conn}^1$ (resp. $F={\rm Conn}_{\rm int}^1$) 
from \ref{intro:ex3} is of level $2$ (resp. $1$) (see Theorem \ref{thm:irr-mot}), 
and $F=H^1_{\fppf}(-,\Gamma)$ from \ref{intro:ex4}
is of level $1$ if the infinitesimal unipotent part of $\Gamma$ is trivial and else is of level $2$ (see Theorem \ref{thm-H1G-RSC}).

\medbreak

\medbreak

We give a description of the content of each section: In section \ref{sec:proPST} we explain how to extend a presheaf 
with transfers to the category of 
regular schemes over $k$ which are pro-smooth; this is well-known and we include it only for  lack of reference.
In section \ref{sec:modulus} we recall the necessary constructions and results from the theory of motives with modulus
as developed in \cite{KMSY-MotModI}, \cite{KMSY-MotModII}, \cite{KSY-Rec}, \cite{KSY-RecII}, and \cite{Saito-Purity-Rec}.
Then we introduce in section \ref{sec:cond} the notion of (semi-continuous) conductors and 
prove Theorems \ref{intro:thm3} and 
\ref{intro:thm2}. 
We close the section with a discussion of the relation between the 
motivic conductor of a reciprocity sheaf with certain vanishing properties of its associated symbol.
This is needed in order to prove in the later sections that a certain conductor is equal to the motivic one;
 the main point being Corollary \ref{cor:localLS3}. In the second  part we consider various 
conductors which are mostly  classical and show that they are motivic in our sense. 
K\"ahler differentials and rank one connections are considered in section \ref{sec:Conn1}, where $\ch(k)=0$. 
In the following sections we assume $\ch(k)=p>0$. 
In section \ref{subsec:Witt} it is shown that one of the conductors defined by Kato-Russell for $W_n$ is motivic.
We use this in section \ref{subsec:art} to show that the Kato-Matsuda conductor for characters is motivic,
which yields also a description of the motivic conductor for lisse $\bar{\Q}_{\ell}$-sheaves of rank 1.
Finally, in section \ref{sec:torsor} we define
and investigate a conductor for torsors under finite flat $k$-groups, which we believe to be new.
The general pattern of these computations is always the same: First we show that the collection $c=\{c_L\}$ defined 
in the various situations defines a semi-continuous conductor (of a certain level) in the sense of 
Definitions \ref{defn:cond} and \ref{defn:mot-cond}, then we do a symbol computation to show that this conductor
is actually motivic. Note however, that the actual computations in the various cases differ quite a bit. 

\begin{ack}
We thank the referee for helpful remarks.
\end{ack}

\begin{conventions}
We work over a perfect field $k$. If $K/k$ is a field extension, then by a $K$-scheme
we will always mean a scheme which is separated and of finite type over $K$. In contrast, the phrase 
{\em scheme over $K$} refers to any scheme morphism $X\to \Spec K$.
By a {\em smooth} $K$-scheme we mean a $K$-scheme which is
smooth over $K$. We denote by $\Sm_K$ the category  of such schemes and set $\Sm=\Sm_k$.
For $k$-schemes $X$ and $Y$ we write $X\times Y$ instead of $X\times_k Y$.
For any scheme $X$ we denote by $X^{(i)}$ the set of $i$-codimensional points of $X$.
\end{conventions}

\part{The general theory}
\section{Presheaves with transfers on pro-smooth schemes}\label{sec:proPST}
The material in this section is well-known, we give some details for lack of reference. 
\begin{para}\label{para:Cor}
Denote by $\Cor$ the category of finite correspondences of Suslin-Voevodsky. Recall that the objects are
 the smooth $k$-schemes and morphisms are given by  correspondences, i.e.,
$\Cor(X,Y)$ is the free abelian group generated by  prime correspondences, i.e.,
integral closed subschemes $V\subset X\times Y$
which are finite and surjective over a connected component of $X$.
Given two prime correspondences $V\in \Cor(X,Y)$ and $W\in \Cor(Y,Z)$ their composition is given by
the intersection product (see e.g. \cite[V, C]{SerreAL})
\eq{para:Cor1}{W\circ V= p_{13*} (p_{12}^*V \cdot p_{23}^* W),}
where $p_{ij}$ denotes the projection from $X\times Y\times Z$ to the factor $(i,j)$.

Denote by $\ProCor$ the pro-category of $\Cor$, i.e., objects 
are functors $I^{\rm o}\to \Sm$, $i\mapsto X_i$, where $I$ is a filtered essentially small category, and
the morphisms between two pro-objects $(X_i)_{i\in I}$ and $(Y_j)_{j\in J}$ are given by
\[\ProCor((X_i), (Y_j))= \varprojlim_{j\in J}\varinjlim_{i\in I} \Cor(X_i, Y_j).\]
\end{para}

\begin{defn}\label{defn:Corpro}
We define the category $\Cor^{\rm pro}$ as follows:
The objects are the noetherian regular schemes over $k$ of the form
\eq{defn:Corpro1}{X=\varprojlim_{i\in I} X_i,}
where $(X_i)_{i\in I}$ is a projective system of smooth $k$-schemes indexed by a partially ordered set
and with affine transition maps $X_i\to X_j$, $i\ge j$.
If $X$ and $Y$ are two objects in $\Cor^{\rm pro}$, then 
$\Cor^{\rm pro}(X,Y)=\Cor(X,Y)$ is the free abelian group generated by prime correspondences
in the sense of \ref{para:Cor}.
The composition is defined in the same way as in
the case of $\Cor$. (Note that this still makes sense by \cite[V, B, 3., Thm 1]{SerreAL}.)
\end{defn}

\begin{remarks}\label{rmk:Corpro}
\begin{enumerate}
\item All objects in $\Cor^{\rm pro}$ are separated, noetherian, and regular schemes over $k$. 
     Any affine, noetherian, and   regular scheme over $k$ defines an object   in $\Cor^{\rm pro}$,
by \cite[(1.8) Thm]{Popescu86} and \cite[Exp I, Prop 8.1.6]{SGA4I}.
\item Note, that for $X, Y\in \Cor^{\rm pro}$  the cartesian product $X\times Y$ does not need to be noetherian;
but if $Y\in \Sm$ and $X\in \Cor^{\rm pro}$,  then $X\times Y\in \Cor^{\rm pro}$.
\end{enumerate}
\end{remarks}

\begin{lemma}\label{lem:regular-proj}
Let $A$ be a $k$-algebra which is noetherian, regular, and is a directed limit $A=\varinjlim_{i\in I} A_i$,
where the $A_i$ are smooth and of finite type over $k$ and the transition maps  $A_i\to A_j$, $j\ge i$ are flat.
Let $X$ be a regular quasi-projective $A$-scheme. Then $X\in \Cor^{\rm pro}$.
\end{lemma}
\begin{proof}
Set $S_i=\Spec A_i$ and $S=\Spec A=\varprojlim_i S_i$.
Choose an $S$-embedding $X\subset \P^n_S$. We find an $i_0$ and a  subscheme $X_{i_0}\subset \P^n_{S_{i_0}}$
such that $X=X_{i_0}\times_{S_{i_0}} S$. Set $X_i:=X_{i_0}\times_{S_{i_0}} S_i$,  for $i\ge i_0$.
Then the transition maps $X_j\to X_i$, $j\ge i\ge i_0$, are affine and flat, hence so is the projection 
$\tau_i: X=\varprojlim_i X_i \to X_{i_0}$.
Since $X$ is regular, there exists an open neighborhood $U_{i_0}\subset X_{i_0}$ containing $\tau_{i_0}(X)$ which is regular
(see \cite[Cor (6.5.2)]{EGAIV2}). Since $U_{i_0}$ is of finite type over the perfect field $k$, it is even smooth.
Set $U_i=U_{i_0}\times_{S_{i_0}} S_i$. Then the transition maps $U_j\to U_i$, $j\ge i \ge i_0$, are affine and flat,
 each $U_i$ is smooth, and we have $X=\varprojlim_i U_i$; hence $X\in \Cor^{\rm pro}$.
\end{proof}

\begin{lemma}\label{lem:Corpro-proCor}
There is a (up to isomorphism) canonical  and faithful functor 
\[\Cor^{\rm pro}\to \ProCor, \quad \varprojlim_i X_i\mapsto (X_i).\]
\end{lemma}
\begin{proof}
For any $X\in \Cor^{\rm pro}$ we choose once and for all a projective system $(X_i)_{i\in I}$ as in \eqref{defn:Corpro1}.
In particular, $(X_i)\in \ProCor$.
Note, if $X=\varprojlim_{j\in J} X'_j$, then $(X_i)\cong (X_j')$ in $\ProSm$.
Take $X=\varprojlim_{i\in I} X_i$ and $Y=\varprojlim_{j\in J} Y_j$ in $\Cor^{\rm pro}$ 
and let $V\subset X\times Y$ be a prime correspondence. 
For any scheme $S$ over $k$ we denote by 
\[\rho_i: X\times S\to X_i\times S, \quad \rho_{i', i}: X_{i'}\times S\to X_i\times S, \quad  i'\ge i,\]
and by 
\[\sigma_j: Y\times S\to Y_j\times S,\quad \sigma_{j', j}: Y_{j'}\times S\to Y_j\times S,\quad j'\ge j,\]
the projection  and transition maps of $(X_i\times S)$ and $(Y_j\times S)$, respectively. 
By assumption all these maps are affine.
For all $j$, the morphism  $V\to X\times Y_j$ induced by $\sigma_j$ is a morphism of finite type $X$-schemes.
Since $V$ is finite over $X$, its image $\sigma_j(V)\subset X \times Y_j$ is proper over $X$.
Hence $V\to \sigma_j(V)$ is proper and affine, hence finite.
Since $X$ is noetherian $\sigma_j(V)$ is finite over $X$, hence 
we obtain a well defined correspondence $\sigma_{j*}V\in \Cor(X, Y_j)$  with the property
\eq{lem:Corpro-proCor0}{ \sigma_{j*} V=0 \Longleftrightarrow V=0.}
Furthermore, since $X\times Y_j$ is noetherian,
we find an index $i$ (depending on $j$) and a correspondence $V_{i,j}\in \Cor(X_i, Y_j)$
such that 
\[\sigma_{j*}V=\rho_i^*V_{i,j}.\]
If we find $i'$ and $V_{i',j}'$ with $\rho_{i'}^*V_{i', j}'=\sigma_{j*}V$, then clearly
$V_{i,j}=V_{i',j}'$ in $\varinjlim_i\Cor(X_i, Y_j)$. Therefore we obtain a well-defined element $V_j$
\[\Cor(X_i, Y_j)\to \varinjlim_{i}\Cor(X_i, Y_j), \quad V_{i,j}\mapsto V_j.\]
By the base change formula (see \eqref{lemCorpro-proCor-bc} below) we obtain $\sigma_{j', j*}V_{j'}= V_{j}$.
We obtain a  morphism 
\eq{lem:Corpro-proCor1}{\Cor^{\rm pro}(X,Y)\to \ProCor((X_i), (Y_j)), \quad V\mapsto (V_{j})_j.}
It is injective by \eqref{lem:Corpro-proCor0}.
Finally we have to check that \eqref{lem:Corpro-proCor1}
is compatible with composition. Take $Z=\varprojlim_{l\in L} Z_l\in \Cor^{\rm pro}$. 
For any scheme $S$ over $k$ denote by 
\[\tau_l : Z\times S\to Z_l\times S\]
the projection map. Take prime correspondences
$V\in \Cor^{\rm pro}(X,Y)$ and $W\in \Cor^{\rm pro}(Y,Z)$. 
For any $l\in L$ we find an index $j(l)\in J$  and a correspondence $W_{j(l), l}\in \Cor(Y_{j(l)}, Z_l)$
such that $\tau_{l*}W= \sigma_{j(l)}^*W_{j(l), l}$.
For any $j(l)$ we find an index $i(j(l))\in I$ and a correspondence $V_{i(j(l)), j(l)}\in \Cor(X_{i(j(l)), Y_{j(l)}})$ 
such that $\sigma_{j(l)*}V= \rho^*_{i(j(l))}V_{i(j(l)), j(l)}$.
Then the compatibility of \eqref{lem:Corpro-proCor1} will hold if we can show 
\eq{lem:Corpro-proCor2}{\tau_{l*}(W\circ V)= 
\rho^*_{i(j(l))} (W_{j(l),l}\circ V_{i(j(l)), j(l)}), \quad \text{for all } l\in L.}
To this end we recall some well-known formulas. 
Assume we are given the following diagram of schemes over $k$ which are in $\Cor^{\rm pro}$, 
\[\xymatrix{ X'\ar[r]^{f'}\ar[d]_{h'} & Y'\ar[d]^h & \\
                  X\ar[r]^f           & Y\ar[r]^g  & Z,
}\]
and assume the square is cartesian and tor-independent.
Then for cycles $\alpha, \beta, \beta', \gamma$ on $X$, $Y$, $Z$, respectively,
the following relations hold as soon as both sides of the equation are defined (see \cite[V, C]{SerreAL}):
\begin{align}
 f^*g^*\gamma & = (g\circ f)^*\gamma.\\
 g_*f_*\alpha & = (g\circ f)_*\alpha.\\
\label{lemCorpro-proCor-bc} h^*f_*\alpha &= f'_*{h'}^* \alpha.\\
f^*(\beta\cdot \beta') & =f^*(\beta)\cdot f^*(\beta').\\
f_{*}(\alpha\cdot f^*(\beta)) & = f_*(\alpha)\cdot \beta.
\end{align}
Using these formulas it is straightforward but a bit longish to check, that \eqref{lem:Corpro-proCor2} holds.
Indeed, since all cycles involved are always finite over some scheme over $k$ it will be clear that the formulas in
question are defined; the base change formula \eqref{lemCorpro-proCor-bc} will be only applied 
in cases where one of the maps $f$ or $h$ is flat, hence the tor-independence condition will be automatic.
(But note that $h$ might not be flat so there might appear higher Tor's in the computation of $h^*$ and ${h'}^*$,
respectively.) This finishes the proof.
\end{proof}

\begin{para}\label{para:PST}
A presheaf with transfers in the sense of Suslin-Voevodsky is a functor $F: \Cor^{\rm o}\to \Ab$;
they form the category $\PST$.
We extend it to a functor $F: \ProCor^{\rm o}\to \Ab$ by the formula
\[F((X_i)_{i\in I}):=\varinjlim_{i} F(X_i).\]
Precomposing $F$ with the functor from Lemma \ref{lem:Corpro-proCor}
we obtain presheaves on $\Cor^{\rm pro}$, which we again denote by $F$,
\[F: (\Cor^{\rm pro})^{\rm o}\to \Ab.\]
For $\alpha\in \Cor^{\rm pro}(X,Y)$ we denote by $\alpha^*=F(\alpha): F(Y)\to F(X)$, the induced map.
If $f: X\to Y$ is a morphism  with graph $\Gamma_f\subset X\times Y$ 
between $k$-schemes which are objects in $\Cor^{\rm pro}$, 
then we set 
\eq{para:PST1}{f^*:=\Gamma_f^*: F(Y)\to F(X);}
if $f$ is a finite morphism and $\Gamma_f^t\subset Y\times X$ is the transposed of the graph of $f$ we set 
\eq{para:PST2}{f_*:=(\Gamma_f^t)^*: F(X)\to F(Y).}
\end{para}

\section{Review of reciprocity sheaves}\label{sec:modulus}
In this section we collect some definitions, notations and results from \cite{KMSY-MotModI}, \cite{KMSY-MotModII}, 
\cite{KSY-RecII}, and \cite{Saito-Purity-Rec}. 
\begin{para}\label{para:modulus}
A {\em modulus pair} $\sX=(\ol{X}, X_\infty)$ consists of a separated and finite type $k$-scheme $\ol{X}$
and an effective Cartier divisor $X_\infty\ge 0$ such that the open complement $X:=\ol{X}\setminus |X_\infty|$ is smooth.
We say $\sX$ is a {\em proper modulus pair} if $\ol{X}$ is proper over $k$.
A basic example is the cube 
\[\bcube:=(\P^1_k,\infty).\]

Let $\sX=(\ol{X},X_\infty)$ and $\sY=(\ol{Y}, Y_\infty)$ be two modulus pairs with corresponding opens 
$X=\ol{X}\setminus |X_\infty|$ and  $Y=\ol{Y}\setminus |Y_\infty|$, respectively.
The modulus pair $\sX\otimes \sY$ is defined by 
\eq{para:modulus1}{\sX\otimes \sY:=(\ol{X}\times\ol{Y}, X_\infty\times \ol{Y}+ \ol{X}\times Y_\infty).}
An {\em admissible prime correspondence from $\sX$ to $\sY$} is a  prime correspondence $V\in \Cor(X,Y)$
satisfying the following condition
\eq{para:modulus2}{{X_{\infty}}_{|\ol{V}^N}\ge {Y_\infty}_{|\ol{V}^N},}
where $\ol{V}^N\to \ol{V}\subset \ol{X}\times\ol{Y}$ is the normalization of the closure of $V$.
We denote by $\Cor_{\rm adm}(\sX,\sY)\subset \Cor(X,Y)$
the subgroup generated by admissible correspondences.
Assume $\sX$ is a {\em proper} modulus pair. Recall from \cite[Lem 2.2.2]{KSY-RecII}, that 
the presheaf with transfers $h_0(\sX)\in \PST$ is defined by
\[h_0(\sX)(S)= \Coker\left(\Cor_{\rm adm}(\bcube\otimes S, \sX)\xr{i_0^*-i_1^*} \Cor(S, X)\right),\]
where we write $S$ instead of $(S,\emptyset)$ and $i_\epsilon: 
S\inj \A^1_S$ is the $\epsilon$-section, $\epsilon\in \{0,1\}$.
We have a natural quotient map $\Ztr(X)\to h_0(\sX)$,
where $\Ztr(X)$ is the presheaf with transfers representing $X$, i.e., $\Ztr(X)(S)=\Cor(S,X)$.
\end{para}

\begin{defn}[{\cite[Def 2.2.4]{KSY-RecII}}]\label{defn:rec}
Let $F\in \PST$,   $X\in\Sm$  and $a\in F(X)$. 
We say $a$ {\em  has SC-modulus (or just modulus) $\sX$}, if $\sX=(\ol{X}, X_\infty)$ is a proper modulus pair 
with $X=\ol{X}\setminus |X_\infty|$ and the Yoneda map $a: \Ztr(X)\to F$, factors via
\[\xymatrix{ \Ztr(X)\ar[rr]^a\ar[dr] && F,\\
                    & h_0(\sX)\ar[ur]_{\exists}
}\]
i.e., for any $S\in \Sm$ and any  correspondence 
$\gamma\in \Cor_{\rm adm}(\bcube\times S, \sX)\subset \Cor(\A^1\times S, X)$
we have $i_0^*\gamma^*a= i_1^*\gamma^* a$.

We say $F$ has {\em SC-reciprocity}, if for all $X\in \Sm$  any $a\in F(X)$ has a modulus.
We denote by $\RSC\subset \PST$ the full subcategory consisting of presheaves with transfers which have SC-reciprocity.
Further we set 
\[\RSCNis= \RSC\cap \NST,\]
where $\NST\subset \PST$ is the full subcategory of Nisnevich sheaves with transfers.
\end{defn}

\begin{para}\label{para:modulusII}
It is shown in \cite{KSY-RecII} that the presheaves in $\RSC$ are in fact induced by presheaves on modulus pairs
in the following way:
Let $\sX=(\ol{X}, X_\infty)$ and $\sY=(\ol{Y}, Y_\infty)$ be modulus pairs with corresponding opens $X$ and $Y$, 
respectively. An admissible correspondence from $\sX$ to $\sY$  (see \ref{para:modulus1}) is called {\em left proper}, 
if the closure in $\ol{X}\times \ol{Y}$ of all its irreducible components is proper over $\ol{X}$.
We denote by $\ulMCor(\sX, \sY)\subset \Cor(X,Y)$ the subgroup of all left proper admissible correspondences.
This subgroup is stable under composition of correspondences (see \cite[Prop 1.2.3]{KMSY-MotModI}).
Hence we can define the category 
$\ulMCor$ with objects the modulus pairs and morphisms given by admissible left proper correspondences.
We denote by $\MCor$ the full subcategory with objects the proper modulus pairs.
We denote by $\ulMPST$  the category of presheaves on $\ulMCor$ and by $\MPST$ the category of
presheaves on $\MCor$. 
By \cite[Prop 2.2.1, Prop 2.3.1, Prop 2.4.1]{KMSY-MotModI} there are three pairs of adjoint functors 
$(\omega_!, \omega^*)$, $(\ul{\omega}_!, \ul{\omega}^*)$ and $(\tau_!, \tau^*)$: 
\[\xymatrix{          \PST\ar@<-4pt>[r]_{\ul{\omega}^*} &
                       \ulMPST\ar@<-4pt>[r]_-{\tau^*}\ar@<-4pt>[l]_{\ul{\omega}_!}  &
                          \MPST\ar@<4pt>[r]^-{\omega_!}\ar@<-4pt>[l]_-{\tau_!} &
                                                              \PST,\ar@<4pt>[l]^-{\omega^*}
}\]
which are given by 
\begin{align}
\label{para:modulusII0} \ul{\omega}^*F(\ol{X}, X_\infty)&= F(\ol{X}\setminus |X_\infty|),&   
                                                                                \ul{\omega}_!H(X)&= H(X,\emptyset), \\
\label{para:modulusII1} \omega^*F(\ol{X}, X_\infty)&= F(\ol{X}\setminus |X_\infty|),  &
                                                                  \omega_!G(X)&\cong \varinjlim_{\sX\in \MSm(X)} G(\sX),\\
\label{para:modulusII2} \tau^*F(\sX)&=F(\sX),&\quad \tau_! G(\sU)&\cong \varinjlim_{\sX\in \Comp(\sU)} G(\sX),
\end{align}
where $\MSm(X)$ is the subcategory of $\MCor$ with objects the proper modulus pairs with corresponding opens $X$ 
and only those morphism which map to the identity in $\Cor(X,X)$,
and  $\Comp(\sU)$ is the category of compactifications of $\sU=(\ol{U}, U_\infty)$, i.e., 
objects are proper modulus pairs $\sX=(\ol{X}, \ol{U}_\infty+\Sigma)$, where $\ol{U}_\infty$ and $\Sigma$ are 
effective Cartier divisors such that $\ol{X}\setminus |\Sigma|=\ol{U}$ and  $\ol{U}_{\infty|\ol{U}}=U_\infty$,
and the morphisms are those
which map to the identity in $\ulMCor(\sU,\sU)$, see \cite[Lem 2.4.2]{KMSY-MotModI}.
The functors $\omega_!$, $\ul{\omega}_!$, $\tau_!$ are exact and we have $\omega_!= \ul{\omega}_!\tau_!$.

We denote by $\CI$ the full subcategory of $\MPST$ of {\em cube invariant} objects, i.e., those
         $F\in \MPST$, which satisfy that for any proper modulus pair $\sX$
          the pullback along $\sX\otimes \bcube\to\sX$ induces an isomorphism
           \[F(\sX)\cong F(\sX\otimes \bcube).\]
By \cite[Prop 2.3.7]{KSY-RecII} we have $\omega_!(\CI)=\RSC$
and there is a fully faithful left exact functor $\omega^{\CI}: \RSC\to \CI$
given by 
\eq{para:modulusII3}{\omega^{\CI}(F)(\ol{X}, X_\infty)= \{a\in F(\ol{X}\setminus X_\infty)\mid
 a \text{ has moduls } (\ol{X}, X_\infty) \}.}
We have 
\eq{para:modulusII4}{\ul{\omega}_!\tau_!\omega^{\CI}(F)\cong \omega_!\omega^{\CI}(F)\cong F.}
\end{para}

\begin{para}\label{para:modulusIII}
We recall some more definitions and results from \cite{KMSY-MotModI}, \cite{KMSY-MotModII}, and
\cite{Saito-Purity-Rec} related to Nisnevich sheaves.

For $F\in \ulMPST$ and $\sX=(\ol{X}, X_\infty)\in \ulMCor$ we denote by $F_{\sX}$ the presheaf on $\ol{X}_{\et}$ 
defined by 
\eq{para:modulusIII0}{(U\xr{u} \ol{X})\mapsto F_\sX(U):=F(U, u^* X_\infty).}
We denote by $\ulMNST$ the full subcategory of $\ulMPST$ consisting of those $F$ such that 
$F_\sX$ is a Nisnevich sheaf on $\ol{X}$, for any $\sX=(\ol{X}, X_\infty)\in \ulMCor$.
Further, $\MNST$ is the full subcategory of $\MPST$ consisting
of $F$ such that $\tau_!F\in \ulMNST$. 
By \cite[Thm 4.5.5]{KMSY-MotModI} and \cite[Thm 4.2.4]{KMSY-MotModII} there are exact sheafification functors 
(i.e., left adjoints to the natural inclusions)
\[\ul{a}_{\Nis}: \ulMPST\to \ulMNST, \quad a_{\Nis}: \MPST\to \MNST,
\]
such that
\begin{enumerate}
\item\label{para:modulusIII1}
$(\ul{a}_{\Nis}F)(\sX)= 
                \varinjlim_{f: Y\xr{\sim} \ol{X}} F_{\sX, \Nis}(Y, f^*X_\infty)$,
where $\sX=(\ol{X}, X_\infty)\in\ulMCor$, $F_{\sX, \Nis}$ denotes the Nisnevich sheafification 
of the presheaf $F_{\sX}$ on $\ol{X}_\et$, and 
the limit is over all proper birational morphisms $f: Y\to \ol{X}$ which restrict to an isomorphism 
$Y\setminus |f^*X_\infty|\xr{\simeq} \ol{X}\setminus |X_\infty|$;
\item\label{para:modulusIII2} $\tau_!$ restricts to an exact functor $\tau_{!} : \MNST\to \ulMNST$ and satisfies
          \eq{para:modulusIII2.1}{\ul{a}_{\Nis}\tau_! F= \tau_! a_{\Nis} F, \quad \text{for all } F\in \MPST.}
\end{enumerate}
It follows that $a_{\Nis}= \tau^* \ul{a}_{\Nis}\tau_!$ and  
\[{\ul{a}_{\Nis}}_{\restriction \ulMNST}=\id_{\ulMNST}, \quad {a_{\Nis}}_{\restriction \MNST}=\id_{\MNST}.\]
By \cite[Prop 6.2.1]{KMSY-MotModII}, 
\[ \ul{a}_{\Nis}\ul{\omega}^*= \ul{\omega}^* a_{\Nis}^V, \quad a_{\Nis}\omega^*=\omega^* a_{\Nis}^V,\]
 where $a_{\Nis}^V: \PST\to \NST$ is Voevodsky's Nisnevich sheafification functor (see \cite[Lem 3.1.6]{VoDM}),
and we obtain induced functors
\[\ul{\omega}^*: \NST\to \ulMNST, \quad \omega^* : \NST\to \MNST.\]
\end{para}

\begin{lemma}\label{lem:CINis}
For $F\in \RSC_\Nis$ we have $\omega^{\CI} F\subset a_{\Nis} \omega^{\CI} F\subset \omega^* F$ in $\MPST$
(see Definition \ref{defn:rec} and \eqref{para:modulusII3} for notation). 
Here the first inclusion is given by the unit of adjunction.
\end{lemma}
\begin{proof}
By definition $\omega^{\CI}F\subset \omega^* F$. We obtain the following commutative diagram
\[\xymatrix{ a_{\Nis} \omega^{\CI} F\ar@{^(->}[r] & a_{\Nis}\omega^*F\\
                        \omega^{\CI}F\ar[u]\ar@{^(->}[r] & \omega^* F,\ar@{=}[u] 
}\]
in which the vertical maps are induced by adjunction. The vertical map on the right is an isomorphism
since $\omega^*F\in \MNST$, the top horizontal map is an inclusion since $a_{\Nis}$ is exact.
This gives the statement.
\end{proof}
\begin{remark}\label{rem:CINis}
It follows from Corollary \ref{cor:omegaCI} below that the first inclusion in Lemma \ref{lem:CINis} is actually an equality.
\end{remark}

\begin{para}\label{para:ulMCorpro}
We define the category $\ulMCor^{\rm pro}$ as follows:
The objects are pairs $\sX=(\ol{X}, X_\infty)$, where 
\begin{enumerate}
\item $\ol{X}$ is a separated noetherian scheme over $k$ of the form
$\ol{X}=\varprojlim_{i\in I} \ol{X}_i$, with $(\ol{X}_i)_{i\in I}$  a projective system of 
separated  finite type $k$-schemes
indexed by a partially ordered set with affine transition maps $\tau_{i,j}: \ol{X}_i\to \ol{X}_j$, $i\ge j$,
\item $X_\infty=\varprojlim_{i\in I} X_{i,\infty}$, with $X_{i,\infty}$ an effective Cartier divisor on $\ol{X}_i$,
such that $\ol{X}_i\setminus |X_{i,\infty}|$ is smooth, for all $i$, and 
$\tau_{i,j}^*X_{j,\infty}= X_{i,\infty}$, $i\ge j$,
\item $\ol{X}\setminus |X_\infty|$ is regular.
\end{enumerate}
The morphisms are given by the admissible left proper correspondences which are verbatim defined as in 
\ref{para:modulusII}. That the composition of correspondences in $\Cor^{\rm pro}$ induces a well-defined composition
in $\ulMCor^{\rm pro}$ is shown in the same way as in \cite[Prop 1.2.3]{KMSY-MotModI}.
\end{para}
\begin{lemma}\label{lem:MCorpro-proMCor}
There is a (up to isomorphism) canonical  and faithful functor 
\[\ulMCor^{\rm pro}\to \ProulMCor, \quad \varprojlim_i (\ol{X}_i, X_{i,\infty})\mapsto (\ol{X}_i, X_{i,\infty})_i.\]
\end{lemma}
\begin{proof}
Let $\sX=(\ol{X}, X_\infty)$, 
$\sY=(\ol{Y}, Y_\infty)\in \ProulMCor$.
We write $\sX=\varprojlim_{i\in I}\sX_i$ with $\sX_i=(\ol{X}_i, X_{i,\infty})$, and similarly
$\sY=\varprojlim_{j\in J} \sY_j$. Set $X=\ol{X}\setminus |X_\infty|$, etc.
We have to show that the injection \eqref{lem:Corpro-proCor1} restricts to  
\eq{lem:MCorpro-proMCor1}{\ulMCor^{\rm pro}(\sX, \sY)\to \ProulMCor((\sX_i), (\sY_j)).}
To this end let $V\in \ulMCor^{\rm pro}(\sX, \sY)$ be a left proper admissible correspondence.
For $j\in J$ denote by $\sigma_j(V)$ the image  of $V$ under the projection $X\times Y\to X\times Y_j$.
Then $\sigma_j(V)$ is a finite prime correspondence as was observed in the proof of Lemma \ref{lem:Corpro-proCor}.
Let $\ol{V}\subset \ol{X}\times \ol{Y}$ be the closure of $V$.
By assumption $\ol{V}$ is proper over $X$. Since $\ol{X}\times Y_j$ is separated and of finite type over $\ol{X}$
the image of $\ol{V}$ in $\ol{X}\times Y_j$ is closed and proper over $X$; 
hence it is equal to  the closure $\ol{\sigma_j(V)}$ of $\sigma_j(V)$.
Now \cite[Lem 1.2.1]{KMSY-MotModI} yields 
\eq{lem:MCorpro-proMCor2}{{X_\infty}_{|\ol{\sigma_j(V)}^N}\ge {Y_{j,\infty}}_{|\ol{\sigma_j(V)}^N},}
with the notation from \eqref{para:modulus2}.
As in the proof of Lemma \ref{lem:Corpro-proCor} we find an index $i_0\in I$ and a finite correspondence
$V_{i_0,j}\subset X_{i_0}\times Y_j$ which pulls back to $\sigma_j(V)$. 
We can  also assume (after possibly enlarging $i_0$)
that the closure $\ol{V}_{i_0,j}\subset \ol{X}_{i_0}\times \ol{Y}_j$ of $V_{i_0,j}$ pulls back to $\ol{\sigma_j(V)}$.
We obtain the cartesian diagram
\[\xymatrix{
\ol{\sigma(V_j)}\ar[r]\ar[d] &\ol{X}=\varprojlim \ol{X}_i\ar[d]\\
\ol{V}_{i_0,j}\ar[r] &\ol{X}_{i_0}.
}\]
Since the upper horizontal arrow is proper, the lower horizontal arrow becomes proper after possibly enlarging $i_0$,
see \cite[Thm (8.10.5), (xii)]{EGAIV3}. Hence by our construction and \eqref{lem:MCorpro-proMCor2}, 
the scheme $V_{i_0, j}=\ol{V}_{i_0,j}\cap (X_{i_0}\times Y_j)$
is a left proper admissible  correspondence from $\sX_{i_0}$ to $\sY_j$ and gives a well-defined element 
\[V_j\in\varinjlim_{i\in I}\ProulMCor(\sX_i, \sY_j).\]
This shows that \eqref{lem:Corpro-proCor1} restricts to \eqref{lem:MCorpro-proMCor1}.
\end{proof}
\begin{para}\label{para:MPSTpro}
Let $F\in \ulMPST$. Using Lemma \ref{lem:MCorpro-proMCor} we can extend $F$ to a presheaf on $\ulMCor^{\rm pro}$
by the formula
\[F(\sX)= \varinjlim_i F(\sX_i), \quad \sX=\varprojlim_i \sX_i\in \ulMCor^{\rm pro}.\]
\end{para}

\section{Conductors for presheaves with transfers}\label{sec:cond}

\begin{defn}\label{defn-hdvf}
\begin{enumerate}
\item We say that $L$ is a {\em henselian discrete valuation field  of geometric type (over $k$)} 
(or short that $L$ is a {\em henselian dvf}) if $L$ is a discrete valuation field and its ring of integers is equal to 
the henselization of the local ring of a smooth $k$-scheme $U$ in a 1-codimensional point $x\in U^{(1)}$, i.e., 
$\sO_L= \sO_{U,x}^h$.  For $n\in \N_{\ge 1}\cup \{\infty\}$ we set
\[\Phi=\{L \text{ henselian dvf}\}, \quad \Phi_{\le n}=\{L\in \Phi\mid \td(L/k)\le n\}.\]
Note that in positive characteristic $\Phi_{\le 1}$ consists precisely of the henselian dvf's with perfect residue field.
\item Let $X$ be a smooth $k$-scheme. A {\em henselian dvf point} of $X$ 
     is  a $k$-morphism $\Spec L \to X$, with $L\in \Phi$.
\item  Let $\sX=(\ol{X}, X_\infty)$ be a modulus pair with $X=\ol{X}\setminus |X_\infty|$.
         A henselian dvf point of $\sX$ is a henselian dvf point $\rho: \Spec L\to X$
extending to $\Spec \sO_L\to \ol{X}$. Note, if it exits, such an extension is unique, and 
if $\sX$ is proper, then there always exists an extension. We will denote this extension also by $\rho$.
    We will also write $\rho: \Spec L\to \sX$ for the henselian dvf point of $\sX$ defined by $\rho$.
\end{enumerate}
\end{defn}

\begin{nota}\label{nota:hdvf}
\begin{enumerate}
\item Let $F\in\PST$ and  $X\in \Sm$. 
A henselian dvf point $\rho:\eta=\Spec L\to X$ is a morphism in $\Cor^{\rm pro}$ (see \ref{defn:Corpro}).
Hence we get a morphism (see \ref{para:PST})
\[\rho^*:F(X)\to F(\eta)=:F(L), \quad a\mapsto \rho^*a .\]
Also $\eta=\Spec L\to \Spec \sO_L=\ol{\eta}$ is in $\Cor^{\rm pro}$ and we get an induced map
$F(\sO_L):=F(\ol{\eta})\to F(L)$.
\item Let $\sX=(\ol{X}, X_\infty)$ be a modulus pair with $X=\ol{X}\setminus |X_\infty|$ and 
$\rho: \Spec L\to \sX$ a henselian dvf point. 
Then we denote by
\[v_L(X_\infty)=v(\rho^*X_\infty)\in \N_0\]
the multiplicity of ${X_\infty}$ pulled back along $\ol{\rho}$.  
\end{enumerate}
\end{nota}

\begin{defn}\label{defn:cond}
Let $F\in \PST$ and $n\in [1,\infty]$.
A {\em conductor of level $n$ for $F$} is a collection of set maps 
\[c=\{c_{L}: F(L)\to \N_0\mid L\in\Phin \}\]
satisfying the following properties for all $L\in\Phin$ and all $X\in \Sm$:
\begin{enumerate}[label = {(c\arabic*)}]
\item\label{c1} $c_{L}(a)=0 \,\Rightarrow\, a\in \Im(F(\sO_L)\to F(L))$.
\item\label{c2} $c_L(a+b)\le \max\{c_L(a), c_L(b)\}$.
\item\label{c3}
       $c_L(f_*a)\le \lceil\frac{c_{L'}(a)}{e(L'/L)}\rceil$, for any finite morphism $f:\Spec L'\to \Spec L$ and any $a\in F(L')$.
        Here $e(L'/L)$ denotes the ramification index of $L'/L$ and $\lceil -\rceil$ is the round up.
\item\label{c4}
      Assume $a\in F(\A^1_X)$ satisfies
      $c_{k(x)(t)_\infty}(\rho_x^*a)\le 1$, for all $x\in X$ with $\td(k(x)/k)\le n-1$, 
       where $k(x)(t)_\infty:=\Frac(\sO_{\P^1_x,\infty}^h)$ and  
$\rho_x :\Spec k(x)(t)_\infty \to \A^1_X$ is the natural map.
    Then $a\in \pi^*F(X)$, with $\pi: \A^1_X\to X$ the projection.
\item\label{c5}
        For any $a\in F(X)$ there exists a proper modulus pair $\sX=(\ol{X}, X_\infty)$
       with $X=\ol{X}\setminus |X_\infty|$, such that for all  $\rho:\Spec L \to \sX$ 
         we have  
            \[c_L(\rho^*a)\le v_L(X_\infty).\]
\end{enumerate}
A  conductor of level $\infty$ will be simply called {\em conductor}.
\end{defn}

\begin{remarks}\label{rmk:cond}
\begin{enumerate}[label=(\arabic*)]
\item\label{rmk:cond1} If $F$ is homotopy invariant, then setting $c_L(a)=0$, if $a\in \Im(F(\sO_L)\to F(L))$, 
                  and $c_L(a)=1$ else, defines a conductor (of any level).
\item If $c=\{c_L\}$ is a conductor for $F$. Then for any $L$ we have
\eq{rmk:cond2}{a\in \Im(F(\sO_L)\to F(L)) \Longleftrightarrow c_L(a)=0.} 
  Indeed, if $a\in \Im(F(\sO_L)\to F(L))$, then 
  we find a smooth $k$-scheme $U$, a 1-codimensional point $x\in U^{(1)}$,  a $k$-morphism
$\Spec \sO_L\to \Spec \sO_{U,x}\to U$ and an element $\tilde{a}\in F(U)$ such that 
 $\rho^*\tilde{a}=a\in F(L)$, where $\rho: \Spec L\to \Spec\sO_L\to U$.
The vanishing of $c_L(a)$ hence follows directly from \ref{defn:cond}\ref{c5}.
\item\label{rmk:cond3} Let $c=\{c_L\}$ be a conductor. Then
 $c^{\le n}:=\{c_L\mid \td(L/k)\le n\}$ is a conductor if and only if $c^{\le n}$ satisfies \ref{c4}.
\end{enumerate}
\end{remarks}

\begin{defn}\label{defn:condmod}
Let $F\in \PST$ and let $c=\{c_L\}$ be a conductor of level $n$ for $F$.
Let  $\sX=(\ol{X}, X_\infty)$ be  a modulus pair with $X=\ol{X}\setminus |X_\infty|$.
For  $a\in F(X)$, we write
\[c_{\ol{X}}(a)\le X_\infty\]
to mean  $c_L(\rho^*a)\le v_L(X_\infty)$, for all henselian dvf points $\rho : \Spec L\to \sX$ 
with $\td(L/k)\le n$ (see Definition \ref{defn-hdvf}).
\end{defn}

\begin{lemma}\label{lem:indepmod}
Let $c$ be a conductor of some level for $F\in \PST$, $X\in \Sm$, and $a\in F(X)$.
Let $\sX=(\ol{X},X_\infty)$ be {\em any} proper modulus pair with $X=\ol{X}\setminus X_\infty$.
Then there exists a natural number $n\ge 1$ such that $c_{\ol{X}}(a)\le n\cdot X_\infty$.
\end{lemma}
\begin{proof}
By \ref{defn:cond}\ref{c5}, there exists a proper modulus pair $\sX_1=(\ol{X}_1, X_{1,\infty})$ with
corresponding open $X$ and such that $c_L(\rho^*a)\le v_L(X_{1,\infty})$, for all $\rho$.
We find a proper normal $k$-scheme  $\ol{X}_2$ with $k$-morphisms $f:\ol{X_2}\to \ol{X}$, $f_1:\ol{X_2}\to \ol{X_1}$
such that $\ol{X_2}\setminus |f^*X_\infty|=X= \ol{X}_2\setminus |f_1^*X_{1,\infty}|$. 
Take $n\ge 1$ with $f_1^*X_{1,\infty}\le n\cdot f^*X_\infty$.
Then for $\rho: \Spec L\to X$
\[c_L(\rho^* a)\le v_L(X_{1,\infty}) =v_L(f_1^* X_{1,\infty})\le v_L(n\cdot f^*X_\infty)= v_L(n\cdot X_\infty).\]
Hence the statement.
\end{proof}

\begin{prop}\label{prop:cCI}
Let $F\in \PST$ and let $c$ be a conductor of level $n$ for $F$.
Then 
\[\ulMCor\ni \sX=(\ol{X}, X_\infty)\mapsto F_c(\sX):=\{a\in F(\ol{X}\setminus|X_\infty|)\,|\, c_{\ol{X}}(a)\le X_\infty\}\]
defines an object in $\ulMPST$. Furthermore (see \ref{para:modulusII} for notations):
\begin{enumerate}[label=(\arabic*)]
\item\label{prop:cCI01} For any $\sX\in\ulMCor$ the pullback along the projection map $\sX\otimes \bcube\to \sX$ induces 
      an isomorphism $F_c(\sX)\cong F_c(\sX\otimes \bcube)$. In particular, 
       $\tau^*F_c\in\CI$.
\item\label{prop:cCI02} $\omega_! \tau^*F_c\cong F$.
\item\label{prop:cCI03} $F\in \NST\Rightarrow F_c\in \ulMNST$.
\end{enumerate}
\end{prop}
\begin{proof}
We start by showing $F_c\in \ulMPST$. Let $\sX=(\ol{X}, X_\infty)$ and $\sY=(\ol{Y}, Y_\infty)$ be two 
modulus pairs with corresponding opens $X$ and $Y$, respectively. 
We have to show that a left proper admissible prime correspondence $V\in \ulMCor(\sX, \sY)\subset \Cor(X, Y)$
sends the subgroup $F_c(\sY)\subset F(Y)$ to $F_c(\sX)\subset F(X)$. 
Take $a\in F_c(\sY)$ and a henselian dvf point $\rho: \eta=\Spec L\to \sX$  with $\td(L/k)\le n$. 
We have to show
\eq{prop:cCI1}{ c_L(\rho^* V^* a)\le v(\ol{\rho}^*X_\infty).}
Since $V\to X$ is finite, $(\eta\times_X V)_{\red}$ is a disjoint union of points $\eta_i=\Spec L_i$, 
with $L_i\in\Phin$. Thus 
\[V\circ \rho=\sum_i m_i\cdot \eta_i\in \Cor(\eta, Y),\]
with some multiplicities  $m_i\in \N$. For each $i$ we get a commutative diagram
\[\xymatrix{ \eta_i\ar[r]\ar@/^1.5pc/[rr]^{\rho_i}\ar[d]_{f_i} & V\ar[d]\ar[r]  & Y\\
                  \eta\ar[r]^{\rho} & X, 
}\]
where $\rho_i$ is a henselian dvf point of $Y$ and $f_i$ is finite. We have
$\eta_i=\Gamma_{\rho_i}\circ \Gamma_{f_i}^t$ in $\Cor(\eta, Y)$ (see \ref{para:PST} for the notation). 
Thus 
\eq{prop:cCI2}{\rho^*V^*=\sum_i m_i \cdot f_{i*}\rho_i^* : F(Y)\to F(\eta).}
Since the closure $\ol{V}$ of $V$ in $\ol{X}\times \ol{Y}$ is proper over $\ol{X}$ and 
$\rho$ extends to $\ol{\rho}$, we see that  $\rho_i$ extends to
$\ol{\rho}_i$ as in the diagram
\[\xymatrix{ \Spec \sO_{L_i}\ar@/^1.5pc/[rr]^{\ol{\rho}_i}\ar[r]\ar[d] & \ol{V}\ar[d]\ar[r] & \ol{Y}\\
                    \Spec \sO_L\ar[r]^{\ol{\rho}} & \ol{X}.
}\] 
Since $V$ satisfies the modulus condition \eqref{para:modulus2} we get
\eq{prop:cCI3}{v_{L_i}(X_\infty) \ge v_{L_i}(Y_\infty).}
Indeed, let $B$ be the local ring of $\ol{V}$ at the image of the closed point of $\sO_{L_i}$,
$x$ and $y\in B$ the local  equations for  ${X_{\infty}}_{\restriction \ol{V}}$ 
and  ${Y_{\infty}}_{\restriction \ol{V}}$, respectively,
and $\bar{x}$ and $\bar{y}$ their images in $\sO_{L_i}\setminus \{0\}$.
Then $\eqref{para:modulus2}$, says that $x/y\in \Frac(B)$ is a root of a monic polynomial $P(T)\in B[T]$.
It follows that $\bar{x}/\bar{y}\in L_i$ is a root of the image of $P(T)$ under $B[T]\to \sO_{L_i}[T]$,
i.e., $\bar{x}/\bar{y}$ is integral over $\sO_{L_i}$, i.e., $v_{L_i}(\bar{x})\ge v_{L_i}(\bar{y})$.

Let $j$ be an index with $c_L(f_{j*}\rho_{j}^*a) =\max_i\{c_L(f_{i*}\rho_{i}^*a) \}$. We obtain
\begin{align}
c_L(\rho^* V^* a) & = c_{L}(\sum_i m_i\cdot f_{i*}\rho_i^* a), &  & \text{by }\eqref{prop:cCI2}\label{prop:cCI3.1}\\
                          & \le  c_L(f_{j*}\rho_{j}^*a) , & &\text{by }\ref{defn:cond}\ref{c2}\notag\\
                          &\le  \left\lceil \frac{c_{L_{j}}(\rho_{j}^*a)}{e(L_j/L)}\right\rceil, &  &
                                                                        \text{by } \ref{defn:cond}\ref{c3}\notag\\
                          & \le \left\lceil \frac{v_{L_j}(Y_\infty)}{e(L_j/L)}\right\rceil, &  &a\in F_c(\sY)\notag\\
                          & \le \left\lceil \frac{v_{L_j}(X_\infty)}{e(L_j/L)}\right\rceil, & &\text{by }\eqref{prop:cCI3}\notag\\
                          & = v_{L}(X_\infty), & &\notag
\end{align}
where the last equality follows from $ v_{L_j}(X_\infty)=e(L_j/L)v_L(X_\infty)$.
This proves \eqref{prop:cCI1} and hence that $F_c$ is in $\ulMPST$.

Next, we prove \ref{prop:cCI01}.
Let $\sX=(\ol{X}, X_\infty)$ be a modulus pair with  $X= \ol{X}\setminus |X_\infty|$.
Denote by $\pi: X\times \A^1_k\to X$ the projection and by $i_0: X\inj X\times \A^1_k$ the zero section.
These define morphisms $\pi\in \ulMCor(\sX\otimes\bcube,\sX)$
and $i_0\in \ulMCor(\sX, \sX\otimes \bcube)$. We have to show that $\pi^*: F_{c}(\sX)\to F_c(\sX\otimes \bcube)$
is an isomorphism. Since $i_0^*\pi^*=\id_{F_c(\sX)}$, it suffices to show that $\pi^*$ is surjective.
Take $a\in F_c(\sX\otimes \bcube)$. For any henselian dvf point $\rho:\Spec L\to (\P^1_X, \{\infty\}_X)$, with $\td(L/k)\le n$,
we have
\[c_{L}(\rho^* a)\le v_{L}(X_{\infty}\times \P^1 +\ol{X}\times\{\infty\} )= v_L(X\times\{\infty\}). \]
Hence by \ref{defn:cond}\ref{c4}, there exists an element $b\in F(X)$ with  $\pi^*(b)=a$.
We have to check that $b\in F_c(\sX)$. Take $\rho: \Spec L\to \sX$ a henselian dvf point with $\td(L/k)\le n$.
Then $i_0\circ \rho: \Spec L\to \sX\otimes \bcube$ is a henselian dvf point and thus
\begin{align*}
c_{L}(\rho^*b) &= c_L(\rho^*i_0^*\pi^*b)=c_L((i_0\circ\rho)^*a)\\
                     &\le v_L(X_\infty\times \P^1+ \ol{X}\times \{\infty\})
                     = v_L(X_\infty).
\end{align*}
Hence $b\in F_c(\sX)$.  
Statement \ref{prop:cCI02} follows directly from \eqref{para:modulusII1} and
\ref{defn:cond}\ref{c5}. Finally \ref{prop:cCI03}.
For $\sX=(\ol{X}, X_\infty)$, the presheaf $F_{c, \sX}$ on $\ol{X}_{\et}$ (see \eqref{para:modulusIII0}) is given by 
\[(U\xr{u} \ol{X})\mapsto \{a\in F(U\setminus |u^*X_\infty|)\,|\, c_{U}(a)\le u^*X_\infty\}.\]
We have to show that this is a Nisnevich sheaf. Since $F$ is a Nisnevich sheaf it suffices 
to show the following: Let  $u:U\to \ol{X}$ be an \'etale map, $a\in F(U\setminus |u^*X_\infty|)$
 and assume there is a Nisnevich cover $\sqcup_i U_i\xr{\sqcup u_i} U$  
 so that $c_{U_i}(u_i^*a)\le u_i^*u^*X_\infty$, all $i$.
Then we have to show $c_{U}(a)\le u^*X_\infty$.
To this end, observe that 
if $\rho: \Spec L\to (U, u^*X_\infty)$ is a henselian dvf point with $\td(L/k)\le n$ 
 and $x\in U$ is the image point of the closed point of 
$\Spec \sO_L$, then by the functoriality of henselization $\ol{\rho}$ factors via
$\Spec \sO_L\to \Spec \sO_{U,x}^h\to U$. Hence there is an $i$ such that $\ol{\rho}$
factors via $\Spec \sO_L\to U_i\xr{u_i} U$. Thus 
$c_L(\rho^*a)\le v_L(u_i^*v^*X_\infty)=v_L(v^*X_\infty)$.
This completes the proof.
\end{proof}

\begin{para}\label{para:tFc}
Let $F\in \PST$ and let $c$ be a conductor of some level for $F$.
Let $F_c\in \ulMPST$ be as in Proposition \ref{prop:cCI}.
We set (see \ref{para:modulusII} for notation)
\[\tF_c:=\tau_!\tau^*F_c\in\ulMPST.\]
By adjunction we have a natural map
\[\tF_c\to F_c\]
which is injective. Indeed, on $\sX=(\ol{X}, X_\infty)\in\ulMCor$ it is given by the  inclusion inside 
$F(\ol{X}\setminus|X_\infty|)$
\[\tF_c(\sX)=\varinjlim_{\sY\in \Comp(\sX)}F_c(\sY)\to F_c(\sX).\]
By Proposition \ref{prop:cCI} and  \cite[Lem 4.2.5]{KMSY-MotModII} 
(or a similar argument as in the proof of \ref{prop:cCI}\ref{prop:cCI03}) we have 
\eq{para:tFc1}{F\in\NST\Rightarrow \tF_c\in\ulMNST.}
\end{para}

\begin{para}\label{para:CIF}
Let $F\in \RSC$. 
Denote by $\CI(F)$ the partially ordered set consisting  of those
subobjects  $G\subset \omega^{\CI} F$ in $\MPST$,
such that the induced map $\omega_!G\to \omega_!\omega^{\CI}F=F$ is an isomorphism, and where the partial order
is given by inclusion $G_1\subset G_2$ 
We set 
\[\CI(F)_{\Nis}:= \CI(F)\cap \MNST.\]
\end{para}

\begin{lemma}\label{lem:CItsp}
Let $F\in \RSC$ and $G\in \CI(F)$. Then $G_1=\tau_!G\in \ulMPST$ has the following properties:
\begin{enumerate}[label= (\arabic*)]
\item\label{para:CI-tau-sp1} the unit  $G_1\inj \ul{\omega}^*\ul{\omega}_!G_1$ 
               of the adjunction $(\ul{\omega}_!,\ul{\omega}^*)$ is injective;
\item\label{para:CI-tau-sp2} the counit $\tau_!\tau^*G_1\xr{\simeq} G_1$ of the adjunction $(\tau_!,\tau^*)$ is an isomorphism;
\item\label{para:CI-tau-sp3} for all $\sX\in \ulMCor$ the pullback $G_1(\sX)\xr{\simeq} G_1(\sX\otimes \bcube)$
                                       is an isomorphism.
\end{enumerate}
\end{lemma}
\begin{proof}
Note that \ref{para:CI-tau-sp2} follows directly from $\tau^*\tau_!=\id$. 
We show  \ref{para:CI-tau-sp1} and \ref{para:CI-tau-sp3}.
The inclusion $G\inj \omega^{\CI}F$ yields a commutative diagram
\[\xymatrix{
 G_1\ar@{^(->}[r]\ar[d] & \tau_!\omega^{\CI}F\ar@{^(->}[d]\\
\ul{\omega}^*\ul{\omega}_!G_1\ar[r] & \ul{\omega}^*\ul{\omega}_!\tau_!\omega^{\CI}F=\ul{\omega}^*F.
}\]
Here the top horizontal row is injective by the exactness of $\tau_!$, 
the vertical maps are induced by adjunction,
the vertical map on the right is injective by \eqref{para:modulusII3}.
It follows that the vertical map on the left is injective; furthermore the injectivity of the top horizontal map and 
\cite[Lem 1.15, 1.16]{Saito-Purity-Rec} imply that $G_1$ is $\bcube$-invariant. 
\end{proof}

\begin{remark}\label{rmk:CItsp}
The above lemma says that $\tau_!\CI(F)\subset  {}^\tau\mathbf{CI}^{sp}$, in the notation of \cite{Saito-Purity-Rec}.
\end{remark}

\begin{lemma}\label{lem:MspCI}
 Let $F\in \PST$  and let $c$ be a conductor of some level for $F$. 
Then $\tau^*\tF_c=\tau^*F_c\in \CI(F)$  (see \ref{para:tFc} for notation). If $F\in \NST$, then $\tau^*F_c\in \CI(F)_{\Nis}$.
\end{lemma}
\begin{proof}
By Proposition \ref{prop:cCI}\ref{prop:cCI02}, it suffices to show that there is an inclusion
$\tau^*F_c\inj \omega^{\CI}F$ inside $\omega^*F$. 
For $\sX$ a proper modulus pair set $\Ztr(\sX):= \MCor(-,\sX)$, and 
\[h_0^{\bcube}(\sX)= \Coker \left(\Ztr(\sX)(\bcube\otimes -) \xr{ i_0^*-i_1^*} \Ztr(\sX)\right).\]
By \cite[Lem 1.1.3]{KMSY-MotModI} and \cite[Lem 2.2.2]{KSY-RecII} we have (see \ref{para:modulus} and \ref{para:modulusII} 
for notation)
\[\omega_!\Ztr(\sX)= \Ztr(X), \quad \omega_! h_0^{\bcube}(\sX)= h_0(\sX),\]
where $X=\ol{X}\setminus|X_\infty|$.
Take $a\in F_c(\sX)\subset F(X)$. 
Since $F_c$ is cube invariant, by Proposition \ref{prop:cCI},  the Yoneda map $a: \Ztr(\sX)\to \tau^*F_c$ factors
via the quotient map $\Ztr(\sX) \to h_0^{\bcube}(\sX)$. Applying $\omega_!=\ul{\omega}_!\tau_!$
 we see that  the Yoneda map 
$a: \Ztr(X)\to F$ in $\PST$ defined by $a\in F(X)$ factors via $\Ztr(X)\to h^0(\sX)$, i.e., 
$a\in \omega^{\CI}F(\sX)$. This proves the lemma.  
\end{proof}

\begin{nota}\label{nota:Gmod}
Let $L\in\Phi$. Denote by $s\in S:=\Spec \sO_L$ the closed point.
For all $n\ge 1$ we have $(S, n\cdot s)\in\ulMCor^{\rm pro}$ (see \ref{para:ulMCorpro}).
Let $G\in \ulMPST$; we extend it to a presheaf on $\ulMCor^{\rm pro}$.
For $n\ge 0$ we introduce the following notation:
\[G(\sO_L,\fm_L^{-n}):=\begin{cases} 
\ul{\omega}_!G(S)= G(S,\emptyset) & \text{if }n=0\\ 
G(S, n\cdot s) & \text{if }n\ge 1. \end{cases}\]
\end{nota}


\begin{defn}\label{defn:mot-cond}
Let $F\in \RSC_{\Nis}$ and $G\in \CI(F)$ (see \ref{para:CIF}).
We denote by $c^G=\{c^G_L\}$ 
the family of maps $c_L^G: F(L)\to \N_0$, $L\in\Phi$, defined as follows
\[c^G_L(a):= \min\{n\ge 0\mid a\in \tau_!G(\sO_L,\fm_L^{-n})\}. \]
This is well-defined  since  
\[F(L)=\ul{\omega}_!\tau_!(G)(L)= \tau_!(G)(L,\emptyset) = \bigcup_n \tau_!G(\sO_L,\fm^{-n}).\]
In case $G=\omega^{\CI}F$  we write
\eq{defn:mot-cond1}{c^F:=c^{\omega^{\CI}F},}
and  call $c^F$ the {\em motivic conductor} of $F$. 
\end{defn}

\begin{thm}\label{thm:rec-cond}
Let $F$ be a presheaf with transfers. 
\begin{enumerate}[label=(\arabic*)]
\item\label{thm:rec-cond1}  If $F$ has a conductor $c$ of some level, then $F\in \RSC$.
\item\label{thm:rec-cond2} If $F\in \RSC_\Nis$ and $G\in \CI(F)$ (see \ref{para:CIF}),
then the family $c^G=\{c^G_L\}$ (see Definition \ref{defn:mot-cond})
                                     is a conductor for $F$ in the sense of Definition \ref{defn:cond}.
In particular, $c^F$ is a conductor for $F$.
\item\label{thm:rec-cond3}
Let $F\in\RSC_{\Nis}$ and  $G\in \CI(F)$.
Then in $\ulMPST$ 
\[\tau_!G\subset \tF_{c^G}\]
and for all $L\in \Phi$ and $n\ge 0$, we have
\[\tau_!G(\sO_L,\fm_L^{-n})= \tF_{c^G}(\sO_L,\fm_L^{-n}).\]
\item\label{thm:rec-cond4}
Let $F\in \RSC_{\Nis}$ and let $c$ be a conductor for $F$ (of some level). Then
\[\tF_c\subset \tau_!\omega^{\CI}F=\tF_{c^F},\]
where $c^F$ is the motivic conductor, see \eqref{defn:mot-cond1}.
\end{enumerate}
In particular, 
\[F\in \RSC_{\Nis} \Longleftrightarrow    F\in \NST \text{ and  $F$ has a conductor (of some level).}\]
\end{thm}
\begin{proof}
\ref{thm:rec-cond1}. We have $F=\omega_!\tau^*F_c\in \omega_!(\CI)\subset\RSC$,
by Proposition \ref{prop:cCI} and \cite[Prop 2.3.7]{KSY-RecII}.
Next \ref{thm:rec-cond2}. We check the properties from Definition \ref{defn:cond}. Set $G_1:=\tau_!G$.
\ref{c1} follows from $\ul{\omega}_!G_1(\sO_L)=\omega_!G(\sO_L)=F(\sO_L)$; \ref{c2} is obvious.
As for \ref{c3}, let $L'/L$ be a finite extension of henselian dvf's
with ramification index $e$. The induced finite morphism $f:\Spec \sO_{L'}\to \Spec \sO_L$ induces
 a morphism in $\ulMCor^{\rm pro}$:
\[ (\Spec \sO_L, n\cdot s_L)\rightarrow  (\Spec\sO_{L'},en\cdot s_{L'}),\]
where $s_L$ (resp. $s_{L'}$) are the closed points. This yields
the commutative diagram
\[ \xymatrix{
G_1(\sO_{L'},\emptyset)\ar[d]^{f_*}\ar[r] &G_1(\sO_{L'}, \fm_{L'}^{-en}) \ar[r]\ar[d]^{f_*} & 
                                                          \ul{\omega}^*\ul{\omega}_!G_1(\sO_{L'},\fm_{L'}^{-en})=F(L') \ar[d]^{f_*}\\ 
G_1(\sO_L,\emptyset)\ar[r] &G_1(\sO_L,\fm_L^{-n}) \ar[r] & 
\ul{\omega}^*\ul{\omega}_!G_1(\sO_{L}, \fm_L^{-n})=F(L).}\]
Hence, we obtain the following inequality which implies \ref{c3}:
\[ c^G_{L}(f_*a)\leq \min\{n\mid a\in G_1(\sO_{L'},\fm_{L'}^{-en})\} =\min\{n \mid c^G_{L'}(a)\leq en\}. \]

The following claim clearly implies \ref{c4}:
\begin{claim}\label{claim:rec-cond}
Let $X\in \Sm$ and $a\in F(\A^1_X)$. 
Assume $X$ connected with function field $K$.
Set 
$K(t)_\infty:=\Frac (\sO_{\P^1_K, \infty}^h)$
 inducing the henselian dvf point 
$\Spec K(t)_\infty \to (\P^1_K, \infty)$.
Assume $c^G_{K(t)_\infty}(a_K)\leq 1$, where $a_K\in F(\A^1_K)$ is the restriction of $a$. Then $a\in F(X)$.
\end{claim}
{\em Proof of Claim \ref{claim:rec-cond}.} The restriction map $F(\A^1_X)\to F(\A^1_K)$ is injective, 
by \cite[Thm 6]{KSY-Rec} and \cite[Cor 3.2.3]{KSY-RecII}; thus
it suffices to show $a_K\in F(K)$.
Set  $G_{1,\Nis}:=\ul{a}_\Nis(G_1)$ (see \ref{para:modulusIII}).
Consider the Nisnevich localization exact sequence
\[ G_{1,\Nis}(\P^1_K,\infty) \to G_{1,\Nis}(\A^1_K,\emptyset) \to 
G_1(K(t)_\infty,\emptyset)/G_1(\sO_{K(t)_\infty},\infty).\]
By \cite[Thm 4.1]{Saito-Purity-Rec}, we have 
$G_{1,\Nis}(\A^1_K,\emptyset)=G_1(\A^1_K,\emptyset)=F(\A^1_K)$.
Hence our assumption implies $a_K$ comes from $G_{1,\Nis}(\P^1_K,\infty)$ and
the desired assertion follows from the cube invariance of $G_{1,\Nis}$, see \cite[Thm 10.1]{Saito-Purity-Rec} 
(and Remark \ref{rmk:CItsp}),
\[ G_{1,\Nis}(\P^1_K,\infty)\simeq G_{1,\Nis}(K,\emptyset)=G_1(K,\emptyset)\simeq F(K).\]

Next we prove \ref{c5}. Let $X\in \Sm$ and $a\in F(X)$. We can assume that $X$ is not proper over $k$.
Take any $\sX=(\Xb,\Xinf)\in \MCor$ such that $X=\Xb-|\Xinf|$.
We have
\[ F(X) = \omega_!G(X)=\varinjlim_{n>0} G(\Xb,n\cdot \Xinf),\]
and hence $a\in G(\Xb,n\cdot \Xinf)$, for some $n$.
Then, for any henselian dvf point $\Spec L \to (\ol{X},n\cdot X_\infty)$, 
we get $a\in G_1(\sO_L,\fm_L^{-n v_L(\Xinf)})$ so that $c^F_L(a)\leq n\cdot v_L(\Xinf)$.
This completes the proof of \ref{thm:rec-cond2}.

\ref{thm:rec-cond3}. It follows directly from the definition of $F_{c^G}$ in Proposition \ref{prop:cCI}, that 
we have $\tau_!G\subset F_{c^G}$; hence also 
$\tau_!G=\tau_!\tau^*\tau_!G\subset \tau_!\tau^* F_{c^G}=\tF_{c^G}$.
Furthermore,  the equality  in the second part of the statement comes from the inclusions
\[\tau_!G(\sO_L,\fm_L^{-n})\subset \tF_{c^G}(\sO_L,\fm_L^{-n})\subset \{a\in F(L)\mid a\in \tau_!G(\sO_L,\fm_L^{-n})\},\]
where the first inclusion comes from the above  and the second holds by definition.
Finally \ref{thm:rec-cond4}. The inclusion $\tF_{c}\subset \tau_!\omega^{\CI}F$ follows from
Lemma \ref{lem:MspCI}. The equality $\tF_{c^F}=\tau_!\omega^{\CI}F$, now
follows from this and \ref{thm:rec-cond3}. This completes the proof.  
\end{proof}

\begin{cor}\label{cor:omegaCI}
The functor $\omega^{\CI}: \RSC\to \CI$ restricts to a functor
$\omega^{\CI}: \RSC_{\Nis}\to \CI_{\Nis}:=\CI\cap \MNST$.
\end{cor}
\begin{proof}
Take $F\in \RSC_\Nis$. By Theorem \ref{thm:rec-cond}, Proposition \ref{prop:cCI}\ref{prop:cCI03}, and \eqref{para:tFc1} 
we have $\tau_! \omega^{\CI}F=\tilde{F}_{c^F}\in\ulMNST$. Hence $\omega^{\CI} F\in \MNST$,
by definition, see \ref{para:modulusIII}. 
\end{proof}

\begin{nota}\label{not:ExtendingRSC}
Let $F\in \RSC_{\Nis}$. In the following we will simply write 
\[\tF:= \tF_{c^F}=\tau_!\omega^{\CI}F.\]
By Corollary \ref{cor:omegaCI} we have $\tau^*\tF\in \CI(F)_{\Nis}$ (see \ref{para:CIF}).
\end{nota}

\begin{cor}\label{cor:extending-cond}
Let $F\in \RSC_{\Nis}$.
Denote by $(c^{F})^{\le n}$ the restriction of the motivic conductor to $\td\le n$.
Assume $(c^{F})^{\le n}$ is a conductor of level $n$. Then 
\[\tF_{(c^{F})^{\le n}}= \tF.\]
\end{cor}
\begin{proof}
Clearly $\tF_{c^F}\subset \tF_{(c^{F})^{\le n}}$, and '$\supset$' holds by Theorem \ref{thm:rec-cond}\ref{thm:rec-cond4}.
\end{proof}

\begin{proposition}\label{prop:mot-cond-sub-sheaf}
Let $F_1\subset F_2$ be an inclusion in $\RSC_\Nis$.
Then the restriction of the motivic conductor of $F_2$ to $F_1$ is equal to the motivic conductor on $F_1$, i.e.,
\[c^{F_1}= (c^{F_2})_{|F_{1}}.\] 
\end{proposition}
\begin{proof}
Let $a\in F_1(X)$.
By the definition of the motivic conductor it suffices to show:
$a$ has modulus $(\ol{X}, X_\infty)$ as an element
in $F_2(X)$, if and only if it has the same modulus as an element in $F_1(X)$. This is obvious, see Definition \ref{defn:rec}. 
\end{proof}

\begin{lemma}\label{lem:cond-sum}
Let $F_1,F_2\in \RSC_\Nis$. Let $L\in\Phi$ and $a_i\in F_i(L)$.
Then $c_L^{F_1\oplus F_2}(a_1+a_2)=\max\{c^{F_1}_L(a_1), c^{F_2}_L(a_2)\}$.
\end{lemma}
\begin{proof}
Direct from Definition \ref{defn:mot-cond}.
\end{proof}

\begin{proposition}\label{prop:mot-cond-restr}
Let $k_1/k$ be an algebraic (hence separable) field extension and  let $F\in \RSC_{\Nis, k_1}$
(i.e. $F$ is a contravariant functor $\Cor_{k_1}\to \Ab$ which is a Nisnevich sheaf on $\Sm_{k_1}$ and has SC-reciprocity).
Denote by $R_{k_1/k}F: \Sm=\Sm_k\to \Ab$ the functor given by
\[X\mapsto R_{k_1/k}F(X):= F(X_{k_1}),\]
where $X_{k_1}= X\times_{\Spec k} \Spec k_1$. Then $R_{k_1/k}F\in \RSC_{\Nis}$ and its motivic conductor is given by
\[c_L^{R_{k_1/k}F}(a)= \max_i\{c^{F}_{L_i}(a_i)\}, \]
where $L\otimes_k k_1\cong \prod_i L_i$ and $a=(a_i)\in R_{k_1/k}F(L)=\prod_i F(L_i)$.
\end{proposition}
\begin{proof}
The first statement follows from the definition of $\RSC_{\Nis}$; for the second observe that for $L\in\Phi$ 
the $k_1$-algebra $L\otimes_k k_1=\prod_i L_i$ is unramified over $L$, hence (see \ref{not:ExtendingRSC} for notation)
\[\widetilde{R_{k_1/k}F}(\sO_L, \fm_L^{-n})=\prod_i \tF(\sO_{L_i}, \fm_{L_i}^{-n}).\]
This yields the statement.
\end{proof}

\subsection{Semi-continuous conductors}\label{subsec:c6}
\begin{definition}\label{defn:c6}
Let $F\in \PST$ and let $c$ be a conductor of level $n\in [1,\infty]$ for $F$.
We say $c$ is  {\em semi-continuous} if it satisfies the following condition:
\begin{enumerate}[label = {(c6)}]
\item\label{c6} Let $X\in \Sm$ with $\dim(X)\le n$ and $Z\subset X$ a smooth prime divisor with 
         generic point $z$  and $K=\Frac(\sO_{X,z}^h)$.  
            Then for any $a\in F(X\setminus Z)$ with $c_K(a_K)\le r$ there exists a Nisnevich neighborhood $u:U\to X$ of $z$ 
          and a compactification $\sY=(\ol{Y}, Y_\infty)$ of $(U, r\cdot u^*Z)$ such that 
(see Definition \ref{defn:condmod} for notation)
\[c_{\ol{Y}}(a_{U})\le Y_\infty,\]
 where $a_U$ (resp. $a_K$) denotes the restriction of $a$ to $U$ (resp. $K$).
\end{enumerate}
\end{definition}

\begin{lemma}\label{lem:c6}
Let $F\in \PST$ and let $c$ be a conductor of level $n$ for $F$.
The following statements are equivalent:
\begin{enumerate}[label=(\arabic*)]
\item\label{lem:c61} $c$ is semi-continuous;
\item\label{lem:c62} $\tF_c(\sO_L,\fm_L^{-r})=\{a\in F(L) \mid c_L(a)\le r\}$, for $L\in\Phin$,  $r\ge 0$.
\end{enumerate}
\end{lemma}
\begin{proof}
Let $a\in F(L)$. Then $a\in \tF_c(\sO_L,\fm_L^{-r})$ if and only if
there exists a smooth scheme $X$, a smooth prime  divisor $Z$ on $X$ with generic point $z$,  
 a $k$-isomorphism $\sO_L\cong\sO_{X,z}^h$, an element $\tilde{a}\in F(X\setminus Z)$ 
restricting to $a$, and a compactification
$\sY=(\ol{Y}, Y_\infty)$ of $(X, r\cdot Z)$, such that $c_{\ol{Y}}(\tilde{a})\le Y_\infty$.   
From this description we see that this '$\subset$' inclusion in \ref{lem:c62} always holds, while this '$\supset$' inclusion 
is equivalent to the semi-continuity of $c$.
\end{proof}

\begin{corollary}\label{cor:mot-cond-mini}
Let $F\in \RSC_{\Nis}$ and let $c$ be a semi-continuous conductor of level $n$ for $F$.
Then $(c^{F})^{\le n}\le c$, i.e., for all $L\in\Phin$ and all $a\in F(L)$ we have $c_L^F(a)\le c_L(a)$.
\end{corollary}
\begin{proof}
Follows from Theorem \ref{thm:rec-cond}\ref{thm:rec-cond4} and  Lemma \ref{lem:c6}.
\end{proof}

We can summarize part of the above as follows:

\begin{thm}\label{thm:summary}
Let $F\in \RSC_{\Nis}$.
\begin{enumerate}[label=(\arabic*)]
\item\label{thm:summary1} Any $G\in \CI(F)$ (see \ref{para:CIF}) 
defines a semi-continuous conductor $c^G$ (see \ref{defn:mot-cond}).
                                       For $G_1\subset G_2$ in $\CI(F)$ we have $c^{G_2}\le c^{G_1}$.
\item\label{thm:summary2} Let $c$ be a conductor of level $n\in [0,\infty]$. Then $\tau^*F_c\in\CI(F)_\Nis$.
        For  $c_1\le c_2$ we have $\tau^*F_{c_2}\subset \tau^*F_{c_1}$.
       If  furthermore $c$ is semi-continuous, then $c=(c^{\tau^*F_c})^{\le n}$. 
\item\label{thm:summary3} Assume $G\in\CI(F)_\Nis$ satisfies 
\[G(\sX)=\left\{a\in F(X)\middle \vert 
                     \rho^*a\in \tau_!G(\sO_L,\fm_L^{-v_L(X_\infty)}), \quad
\begin{minipage}{3.3cm} $\forall \,\rho: \Spec L\to \sX$,\\ with $L\in \Phi$\end{minipage}\right\},\]
for all proper modulus pairs  $\sX=(\ol{X}, X_\infty)$ with $X=\ol{X}\setminus|X_\infty|$. 
Then $\tau^*F_{c^G}=G$.
\item\label{thm:summary4} Let $c$ be a semi-continuous conductor of level $n$ for $F$ (possibly only defined on $\td\le n$).
Then there exists a unique semi-continuous conductor $c^\infty$ for $F$  with the following properties:
\[\tau^*F_c=\tau^*F_{c^\infty} \quad \text{and} \quad c=(c^\infty)^{\le n}.\]
We call $c^\infty$ the {\em canonical extension of $c$.}
\item\label{thm:summary5} Assume the restriction $(c^F)^{\le n}$ of the motivic conductor to $\td \le n$ is a conductor.
Then its canonical extension is the motivic conductor, i.e., $((c^F)^{\le n})^\infty=c^F$.
\end{enumerate}
\end{thm}
\begin{proof}
\ref{thm:summary1} holds Theorem \ref{thm:rec-cond}\ref{thm:rec-cond3} and Lemma \ref{lem:c6}.
\ref{thm:summary2} follows from Lemma \ref{lem:MspCI} and Lemma \ref{lem:c6}.
\ref{thm:summary3} holds by the definitions involved.
\ref{thm:summary4}. Set $G:=\tau^*F_c$. Then $G\in \CI(F)$ by \ref{thm:summary2} and it satisfies 
the condition from \ref{thm:summary3} by Lemma \ref{lem:c6}. Set $c^\infty:=c^{G}$.
Then $c^\infty$ has the wanted properties by \ref{thm:summary3} and  \ref{thm:summary2}.
Finally \ref{thm:summary5} follows from Corollary \ref{cor:extending-cond}.
\end{proof}

We finish this section with some lemmas which are needed later on.

\begin{definition}\label{defn:properRSC}
Let $F\in \RSC_\Nis$. We say $F$ is \em{proper} if the following equivalent conditions are satisfied:
\begin{enumerate}[label=(\arabic*)]
\item\label{defn:properRSC1}
 For all $X\in\Sm$ and any dense open $U\subset X$ the restriction map $F(X)\xr{\simeq} F(U)$ is an isomorphism.
\item\label{defn:properRSC2} Any conductor $c$ on $F$ is trivial, i.e., $c_L=0$ for all  $L$.
\end{enumerate}
(For this  $\ref{defn:properRSC2}\Rightarrow \ref{defn:properRSC1}$ implication use that \ref{c4} implies that $F\in \HI_{\Nis}$
and then the statement follows from Voevodsky's Gersten resolution, cf. \cite[Lem 10.3]{KY}.)
\end{definition}

\begin{lemma}\label{lem:prop-exact}
Let $0\to F_1\xr{\varphi} F\xr{\psi} F_2\to 0$
be an exact sequence in $\NST$ and with $F_1$, $F_2\in \RSC_{\Nis}$ and assume $F_1$ is {\em proper}.

Then $F\in \RSC_{\Nis}$. Any  (semi-continuous) conductor $c$ of level $n$ on $F_2$, 
 induces a (semi-continuous) conductor $c\psi=\{c_L\circ \psi\}_L$ of level $n$ on $F$. 
Furthermore, the motivic conductor of $F$ is given by $c^F= c^{F_2}\psi$
\end{lemma}
\begin{proof}
Let $c$ be a conductor of level $n$ on $F_2$. Then
$c\psi$ clearly satisfies \ref{c2}, \ref{c3}, \ref{c5} (and \ref{c6} if $c$ does). By the properness of $F_1$
we have an isomorphism $F(L)/F(\sO_L)\cong F_2(L)/F_2(\sO_L)$, which implies \ref{c1}.
Assume $a\in F(\A^1_X)$ satisfies the assumption in $\ref{c4}$ for $c\psi$. Let $\pi:\A^1_X\to X$ be the projection and
$i: X\inj \A^1_X$ the zero-section. Then $\psi (a-\pi^*i^*a)= \psi(a)-\pi^*i^*\psi(a)\in F_2(\A^1_X)$ satisfies the 
assumption from \ref{c4} for $c$; hence it lies in $\pi^* F_2(X)$, hence is zero; 
therefore $a-\pi^*i^*a\in F_1(\A^1_X)= \pi^*F_1(X)$, 
hence it is zero, i.e., $a=\pi^*i^*a$. This shows that $c\psi$ satisfies \ref{c4}. Therefore, $c\psi$ is a conductor of level $n$.
Thus Theorem \ref{thm:rec-cond} yields $F\in \RSC_\Nis$ and  
$\tF_{c^{F_2}\psi}(\sO_L,\fm_L^{-n})\subset\tF(\sO_L,\fm_L^{-n})$. We have inclusions 
\[\tF_{c^{F_2}\psi}(\sO_L,\fm_L^{-n})/F_1(\sO_L)\inj \tF(\sO_L,\fm_L^{-n})/F_1(\sO_L)\inj \tF_2(\sO_L,\fm_L^{-n}), \]
where the second map is injective by  the properness of $F_1$.
Since $\tF_2=\tF_{2, c^{F_2}}$, the composition is an isomorphism; hence $c^F=c^{F_2}\psi$.
\end{proof}

\begin{lemma}\label{lem:im-cond}
Let $\varphi: F\surj G$ be a surjection in  $\NST$. 
Let $c=\{c_L: F(L)\to \N\}_{L\in \Phi_{\le n}}$ be a collection of maps.
Define $\bar{c}=\{\bar{c}_L: G(L)\to \N\}_{L\in \Phi_{\le n}}$ by
\[\bar{c}_L(a):=\min\{c_L(\tilde{a})\mid \tilde{a}\in F(L) \text{ with }\varphi(\tilde{a})=a\}.\]
If $c$ satisfies \ref{c1} (resp. \ref{c2}, \ref{c3}, \ref{c6} ),  then so does $\bar{c}$.

Furthermore, if $\varphi$ has the following property:
For all $X\in \Sm$ there exists a proper modulus pair $(\ol{X},X_\infty)$ with $X=\ol{X}\setminus X_\infty$,
such that for all $x\in \ol{X}$ the map $\varphi$ induces a surjection 
\eq{lem:im-cond1}{F(\ol{X}_{(x)}^h\setminus X_{\infty, (x)})\surj G(\ol{X}_{(x)}^h\setminus X_{\infty, (x)}),}
where $\ol{X}_{(x)}^h=\Spec \sO_{\ol{X},x}^h$  and $X_{\infty, (x)}$ denotes the restriction of $X_\infty$
to $\ol{X}^h_{(x)}$.
Then $\bar{c}$ satisfies \ref{c5}, if $c$ does.
\end{lemma}
\begin{proof}
\ref{c1}. If $\bar{c}_L(a)=0$, then there exists a lift $\tilde{a}\in F(L)$ with $c_L(\tilde{a})=0$, hence $\tilde{a}\in F(\sO_L)$, 
by \ref{c1} for $c$, hence $a\in G(\sO_L)$. 

\ref{c2}. Let $a,b\in G(L)$. Take lifts $\tilde{a}, \tilde{b}\in F(L)$ with $c_L(\tilde{a})=\bar{c}_L(a)$ and 
$c_L(\tilde{b})=\bar{c}_L(b)$. Then by $\ref{c2}$ for $c$
\[\bar{c}_L(a+b)\le c_L(\tilde{a}+\tilde{b})\le \max\{c_L(\tilde{a}), c_L(\tilde{b})\}= \max\{\bar{c}_L(a), \bar{c}_L(b)\}.\]

\ref{c3}. Let $f:\Spec L'\to \Spec L$ be a finite  extension with ramification index $e$
 and let $a\in G(L')$. Take a lift $\tilde{a}\in F(L')$ with $\bar{c}_{L'}(a)=c_{L'}(\tilde{a})$. Then by \ref{c3} for $c$
\[\bar{c}_L(f_*a)\le c_L(f_*\tilde{a})\le \left\lceil\frac{c_{L'}(\tilde{a})}{e}\right\rceil
= \left\lceil\frac{\bar{c}_{L'}(a)}{e}\right\rceil.\]

\ref{c6}. Let $X, z\in Z, K$ be as in \ref{c6} and $a\in G(X\setminus Z)$ with $\bar{c}_K(a_K)\le r$.
Let $\tilde{a}_K\in F(K)$ be a lift of $a_K$ with $c_K(\tilde{a}_K)=\bar{c}_K(a_K)$.
Since $\Spec K=\Spec  \sO_K\setminus Z_{\sO_K}$, we find a Nisnevich neighborhood
$U\to X$ of $z$ and an element $\tilde{a}\in F(U\setminus Z)$ which restricts to $\tilde{a}_K$.
After possibly shrinking $U$ around $z$, we may assume that $\varphi(\tilde{a})= a_{|U\setminus Z_U}$.
By \ref{c6} for $F$, we may shrink $U$  further around $z$ to obtain a compactification $\sY=(\ol{Y}, Y_\infty)$
of $(U, r\cdot Z_U)$ such that 
\[\bar{c}_{\ol{Y}}(a_U)\le c_{\ol{Y}}(\tilde{a}_U)\le Y_\infty.\]

\ref{c5}(assuming \eqref{lem:im-cond1}). 
Let $X\in\Sm$ and $a\in G(X)$.  Let $\sX=(\ol{X}, X_\infty)$ be a proper modulus pair with
$X=\ol{X}\setminus |X_\infty|$ as in \eqref{lem:im-cond1}. 
This condition implies that we find a finite Nisnevich cover  $\{U_i\to \ol{X}\}_i$ and $\tilde{a_i}\in F(U_{i,X})$,
 such that  $\varphi(\tilde{a}_i)=a_{|U_{i,X}}$ in $G(U_{i,X})$, where $\{U_{i,X}\to X\}_i$ is the induced Nisnevich cover of $X$.
Let  $\sY_i=(\ol{Y}_i, Y_{i,\infty})$ be a compactification of $(U_i, X_{\infty|U_i})$ 
which admits a  morphism $\ol{Y_i}\to \ol{X}$ extending $U_i\to \ol{X}$ and inducing a morphism of  proper modulus pairs
$\sY_i\to \sX$. By $\ref{c5}$ for $c$ and (the proof of) Lemma \ref{lem:indepmod} we find an integer $N>>0$, such that 
$c_{L}(\rho^*\tilde{a}_i)\le N\cdot v_L(Y_{i,\infty})$, for all $\rho: \Spec L\to U_{i,X}= \ol{Y}_i\setminus|Y_{i,\infty}|$, 
$L\in \Phi_{\le n}$.
Let $\rho:\Spec L\to X$ be any henselian dvf point with $L\in \Phi_{\le n}$; denote by $s\in \ol{X}$ the image of the closed
point under the induced map $\bar{\rho}:\Spec \sO_L\to \ol{X}$. By the Nisnevich property,
 there exists an $i$ and a point $s_i\in U_i$ such that $U_i\to \ol{X}$ induces an isomorphism $s_i\xr{\simeq}s$.
Hence $\bar{\rho}$ factors via $U_i\inj \ol{Y}_i\to \ol{X}$.
Thus
\[\bar{c}_L(\rho^*a)\le c_L(\rho^* \tilde{a}_i)\le N \cdot v_L(Y_{i,\infty})= N\cdot v_L(X_\infty),\]
where for the equality we used $(Y_{i,\infty})_{|U_i}= (X_\infty)_{|U_i}$.
Thus $a$ satisfies \ref{c5} for $(\ol{X}, N\cdot X_\infty)$.
\end{proof}

\subsection{Homotopy invariant subsheaves}

\begin{corollary}\label{cor:HI}
Let $F\in \NST$ be $\A^1$-invariant (in particular $F\in \RSC_{\Nis}$). Then the motivic conductor of $F$ is given by
\[c^F_L(a)=\begin{cases} 0 & \text{if }a\in F(\sO_L)\\ 1 & \text{else.}\end{cases}\]
\end{corollary}
\begin{proof}
The right hand side defines a conductor, as already remarked in \ref{rmk:cond}; it is clearly semi-continuous.
By Corollary \ref{cor:mot-cond-mini} we get '$\le$' in the statement and \ref{c1} forces it to be an equality.
\end{proof}

\begin{para}\label{para:HIsub}
We denote by $\HI$ the category of $\A^1$-invariant presheaves with transfers and set $\HI_{\Nis}:=\HI\cap \NST$.
It follows immediately from Definition \ref{defn:rec} that we have $\HI\subset\RSC$ and $\HI_{\Nis}\subset \RSC_{\Nis}$.

Let $F\in \PST$. For $X\in \Sm$, we denote by 
\[h^0_{\A^1}(F)(X)\]
 the subset of $F(X)$ formed by those sections $a\in F(X)$ for  which
the Yoneda map $a:\Ztr(X)\to F$ factors via 
\[h_0^{\A^1}(X)= \Coker(\Ztr(X)(-\times\A^1_k)\xr{i_0^*- i_1^*}\Ztr(X))\in \PST.\]
We immediately see that $X\mapsto h^0_{\A^1}(F)(X)$ defines a sub-presheaf with transfers of $F$, 
since $h_0^{\A^1}(X)\in \HI$  (see, e.g., \cite[Prop 3.6]{VoPST})
we have $h^0_{\A^1}(F)\in \HI$; furthermore,  it has the following universal
property: any morphism  $H\to F$ in $\PST$ with $H\in \HI$ factors uniquely via
a morphism $H\to h^0_{\A^1}(F)$ in $\HI$. 
Note,  if $F\in\NST$, then $h^0_{\A^1}(F)\in \HI_{\Nis}$. Indeed, by \cite[Thm 3.1.12]{VoDM}
Nisnevich sheafification induces an exact functor $\HI\to \HI_{\Nis}$, thus we obtain natural inclusions in $\PST$
\[h^0_{\A^1}(F)\inj h^0_{\A^1}(F)_{\Nis}\inj F_{\Nis}=F,\]
since $h^0_{\A^1}(F)_{\Nis}\in \HI$ the second inclusion factors via $h^0_{\A^1}(F)$;
hence $h^0_{\A^1}(F)=h^0_{\A^1}(F)_{\Nis}$.
\end{para}

\begin{proposition}\label{prop:HIc}
Let $F\in \PST$ and let $c$ be a conductor of level $n$ for $F$.
Then 
\[ X\mapsto F^{c\le 1}(X):=\left\{a\in F(X)\;\middle\vert\; c_L(\rho^*a)\le 1,\quad 
\begin{minipage}{3.2cm} $\forall \rho: \Spec L\to X$\\ with $L\in\Phin$\end{minipage}\right\}\]
defines a homotopy invariant sub-presheaf with transfers of $F$. 
If $F\in\NST$, then $F^{c\le 1}\in \HI_{\Nis}$.
\end{proposition}
\begin{proof}
To show  $F^{c\le 1}\in \PST$ is equivalent to  the following: let $V\in \Cor(X, Y)$ 
be a finite prime correspondence and $a\in F^{c\le 1}(Y)$; then for all henselian dvf points 
$\rho: \Spec L\to X$ with $\td(L/k)\le n$, we have 
\[c_L(\rho^*V^*a)\le 1.\]
This follows from the calculation in \eqref{prop:cCI3.1}.
The $\A^1$-invariance of $F^{c\le 1}$ follows directly from \ref{c4}.
The last statement is proven similarly as in Proposition \ref{prop:cCI}\ref{prop:cCI03}.
\end{proof}

\begin{cor}\label{cor:HIc}
Let $F\in \RSC_{\Nis}$ with motivic conductor $c^F$. Then 
\[h^0_{\A^1}(F)= F^{c^F\le 1}.\]
\end{cor}
\begin{proof}
By Proposition \ref{prop:HIc} we have $F^{c^F\le 1}\subset h^0_{\A^1}(F)$. 
By Proposition \ref{prop:mot-cond-sub-sheaf} and Corollary \ref{cor:HI}
we have $(c^F)_{|h^0_{\A^1}(F)}=c^{h^0_{\A^1}(F)}\le 1$; hence $h^0_{\A^1}(F)\subset F^{c^F\le 1}$.
\end{proof}

\begin{cor}\label{cor:bir-inv-h0}
Let $F\in \RSC_\Nis$. Assume  for all $L\in\Phi$ we have 
\eq{cor:bir-inv-h01}{\tF(\sO_L,\fm_L^{-1})=F(\sO_L).}
Let $X\in \Sm$ be proper over $k$ and $U\subset X$ dense open. Then
\[h^0_{\A^1}(F)(U)=F(X).\]
In particular, if $F$ satisfies \eqref{cor:bir-inv-h01}, then $X\mapsto F(X)$ is a birational invariant on smooth proper schemes.
\end{cor}
\begin{proof}
By Corollary \ref{cor:HIc} 
\[h^0_{\A^1}(F)(U)= F^{c^{F}\le 1}(U).\]
Hence $F(X)\subset h^0_{\A^1}(F)(U)$ and by \eqref{cor:bir-inv-h01} also
\[h^0_{\A^1}(F)(U)\subset \bigcap_{x\in X^{(1)}} F(\sO_{X,x}^h).\]
By \cite[Cor 0.3]{Saito-Purity-Rec} 
\[\bigcap_{x\in X^{(1)}} F(\sO_{X,x}^h)= F(X).\]
All together yields the statement.
\end{proof}

\subsection{Local symbols}\label{sec:symb-mod}
\begin{para}\label{para:symbol}
We recall the notion of local symbols for reciprocity sheaves, see  \cite[III, \S 1]{SerreGACC},
\cite[Prop 5.2.1]{KSY-Rec} or \cite[1.5]{IR} for details.

Let $F\in \RSC_{\Nis}$. If $L/K$ is a finite field extension of finitely generated fields over $k$, we denote by 
$\Tr_{L/K}: F(L)\to F(K)$ the map induced 
by the transfer structure on $F$. For $X\in \Cor^{\rm pro}$, $x\in X$, and $a\in F(X)$  we denote $a(x)\in F(x)$ 
the pullback of $a$ along  $x\inj X$.

Let $K$ be a function field over $k$ and $C$ a  regular projective $K$-curve. Note that  
$C\in \Cor^{\rm pro}$ by Lemma \ref{lem:regular-proj}.  For $x\in C_{(0)}$ a closed point 
we write $v_x$ for the corresponding normalized discrete valuation on 
$K(C)^\times$, $\fm_x\subset \sO_{C,x}$ for the maximal ideal, and set $U^{(n)}_x:= 1+\fm_x^n\subset \sO_{C,x}^\times$, 
$n\ge 1$.
Let $D=\sum n_x\cdot x$ be an effective Cartier divisor  on $C$ and $a\in\tF(C,D)$
 (see \ref{not:ExtendingRSC} for the notation $\tF$).
Then  there exists a family of maps
\[\{(a,-)_{C/K,x}: K(C)^\times\to F(K)\}_{x\in C_{(0)}}\]
which is uniquely determined by the following properties:
\begin{enumerate}[label=(LS\arabic*)]
\item\label{LS1} $(a, -)_{C/K,x}: K(C)^\times\to F(K)$ is a group homomorphism;
\item \label{LS2} $(a, f)_{C/K,x}= v_x(f) \Tr_{K(x)/K}(a(x))$, for $x\in C\setminus|D|$;
\item \label{LS3} $(a, U^{(n_x)}_x)_{C/K,x}=0$;
\item\label{LS4} $\sum_{x\in C_{(0)}} (a,f)_{C/K,x}=0$.
\end{enumerate}
It follows from the uniqueness that the family $\{(a,-)_{C/K,x}\}$ does not depend on the chosen modulus $D$.
Furthermore, from the uniqueness one can deduce the following properties:
\begin{enumerate}[label=(LS\arabic*), resume]
\item\label{LS5}  $(-,-)_{C/K,x}: F(K(C))\times K(C)^\times\to F(K)$ is bilinear;
\item\label{LS6} let $h: F\to G$ be a morphism in $\RSC_{\Nis}$, then in $G(K)$
                      \[h((a,f)_{C/K,x})=(h(a), f)_{C/K,x}, \quad \text{all }a\in F(K(C)), f\in K(C)^\times.\]
\end{enumerate}
Let $K'/K$ be a finite field extension, $C'/K'\in \Cor^{\rm pro}$ a projective curve,  
 and $\pi: C'\to C$ a finite morphism over $\Spec K' \to \Spec K$, then:
\begin{enumerate}[label= (LS\arabic*), resume]
\item\label{LS7}  for  $b\in F(K'(C'))$, $f\in K(C)^\times$, and $x\in C_{(0)}$ we have 
\[(\pi_*(b), f)_{C/K, x}=\sum_{y/x} \Tr_{K'/K} (b,  \pi^*f)_{C'/K',y};\]
\item\label{LS8} for $a\in F(K(C))$, $g\in K'(C')^\times$, and $x\in C_{(0)}$ we have 
\[(a, \pi_*g)_{C/K,x}=\sum_{y/x} \Tr_{K'/K}(\pi^*(a), g)_{C'/K',y};\]
\end{enumerate}
where in both cases the sum is over all $y\in C'$ mapping to $x$.
\end{para}
\begin{lemma}\label{lem:symbol-bc}
Let $F\in \RSC_{\Nis}$,  $C$ be a regular projective and geometrically connected $K$-curve.
Let $K'/K$ be a finitely generated field extension, 
denote by $\tau: \Spec K'\to \Spec K$ the induced map, and by $\tau_C: C_{K'}= C\otimes_K K'\to C$ the projection.
Then 
\[\sum_{y\in \tau_{C}^{-1}(x)}(\tau^*_C a, \tau_C^*f)_{C_{K'}/K', y} = \tau^*(a,f)_{C/K,x},\quad \text{in } F(K')\]
for all $a\in F(K(C))$, $f\in K(C)^\times$, and $x\in C_{(0)}$. 
\end{lemma}
\begin{proof}
Let $U\subset C$ be open with $a\in F(U)$. Using the Approximation Lemma, \ref{LS1}, and \ref{LS3}
 we can assume that for a given 
$m\ge 1$ we have $f\in U^{(m)}_z$, for all $z\in C\setminus (U\cup\{x\})$; in particular choosing $m$ large enough we get
$(a, f)_{C/K,z}=0$. Identifying $f$ with the finite $K$-morphism $C\to \P^1_K$ 
we obtain $a\in F(f^{-1}(\P^1_K\setminus\{1\})\setminus\{x\})$
and \ref{LS2}, \ref{LS4} yield
\[-(a, f)_{C/K,x}= (i_0^*-i_\infty^*)f_*a, \quad \text{in } F(K).\]
The formula in the statement now follows by applying $\tau^*$ to this equality, using the base change formula 
$\tau_{\P^1}^* \circ f_*= (\tau_C^*f)_* \tau_C^*$ induced by the cartesian diagram
\[\xymatrix{
C_{K'}\ar[r]^{\tau_{C}}\ar[d]_{\tau_C^*f} & C\ar[d]^f\\ 
 \P^1_{K'}\ar[r]^{\tau_{\P^1}} & \P^1_K,
}\]
and using \ref{LS1} - \ref{LS4} backwards.
\end{proof}

\begin{lemma}\label{lem:approx-h}
Let $L\in\Phi$. Let $C$ be a regular curve over a $k$-function field $K$.
Assume there exists a closed point $x\in C$ and a $k$-morphism $u: \Spec \sO_L\to C$ 
inducing an isomorphism $\sO_{C,x}^h\cong \sO_L$.
Then there is an isomorphisms induced via pullback along $u$
\[F(K(C))/F(\sO_{C,x})\xr{\simeq} F(L)/F(\sO_L).\]
If $\sO_{C,x}$ has a coefficient field then we have an isomorphism 
\[F(K(C))/F(\sO_{C,x},\fm)\xr{\simeq} F(L)/F(\sO_L,\fm), \]
where for a local ring $A\in \Cor^{\rm pro}$ with maximal ideal $\fm$ we set
\[F(A,\fm):=\Ker(F(A)\to F(A/\fm)).\]
\end{lemma}
\begin{proof}
We prove the first isomorphism.
The natural map in the statement is compatible with pullbacks and pushforwards on both sides. Thus
we can apply the standard trick replacing $k$ by its maximal pro-$\ell$ extensions for various primes $\ell$,
 to  assume $k$ is infinite.
By  Gabber's Presentation Theorem (see, e.g., \cite[3.1.2]{CTHK}) 
we find an open $U\subset C$ containing $x$, a $k$-function field $E$ and an \'etale morphism 
$\varphi: U\to \P^1_E$ such that $x=\varphi^{-1}(\varphi(x))$ and $\varphi$ induces an 
isomorphism $x\xr{\simeq} \varphi(x)$.
It follows from \cite[Lem 4.2, Lem 4.3]{Saito-Purity-Rec}, that $(U, n\cdot x)$ is a $V$-pair, for all $n\ge 1$, 
in the sense of \cite[Def 2.1]{Saito-Purity-Rec}. If 
$v: U'\to U$ is an affine Nisnevich neighborhood of $x$  with $v^{-1}(x)=\{x'\}$, then 
the pullback $v^*:F(K(U))/F(\sO_{U,x})\xr{\simeq} F(K(U'))/F(\sO_{U',x,})$ is an isomorphism, by
\cite[Lem 4.4, (3)]{Saito-Purity-Rec}.
We obtain the first isomorphism of the statement  by taking the limit over all Nisnevich neighborhoods $v$.
For the second isomorphism observe that if a coefficient field $\sigma: \kappa\inj \sO_{C,x}$ exists, then $\sigma^*$
induces a splitting of  the restriction to the closed point $F(\sO_{C,x})\to F(\kappa)$, in particular it is surjective.
We obtain isomorphisms
$F(\sO_{C,x})/F(\sO_{C,x},\fm_x)\cong F(\kappa)\cong F(\sO_L)/F(\sO_L,\fm_L)$
which together with the first statement and the five lemma  yield the second isomorphism in the statement.
\end{proof}

\begin{para}\label{para:symbol-hdvf}
Let $F\in \RSC_{\Nis}$.  Let $L\in\Phi$ have residue field $\kappa=\sO_L/\fm_L$, and let
$\sigma: K\inj \sO_L$ be a $k$-homomorphism such that the induced map  $K\inj \kappa$ is a finite field extension
 (e.g., $\sigma$ could be a coefficient field.)
We define the local symbol
\[(-,-)_{L,\sigma}: F(L)\times L^\times\to F(K),\]
as follows:
We find a regular projective $K$-curve $C$ and a  $\kappa$-point  $x\in C(\kappa)$ satisfying 
\eq{para:symbol-hdvf1}{ L= \Frac (\sO_{C,x}^h), \quad \sigma: K\to \sO_{C,x}\xr{\rm nat.} \sO_L.}
Additionally we assume that {\em $\sO_{C,x}$ has a coefficient field}.
Denote by $u: \Spec \sO_L\to \Spec \sO_{C,x}$ the induced map.
The  symbol  $(-,-)_{L,\sigma}$ is defined as the composition
\mlnl{(-,-)_{L,\sigma}: F(L)\times L^\times\longrightarrow 
F(L)/F(\sO_L,\fm_L)\times \varprojlim_{n} L^\times/ U_L^{(n)}\\
\stackrel{\ref{lem:approx-h}}{\cong} F(K(C))/F(\sO_{C,x},\fm_x)\times \varprojlim_{n} K(C)^\times/ U_x^{(n)}
\xr{\ol{(-,-)}_{C/K,x}} F(K),}
where the last map is given by 
\[\ol{(a, (f_n))}_{C/K,x}:=(\tilde{a}, \tilde{f}_r)_{C/K,x}\]
with $\tilde{a}\in \tF(\sO_{C,x},\fm^{-r})$ a lift of $a$ and $\tilde{f}_r\in K(C)^\times$ a lift of $f_r$;
this is well-defined and bilinear by \ref{LS2}, \ref{LS3}, and \ref{LS5}. 
\end{para}
\begin{lemma}\label{lem:symbol-well-def}
The symbol $(-, -)_{L,\sigma}$ defined in \ref{para:symbol-hdvf} above is  independent of the choice
of the presentation \eqref{para:symbol-hdvf1}.
\end{lemma}
\begin{proof}
Let $v: C'\to C$ be  a $K$-morphism between regular projective $K$-curves, 
 let $x\in C$ and $x'\in C'$ be  closed points such that $v$ is \'etale in a neighborhood of $x'$ and induces
an isomorphism $x'\xr{\simeq}x$. Assume that $\sO_{C,x}$ has a coefficient field.
Let $E=K(C)$ and $E'=K(C')$ be the function fields.
Then it suffices to show, that for all $a\in F(E)$ and $f\in E^\times$ we have 
\eq{lem:symbol-well-def0}{(a, f)_{C/K,x}= (v^*a, v^*f)_{C'/K, x'}. }
We denote $E^\times_x:=\varprojlim E^\times / U^{(n)}_x$ etc.
Then the composition 
\[E_x^\times\xr{1\otimes\id} \big(E_x\otimes_{E} E'\big)^\times\cong \prod_{y/x} {E'}_y^\times
\xr{\rm proj} {E'}_{x'}^\times\]
is induced by $v^*$ and is an isomorphism  with inverse induced by the norm. 
Thus we can use the Approximation Lemma, \ref{LS3}, and the continuity of the norm map
to choose $g\in {E'}^\times$ close to $v^*f$ at $x'$ and 
close to $1$ at all $y\in v^{-1}(x)\setminus\{x'\}$ to obtain 
\begin{enumerate}[label= (\arabic*)]
\item\label{lem:symbol-well-def1} $(v^*a, v^*f)_{C'/K, x'}= (v^*a, g)_{C'/K, x'}$;
\item\label{lem:symbol-well-def3} $(v^*a, g)_{C'/K, y'}=0$, for all $y'\in v^{-1}(x)\setminus\{x'\}$;
\item\label{lem:symbol-well-def2} $(a, f)_{C/K,x}= (a, {\rm Nm}_{E'/E}(g))_{C/K,x}$.
\end{enumerate}
We obtain
\[ (a, f)_{C/K,x}  \stackrel{\ref{LS8},\, \ref{lem:symbol-well-def2}}{=} 
\sum_{y'\in v^{-1}(x)} (v^*a, g)_{C'/K, y'} 
\stackrel{\ref{lem:symbol-well-def1}, \ref{lem:symbol-well-def3}}{=} (v^*a, v^*f)_{C'/K, x'},
\]
which yields the statement.
\end{proof}
\begin{remark}
Note that if the composition $K\xr{\sigma} \sO_L\to \kappa$ is purely inseparable, then there does not need to exist 
a coefficient field of $\sO_L$ which contains $K$. This is why in \ref{para:symbol-hdvf} it does in general not suffice
to consider coefficient fields. (In characteristic zero it does.)
For coefficient fields $\sigma: K\inj \sO_L$ the symbol $(-,-)_{L,\sigma}$ 
will in general depend on the choice of $\sigma$.
\end{remark}

\begin{corollary}\label{cor:localLS3}
 Let $L_1/ L$   be an extension of henselian dvf's  of ramification index $e$, i.e., 
$\fm_L\sO_{L_1}=\fm_{L_1}^e$. (The extension $L_1/L$ does not need be algebraic or finitely generated.)
 Let $\sigma_1: K\to \sO_{L_1}$ be a $k$-homomorphism inducing
a finite field extension $K\inj \sO_{L_1}/\fm_{L_1}$. 
Let $F\in \RSC_\Nis$ and $a\in \tF(\sO_{L},\fm_L^{-r})$, $r\ge 0$. Then
\[(a_{L_1}, U^{(er)}_{L_1})_{L_1,\sigma_1}=0,\]
where $a_{L_1}\in F(L_1)$ is the pullback of $a$.
\end{corollary}
\begin{proof}
We have $a_{L_1}\in \tF(\sO_{L_1},\fm_{L_1}^{-er})$ and hence the statement follows from
the construction of the symbol in \ref{para:symbol-hdvf} and \ref{LS3} .
\end{proof}

\begin{lemma}\label{lem:CRT-RS}
Let $F\in \RSC_{\Nis}$.
Let $K/k$ be a function field, $X$ a normal affine integral finite type $K$-scheme with function field $E$.
Let $x_i\in X^{(1)}$, $i=1,\ldots, r$, be distinct one codimensional  points.
Then for all integers $n_i\ge 0$ the natural map
\[\frac{F(E)}{\cap_{i=1}^r \tilde{F}(\sO_{X,x_i}, \fm_{x_i}^{-n_i})}\xr{\simeq} 
       \prod_{i=1}^r \frac{F(E)}{\tF(\sO_{X, x_i}, \fm_{x_i}^{-n_i})}\]
is an isomorphism, where ${\tF(\sO_{X, x_i}, \fm_{x_i}^{0})}:= F(\sO_{X, x_i})$.
\end{lemma}
\begin{proof}
Let $A$ be  the semi-localization of $X$ at the points $x_i$ and denote by $D=\sum_i n_i x_i$ the divisor on $U:=\Spec A$.
(Note that we allow $|D|\varsubsetneq\{x_1,\ldots, x_r\}$.)
 We claim 
\eq{lem:CRT-RS1}{\tF(U,D)= \cap_{i=1}^r \tF(\sO_{X,x_i}, \fm_{x_i}^{-n_i}).}
Indeed, by definition $\tF(U,D)= \tF_{(U,D)}(U)$; furthermore
$\tF_{(U,D)}$ is a sheaf on $U_{\Nis}$ and is a subsheaf of the constant sheaf $F(K)$ 
(by \cite[Thm 6]{KSY-Rec} and \cite[Cor 3.2.3]{KSY-RecII});
since $\Spec E \sqcup_i \Spec \sO_{X, x_i}\to U$ is a Nisnevich cover the claim \eqref{lem:CRT-RS1} follows.

The natural map in the statement is compatible with pullbacks and pushforwards on both sides. Thus
we can apply the standard trick replacing $k$ by its maximal pro-$\ell$ extensions for various primes $\ell$,
 to  assume $k$ is infinite.
By  Gabber's Presentation Theorem (see, e.g., \cite[3.1.2]{CTHK}) 
we find a function field $K_1/k$ and an essentially \'etale morphism $\varphi: U\to \A^1_{K_1}$ such that
$\{x_1,\ldots, x_r\}=\varphi^{-1}\varphi(\{x_1,\ldots, x_r\})\cong \varphi(\{x_1,\ldots, x_r\})$.
 By \cite[Lem 4.2, Lem 4.3]{Saito-Purity-Rec} 
$(U, \sum_{i} m_i x_i)$ is a V-pair, for all $m_i\ge 0$.
 Let $U^h$ be the henselization of $U$ with respect to the radical in $A$ (see \cite[XI, \S2, Thm 2]{Raynaud})
and set $D^h:=D_{|U^h}$;
by \cite[Lem 4.4, (2), (3)]{Saito-Purity-Rec} we have an isomorphism
\[F(U\setminus\{x_1,\ldots, x_r\})/\tilde{F}(U,D)\xr{\simeq} F(U^h\setminus \{x_1,\ldots, x_r\})/\tilde{F}(U^h, D^h).\]
Now the statement follows from $U^h= \sqcup_i \Spec \sO_{X, x_i}^h$, see \cite[XI, \S2, Prop 1, 1)]{Raynaud}, 
\eqref{lem:CRT-RS1}, and Lemma \ref{lem:approx-h}.
\end{proof}
\begin{lemma}\label{lem:locLS45}
Let $F\in \RSC_{\Nis}$ and $\pi:\Spec L'\to \Spec L$ be a finite extension of henselian dvf's. 
Denote by $\sigma: K\to \sO_L$ a $k$-morphism, such that the composition $L\to \sO_L/\fm_L$ is a finite field extension;
denote by $\sigma': K\to \sO_{L'}$ the induced map.
Then we have 
\begin{enumerate}[label=(\arabic*)]
\item\label{locLS4} $(\pi_*b, f)_{L,\sigma} = (b, \pi^*f)_{L',\sigma'}$, \, $b\in F(L')$, $f\in \sO_{L}^\times$;  
\item\label{locLS5} $(a, \pi_*g)_{L,\sigma}= (\pi^*a, g)_{L',\sigma'}$,\, $a\in F(L)$, $g\in \sO_{L'}^\times$.
\end{enumerate}
\end{lemma}
\begin{proof}
We can spread out the situation as follows: There exists a finite and surjective morphism 
$\ol{\pi}: C'\to C$ between regular and projective $K$-curves, with function fields $E'=K(C')$, $E=K(C)$,
points $x'\in C'$ and $x=\ol{\pi}(x')\in C$,
and elements $\tilde{a}\in F(E)$, $\tilde{b}\in F(E')$, $\tilde{f}\in \sO_{C,x}^\times$, $g\in \sO_{C',x'}^\times$
inducing $\pi$, $\sigma$, $\sigma'$, $a$, $b$, $f$, $g$, respectively.
We prove \ref{locLS4}: 
By Lemma \ref{lem:CRT-RS} we find an element $b_1\in F(E')$ with 
$\tilde{b}-b_1\in F(\sO_{C',x'})$,  and  $b_1\in F(\sO_{C',y})$, for all $y\in \ol{\pi}^{-1}(x)\setminus\{x'\}$.
Since $\ol{\pi}^*\tilde{f}\in \sO^\times_{C', y}$, for all $y/x$, we obtain
\eq{lem:locLS451}{(b, \pi^*f)_{L',\sigma'} \stackrel{\ref{LS1}, \ref{LS2}}{=} \sum_{y/x}(b_1, \ol{\pi}^*f)_{C'/K, y}
\stackrel{\ref{LS7}}{=} (\ol{\pi}_*(b_1), \tilde{f})_{C/K,x}.}
Note $E'\otimes_E L\cong \prod_{y/x} E'_y$, where $E'_y$ is the henselization of $E'$ at $y$. 
Thus in $F(L)$ we have 
\[\ol{\pi}_*b_1= \sum_{y/x} \pi_{y*} b_1,\]
where $\pi_{y}: \Spec L\to \Spec E'_y$ is the natural map; 
in particular $\pi_{x'}=\pi$. Hence in $F(L)$
\[\ol{\pi}_* b_1\equiv \pi_*b\text{ mod } F(\sO_L);\]
this together with \eqref{lem:locLS451} and \ref{LS1} implies formula \ref{locLS4}.

Now \ref{locLS5}:
By the Approximation Lemma we find $g_1\in {E'}^\times$ such that 
\[(\ol{\pi}^*\tilde{a}, g_1)_{C'/K,y}=0, \quad y\in \pi^{-1}(x)\setminus\{x'\},\]
 and 
\[(\ol{\pi}^*\tilde{a}, g_1)_{C'/K,x'}= (\pi^*a, g)_{L',\sigma'}.\]
Furthermore we have the following equality in $L^\times$ 
\[\Nm_{E'/E}(g_1)= \prod_{y/x} \Nm_{E'_y/L}(g_1).\]
If $g_1$ is close enough to 1 at the points $y\in \ol{\pi}^{-1}(x)\setminus\{x'\}$
 we have $\Nm_{E'_y/L}(g_1)\in U^{(N)}_L$ for $N>>0$. 
Thus we can choose $g_1$ with the additional property 
\[(\tilde{a}, \Nm_{E'/E}(g_1))_{C/K,x}= (a, \Nm_{L'/L}(g))_{L,\sigma}.\]
The formula \ref{locLS5} now follows from \ref{LS8} and the above.
\end{proof}

\part{Applications }
\section{Algebraic groups and the local symbol}
In this section $k$ is a perfect field and $G$ is a commutative algebraic $k$-group.
Note that as sheaves on $\Sm$ we have $G=G_{\red}$ and hence we can always identify
$G$ with the smooth commutative $k$-group $G_{\red}$.
We fix an algebraic closure $\bar{k}$ of $k$; note $\Spec \bar{k}\in \Cor^{\rm pro}$.

\begin{para}\label{para:alg}
Let $G$ be a commutative algebraic $k$-group. Then $G\in \RSC_\Nis$, 
by \cite[Cor 3.2.5]{KSY-RecII}.
Let $L\in\Phi_{\le 1}$ have  residue field $\kappa$. 
Let $\iota: \kappa \inj  \bar{k}$ be a $k$-embedding.
We denote by  $L^{sh}_\iota$ the strict henselization of $L$ with respect to $\iota$.
Note that $L^{sh}_\iota$ is a henselian dvf of geometric type over $\bar{k}$.
We write 
\eq{para:RoSe-cond2}{(-,-)_{L^{sh}_\iota}:  G(L^{sh}_\iota)\times {L^{sh}_\iota}^\times 
                         \to G(\bar{k})}
for the symbol $(-,-)_{L^{sh}_\iota,\sigma}$ from \ref{para:symbol-hdvf}
with $\sigma :\bar{k}\inj \sO_{L,\iota}^{sh}$  the unique
coefficient field; in this case this {\em is} the symbol defined by Rosenlicht-Serre, see \cite[III, \S 1]{SerreGACC}.
If we choose a different $k$-embedding $\iota': \kappa\inj \bar{k}$, 
then we find an automorphism $\tau:\bar{k}\to \bar{k}$
with $\tau\circ \iota=\iota'$ inducing a (unique) isomorphism of $\sO_L$-algebras 
$\tau: \sO_{L,\iota}^{sh}\xr{\simeq} \sO_{L,\iota'}^h$ and by \eqref{lem:symbol-well-def0}
\[\tau( (a, f)_{L^{sh}_\iota}) = (\tau(a), \tau(f))_{L^{sh}_{\iota'}}.\]
We will usually drop the $\iota$ from the notation and write $L^{sh}= L^{sh}_\iota$.
We define the Rosenlicht-Serre conductor of $a\in G(L)$ by 
\[\RoSe_L(a):=\begin{cases} 
0, & \text{if } a\in G(\sO_L),\\ 
\min\{n\ge 1\mid (a, U^{(n)}_{L^{sh}})_{L^{sh}}=0\}, & \text{else}.
\end{cases}\]
Note that it is independent of the choice of $\iota: \kappa\inj \bar{k}$.
\end{para}

\begin{thm}\label{thm:rose-cond}
Let $G$ be a commutative algebraic $k$-group. 
\begin{enumerate}[label=(\arabic*)]
\item\label{thm:rose-cond01} The Rosenlicht-Serre conductor
 $\RoSe=\{\RoSe_L\}_{\td(L/k)=1}$ is a semi-continuous conductor of level 1 on $G$ (in the sense of Definitions 
\ref{defn:cond} and \ref{defn:c6}).
\item\label{thm:rose-cond02} Let $c^{G}$ be the motivic conductor of $G$ (see Definition \ref{defn:mot-cond})
 and denote by $(c^{G})^{\le 1}$ its restriction to $\Phi_{\le 1}$. Then $\RoSe=(c^{G})^{\le 1}$.
\end{enumerate}
In particular, the motivic conductor extends the Rosenlicht-Serre conductor to henselian dvf's over $k$
with non-perfect residue field and we have $\tG= \tG_{\RoSe}$ (see \ref{para:tFc} and \ref{not:ExtendingRSC} for notation).
\end{thm}
\begin{proof}
The last statement follows from Corollary \ref{cor:extending-cond}.
For \ref{thm:rose-cond01} we check that $\RoSe$ satisfies the properties from Definition \ref{defn:cond}.
\ref{c1} and \ref{c2} are obvious.  
Let $L'/ L$ be a finite extension of henselian dvf's  with $\td(L/k)=1$ and $a\in G(L')$.
Let $\kappa\inj \kappa'$ be the induced map on the residue fields and fix an embedding $\kappa'\subset \bar{k}$.
Then ${L'}^{sh}$ is finite over $L^{sh}$ and $e({L'}^{sh}/L^{sh})=e(L'/L)$.
Thus \ref{c3} follows directly from Lemma \ref{lem:locLS45}\ref{locLS4}.
To check \ref{c4}  
first observe, if $a\in G(\A^1_X)$ is not in $G(X)$ (via pullback), then
we find a closed point $x\in X$ such that $a_{\A^1_x}$ is not in $G(x)$. 
(Since $G$ is a finite type $k$-scheme and $X$ is Jacobson.) Thus it  suffices to show the following:
\begin{claim*}
 Let $\kappa/k$ be a finite field extension and set $\kappa(t)_\infty=\Frac(\sO_{\P^1_\kappa,\infty}^h)$.
Assume $a\in G(\A^1_\kappa)$ has $\RoSe_{\kappa(t)_\infty}(a)\le 1$. Then $a\in G(\kappa)$.
\end{claim*}
Else $a\not\in G(\kappa)$. Then its pullback $a_{\bar{k}}\in G(\A^1_{\bar{k}})$ is not in $G(\bar{k})$
and we can thus find two points $x, y\in \A^1(\bar{k})=\bar{k}$ such that $a_{\bar{k}}(x)\neq a_{\bar{k}}(y)$.
Take $f= (t-x)/(t-y)\in \bar{k}(t)$. Then $f\in U^{(1)}_{\bar{k}(t)_\infty}$ and we obtain
\[0  = (a_{\bar{k}}, f)_{\bar{\kappa}(t)_\infty} = - a_{\bar{k}}(x) + a_{\bar{k}}(y),\]
where the first equality follows from $\RoSe_{\kappa(t)_\infty}(a)\le 1$ and the second from \ref{LS4} and \ref{LS2}.
This yields a contradiction and thereby proves the claim.
\ref{c5} follows from the fact that $G$ is a reciprocity sheaf and Corollary \ref{cor:localLS3}.
Finally \ref{c6} (semi-continuity for $n=1$).
Assume $C$ is a  smooth  $k$-curve, $x\in C$ a closed point and $a\in G(C\setminus\{x\})$
with $\RoSe_{L_x}(a_x)\le n$, where $L_x=\Frac(\sO_{C,x}^h)$ and $a_x\in G(L_x)$ denotes the pullback of $a$.
Let $\ol{C}$ be the smooth compactification of $C$ and let $C_\infty= (\ol{C}\setminus C)_\red$.
Choose $N$ such that $\RoSe_{L_y}(a_y)\le N$, for all $y\in |C_\infty|$.
Then $(\ol{C}, n\cdot \{x\}+ N\cdot C_\infty)$ is a compactification of 
$(C, n\cdot\{x\})$ and we claim 
\eq{thm:rose-cond3}{\RoSe_{\ol{C}}(a)\le (n\cdot \{x\}+ N\cdot C_\infty).}
Indeed, let $\Spec L\to C\setminus\{x\}$ be a henselian dvf point with $\td(L/k)=1$.
If $\Spec \sO_{L}$ maps to $C\setminus\{x\}$, then
$\RoSe_{L}(a_{L})=0$. 
Else we get a finite extension $L^{sh}/ L^{sh}_y$, for some $y \in \{x\}\cup |C_\infty|$, 
say of ramification index $e$.
Let $u\in U^{(n_y e)}_{{L}^{sh}}$, where $n_x=n$ and $n_y=N$, for $y\neq x$.
 By Lemma \ref{lem:locLS45}\ref{locLS5} we have 
\[(a_{L}, u)_{L^{sh}} =(a_{L_y}, {\rm Nm}_{{L}^{sh}/L^{sh}_y} (u))_{L^{sh}_y},\]
which vanishes by  ${\rm Nm}_{{L}^{sh}/L^{sh}_y} (u)\in U^{(n_y)}_{L^{sh}_y}$ and $\RoSe_{L_y}(a_y)\le n_y$.
This proves  the claim \eqref{thm:rose-cond3}, hence 
\ref{c6}, and finishes the proof of \ref{thm:rose-cond01}.

By  Corollary \ref{cor:mot-cond-mini} we have $c^{G,1}\le \RoSe$.
Thus for \ref{thm:rose-cond02} it suffices to show:
 If $a\in \tilde{G}(\sO_L,\fm_L^{-r})$, for some  $L\in\Phi_{\le 1}$ and $r\ge 1$, 
then $\RoSe_L(a)\le r$. This follows from Corollary \ref{cor:localLS3}. 
\end{proof}

\begin{remark}
An extension of $\RoSe$ to dvf's of higher transcendence degree over $k$ was also constructed in
\cite{Kato-Russell12} (char 0)  and  \cite{Kato-Russell} (char $p>0$). 
The construction essentially coincides with the extension  from Theorem \ref{thm:rose-cond}, 
but in {\em loc. cit.} the log version is considered, whereas here non-log one,
c.f. Theorem \ref{thm:filF-motivic} below.
\end{remark}

\section{Differential forms and irregularity of rank 1 connections}\label{sec:Conn1}
In this section we assume that the base field $k$ has {\em characteristic $0$}.
We fix a ring homomorphism $R\to k$ which induces the structure of an $R$-scheme on any $k$-scheme.
(Of main interest are $R= k$ or $\Z$.)

\subsection{K\"ahler differentials}
\begin{para}\label{para:dR}
Let $X\in \Sm$.
We denote by $\Omega_{X/R}^\bullet$ the de Rham complex on $X$ relative to $R$ and
by $d: \Omega_{X/R}^\bullet\to \Omega_{X/R}^{\bullet+1}$ the differential. 
We set $\Omega^\bullet_{X}:=\Omega^\bullet_{X/\Z}$.
We have an exact sequence 
\eq{para:dR1}{\Omega^1_{R}\otimes_R\Omega^{\bullet-1}_{X}\to \Omega^\bullet_X\to \Omega^\bullet_{X/R}\to 0.}
We denote the Nisnevich sheaf on $\Sm$ given by
$X\mapsto H^0(X, \Omega^q_{X/R})$ by $\Omega^q_{/R}$ and set $\Omega^q:=\Omega^q_{/\Z}$.
By \cite[Thm A.6.2]{KSY-Rec}  and \cite[Cor 3.2.5]{KSY-RecII} we have
$\Omega^q\in \RSC_{\Nis}$.  Since the action of finite correspondences
on $\Omega^\bullet$ is $\Omega^\bullet_k$-linear (a fortiori it is $\Omega^\bullet_R$-linear),
the morphism $\alpha^*: \Omega^q(Y)\to \Omega^q(X)$, $\alpha\in \Cor(X,Y)$, induces 
via \eqref{para:dR1} the structure of a Nisnevich sheaf with transfers on $\Omega^q_{/R}$ and we obtain
$\Omega^q_{/R}\in \RSC_{\Nis}$.
\end{para}

\begin{lemma}\label{lem:d-map-PST}
The differential $d: \Omega^q_{/R}\to \Omega^{q+1}_{/R}$ is a map in $\RSC_{\Nis}$.
\end{lemma}
\begin{proof}
We have to show, that if $\alpha\in\Cor(X, Y)$ is a finite correspondence, $X,Y\in \Sm$,
then $\alpha^* d=d \alpha^*$ as maps $\Omega^q_{/R}(Y)\to \Omega^q_{/R}(X)$.
Since the restriction $\Omega^q_{/R}(X)\to \Omega^q_{/R}(U)$ is injective for any dense open $U\subset X$
(by \cite[Thm 6]{KSY-Rec}),
it suffices to verify the equality after shrinking $X$ arbitrarily around its generic points. In particular we can assume, that 
$X$ is connected and $\alpha=Z\subset X\times Y$ is a prime correspondence which is finite \'etale over  $X$ 
(here we use ${\rm char}(k)=0$).
Denote by $f:Z\to X$ and $g:Z\to Y$ the maps induced by projection. Then $Z^*= f_*g^*$. The compatibility of $d$
with $g^*$ is clear. Hence it remains to show $f_*d=df_*$ for a finite \'etale map $f:Z\to X$ between smooth schemes.
In this case, we have $f_*\Omega^q_{Z/R}=f_*\sO_Z\otimes_{\sO_X}\Omega^q_{X/R}$ and 
$f_*=\Tr_{f}\otimes \id_{\Omega^q_{/R}}$,
by \cite[Prop 2.2.23]{CR11}. Thus the looked for compatibility is shown 
as in \cite[II, Proof of Prop (2.2), case 2]{HaDR}. 
\end{proof}

\begin{para}\label{para:cond-dR}
Let $L\in\Phi$, with local parameter $t\in\fm_L\subset \sO_L$.
We denote $\Omega^\bullet_{\sO_L/R}(\log)$ the dga of logarithmic differentials, i.e.,
the graded subalgebra of $\Omega^\bullet_{L/R}$ generated by $\Omega^\bullet_{\sO_L/R}$
and $\dlog t$. In particular, $\Omega^0_{\sO_L/R}(\log)=\sO_L$.
For $q\ge 0$ and $a\in \Omega^q_{L/R}$, we define
\[c_L^{\dR}(a):=\begin{cases}
0, & \text{if }a\in \Omega^q_{\sO_L/R},\\
\min\left\{n\ge 1\mid a\in \frac{1}{t^{n-1}}\cdot\Omega^q_{\sO_L/R}(\log)\right\}, & \text{else}.
\end{cases}\]
\end{para}

\begin{thm}\label{thm:cdr-mot}
For all $q\ge 0$, the collection $c^{\dR}=\{c^{\dR}_L\}$ defined in \ref{para:cond-dR} 
coincides with the motivic conductor,
i.e., (see Definition \ref{defn:mot-cond})
\[c^{\dR}= c^{\Omega^q_{/R}}.\]
Furthermore, the restriction $(c^{\dR})^{\le q+1}$ 
 is a semi-continuous conductor.
\end{thm}
\begin{proof}
We start by showing that $c^{\dR}$ is a semi-continuous conductor of level $q+1$.
Properties \ref{c1} and \ref{c2} of Definition \ref{defn:cond} are obvious.

\ref{c3}. Let $L'/L$ be a finite extension of henselian dvf with ramification index $e=e(L'/L)$, 
and denote by $f: \Spec L'\to \Spec L$ the induced map. Let $a\in \Omega^q_{L'/R}$. We have to show:
\eq{thm:cdr-mot0}{c_L^{\dR}(f_*a)\le \left\lceil\frac{c_{L'}^{\dR}(a)}{e}\right\rceil.}
We know that $f_*$ restricts to $\Omega^q_{\sO_{L'}/R}\to \Omega^q_{\sO_L/R}$ and by the well-known formula
$f_*\dlog=\dlog\circ \Nm_{L'/L}$ also to
\eq{thm:cdr-mot1}{f_*:\Omega^q_{\sO_{L'}/R}(\log)\to \Omega^q_{\sO_L/R}(\log).}
Moreover, $f_*$ is continuous with respect to the
$\fm_L$-adic topology (which on $\Omega^q_{L'/R}$ is the same as the $\fm_{L'}$-adic topology).
We may therefore replace $\Omega^q_{L'/R}$ and $\Omega^q_{L/R}$ by the  corresponding completed modules.
Furthermore, it suffices to treat the two cases separately 
in which $L'/L$ is either totally ramified or unramified.

{\em 1st case: $e=1$.} In this case a local parameter $t\in \sO_L$ is also a local parameter of $\sO_{L'}$ and 
hence \eqref{thm:cdr-mot0} follows directly from \eqref{thm:cdr-mot1} and the  $L$-linearity of $f_*$.

{\em 2nd case: $e>1$, $L$, $L'$ complete and $\sO_L/\fm_L=\sO_{L'}/\fm_{L'}$.}
Let $K\inj \sO_L$ be a coefficient field; it also defines a coefficient field of $\sO_{L'}$. 
Let $\tau\in \sO_{L'}$ and $t\in \sO_L$ be local parameters. Then we can identify $L'=K((\tau))$ and
$\tfrac{1}{\tau^{n-1}}\cdot \widehat{\Omega}^q_{\sO_{L'}/R}(\log)$ with the $\tau$-adic completion of
\[\tfrac{1}{\tau^{n-1}}\cdot\left((K[\tau]\otimes_K\Omega^q_{K/R})\oplus 
 (K[\tau]\dlog\tau\otimes_K\Omega^{q-1}_{K/R})\right).\]
Furthermore, observe that $\frac{1}{\tau^i}\dlog\tau= -\frac{1}{i} d(\frac{1}{\tau^i})$, $i\ge 1$.
 Since $f_*$ commutes with the differential
(by \ref{lem:d-map-PST}) we are reduced to show:
\eq{thm:cdr-mot2}{ f_*(\tfrac{1}{\tau^i})\in \tfrac{1}{t^{r-1}}\sO_L, \quad 
r:= \lceil \tfrac{n}{e}\rceil, \text{ for all }i\in [1,n-1]. }
We compute for $i\in[1,n-1]$
\begin{align*}
\fm_L^r\cdot df_*(\tfrac{1}{\tau^i}) & = \fm_L^r\cdot  f_*(-i\tau^{-i-1}d\tau)\\     
                                                       &\subset  f_*(\fm_{L'}^{er-i-1} d\tau)\\
                                                        &\subset f_*(\Omega^1_{\sO_{L'}/R})\subset \Omega^1_{\sO_L/R}.
\end{align*}
This implies \eqref{thm:cdr-mot2}, once we observe that in characteristic zero we have
$\fm_L^r \cdot da\in \Omega^1_{\sO_L/R}$ if and only if $\fm_L^{r-1}\cdot a\in \sO_L$, for any $a\in L=K((t))$.

\ref{c4} for $c^{\dR, q+1}$ follows directly from the following  facts, where $A$ is a finite type smooth $k$-algebra:
\begin{enumerate}[label=(\roman*)]
\item $\Omega^q_{A[t]/R}=
(k[t]\otimes_k\Omega^q_{A/R}) \oplus (\Omega^{q-1}_{A/R}\otimes_k \Omega^1_{k[t]/k})$;
\item\label{thm:cdr-mot:c4ii}
 for any non-zero $\alpha\in \Omega^{q}_{A/R}$ there exists a prime ideal $\fp\subset A$ with $\td(k(\fp)/k)=q$,
where $k(\fp)= A_\fp/\fp$, such that the image of $\alpha$ in $\Omega^{q}_{k(\fp)/R}$ is non-zero;
\item $H^0(\P^1_k, \Omega^1_{\P^1/k}(\log \infty))=0$, $H^0(\P^1, \sO_{\P^1})=k$.
\end{enumerate}
(Note,  \ref{thm:cdr-mot:c4ii} is easy for $R=k$ and follows in general from the natural map 
$\Omega^q_{/R}\to \Omega^q_{/k}$.)

For \ref{c5} it suffices to observe that if $a\in H^0(X, \Omega^q_{X/R} \otimes_{\sO_X}\sO_X(D))$,
for some proper modulus pair $(X,D)$, then $c^{\dR}_X(a)\le D$. 

Finally, \ref{c6}. Let $U=\Spec A$ be  smooth affine  and $Z\subset U$ a smooth divisor which we can assume to be
principal $Z=\Div(t)$. Let 
\[a= \frac{1}{t^{r-1}} a_1 + \frac{1}{t^{r-1}}a_2 \dlog t, \quad 
a_1\in\Omega^q_{A/R}, a_2\in \Omega^{q-1}_{A/R}, r\ge 1.\]
Let $(\ol{Y}, \ol{Z}+\Sigma)$ be a compactification of $(U,Z)$ with $\ol{Z}_{|U}=Z$ and $\ol{Y}$ normal.
Let $\ol{Y}=\cup V_i$ be an open covering
such that $V_i=\Spec B_i$, $\Sigma_{|V_i}=\Div (f_i)$, and $\ol{Z}_{|V_i}=\Div(\tau_i)$, with
$\tau_i, f_i\in B_i$. Note that $\Spec B_{i}[1/f_i]\subset U$ is open,  for all $i$.
Hence, in $B_{i}[1/f_i]$ we can write $t=\tau_i e_i$, with $e_i\in (B_i[1/f_i])^\times$. 
Let $E_i$ be the Cartier divisor on $V_i$ defined by $e_i$. We have $|E_i|\subset |\Sigma_{|V_i}|$.
By Lemma \ref{lem:Cart-div} below,  there exists $N_1>>0$, such that $v_L(E_i)\le N_1  v_L(\Sigma_{|V_i})$, for all 
$\Spec L\to U$ and all $i$.
Furthermore, there exists an $N_2\ge 0$ such that 
$f_i^{N_2} a_1\in \Omega^q_{B_i/R}$ and $f_i^{N_2} a_2\in \Omega^{q-1}_{B_i/R}$, for all $i$.
Choose $N\ge r\cdot N_1+N_2$. Let $\rho: \Spec L\to U$, $L\in \Phi$. 
Assume the closed point of $\Spec \sO_L$ maps into $|\ol{Z}+\Sigma|\cap V_i$, for some $i$.
Then 
\begin{align*}
c_L(\rho^*a) & \le (r-1)v_L(\ol{Z}) +(r-1)v_L(E_i)+ N_2 v_L(\Sigma)+1\\
                  & \le (r-1)v_L(\ol{Z}) +(r-1)N_1 v_L(\Sigma) +N_2v_L(\Sigma)+1\\
                  &\le v_L(r\cdot \ol{Z}+ N\cdot \Sigma).
\end{align*}
Hence 
$c^{\dR}_{\ol{Y}}(a)\le (r\cdot \ol{Z} + N\cdot \Sigma)$, which proves \ref{c6}. 

Thus $c^{\dR}$ is a semi-continuous conductor on $\Omega^q_{/R}$ 
and Theorem \ref{thm:rec-cond}\ref{thm:rec-cond3} yields
for $n\ge 1$
\[\fil{n}:=\tfrac{1}{t^{n-1}}\cdot\Omega^q_{\sO_L/R}(\log)\subset \widetilde{\Omega^q_{/R}}(\sO_L,\fm^{-n}_L),\]
for any  $L\in\Phi$ with local parameter $t\in \sO_L$.
It remains to show the other inclusion. By Corollary \ref{cor:localLS3}
it suffices to show the following:
Let $K\inj \sO_L$ be some coefficient field and extend it in the canonical way to $\sigma: K(x)\inj \sO_{L_x}$,
where $x$ is a variable and $L_x=\Frac(\sO_L[x]_{(t)}^h)$. Assume $a\in \fil{r+1}$.
Then the following implication holds
\eq{thm:cdr-mot3}{(a, 1-xt^{r})_{L_x,\sigma}=0 \,\Rightarrow\, a\in \fil{r},}
where the local symbol on the left hand side is the one from \ref{para:symbol-hdvf} for $\Omega^q_{/R}$.
Since the local symbol for $\Omega^q_{/R}$ is uniquely determined by \ref{LS1} - \ref{LS4}, we see that it is given by 
\[(a, 1-xt^{r})_{L_x,\sigma}= \Res_t(a\dlog (1-xt^{r})),\]
where we use the isomorphism $L_x=K(x)((t))$ defined by $\sigma$ to compute the residue symbol on the right.
To prove the implication \eqref{thm:cdr-mot3} it suffices to consider $a$ modulo $\fil{r}$;  we have 
\[a\equiv \frac{1}{t^r}\alpha + \beta\frac{dt}{t^{r+1}}\quad \text{mod } \fil{r},\]
for $\alpha\in \Omega^{q}_{K/R}$, $\beta\in \Omega^{q-1}_{K/R}$.
We compute in $\Omega^{q}_{K(x)/R}$
\[\Res_t(a\dlog (1-xt^{r}))=-rx\alpha+\beta dx.\]
This shows \eqref{thm:cdr-mot3} and completes the proof.
\end{proof}

\begin{lemma}\label{lem:Cart-div}
Let $X$ be a noetherian integral normal scheme, $E$, $F$ two Cartier divisors on $X$ and assume $F$ is effective.
If $|E|\subset |F|$, then there exists $N\ge 1$, such that for all maps $\Spec \sO\to X$ 
whose image is not contained in $|F|$,
with $\sO$ a DVR with valuation $v$, we have $v(E)\le N\cdot v(F)$.
\end{lemma}
\begin{proof}
The question is local on $X$; hence we can assume $E$ and $F$ are given by functions $e,f\in k(X)^\times$.
Let $\Div(e)$, $\Div(f)$ be the associated Weil divisors. Since $|E|\subset |F|$ and $F$ is effective we find 
$N\ge 1$, such that $\Div(e)\le N\cdot \Div(f)$, which by the normality of $X$ implies $f^N/e\in \Gamma(X,\sO_X)$.
This yields the statement. 
\end{proof}

\begin{remark}\label{rmk:cdR-level}
The proof of Theorem \ref{thm:cdr-mot} also shows that 
\[c^{\dR'}_L(a)=\begin{cases} 
0, &\text{if } a\in \Omega^q_{\sO_L/R},\\
\min\{n\ge 2\mid a\in \frac{1}{t^{n-1}}\cdot \Omega^q_{\sO_L/R}\}, & \text{else},
\end{cases}\]
defines a semi-continuous conductor  on $\Omega^q$, but it coincides with the motivic one, only for $q=0$.
\end{remark}

\begin{cor}\label{cor:cZdR}
Set $Z\Omega^q_{/R}=\Ker(d: \Omega^q_{/R}\to \Omega^{q+1}_{/R})$. Then $Z\Omega^q_{/R}\in \RSC_{\Nis}$
and its motivic conductor $c^{Z\Omega^q_{/R}}= (c^{\Omega^q_{/R}})_{|Z\Omega^q_{/R}}$ restricts to
a conductor of level $q$.
\end{cor}
\begin{proof}
The formula for $c^{Z\Omega^q_{/R}}$ follows from Proposition \ref{prop:mot-cond-sub-sheaf}.
It remains to show that it has level $q$.
Let $a\in Z\Omega^q_{/R}(\A^1_X)$ with $c^{\dR}_{k(x)(t)_\infty}(a)\le 1$, for all 
points $x\in X$ with $\td(k(x)/k)\le q-1$. 
This implies 
$a\in H^0(X, k[t]\otimes_k\Omega^q_{X/R})\cap Z\Omega^q_{/R}(\A^1_X)$,
cf. the proof of \ref{c4} in Theorem \ref{thm:cdr-mot}.
Hence $a\in Z\Omega^q_{/R}(X)$. This shows that $(c^{Z\Omega^q_{/R}})^{\le q}$ satisfies \ref{c4}. 
\end{proof}

\begin{cor}\label{cor:oCi-dr-global}
\begin{enumerate}[label= (\arabic*)]
\item\label{cor:oCi-dr-global1} Let $\sX=(X,D)\in\MCor$ be a proper modulus pair. Then 
\[\widetilde{\Omega^q_{/R}}(\sX)= H^0(X_1, \Omega^q_{X_1/R}(\log D_1)(D_1-D_{1,\red})),\]
where $\pi: X_1\to X$ is any resolution of singularities which is an isomorphism over $X\setminus D$ and 
such that $D_1:=\pi^*D$ is supported on a simple normal crossings divisor. 
(See \ref{not:ExtendingRSC}, for the notation $\widetilde{\Omega^q_{/R}}$.)
\item\label{cor:oCi-dr-global2} Let $h^0_{\A^1}(\Omega^q_{/R})$ be the maximal $\A^1$-invariant subsheaf of 
$\Omega^q_{/R}$.
Then for $X\in \Sm$
\[h^0_{\A^1}(\Omega^q_{/R})(X)= H^0(\ol{X}, \Omega^q_{\ol{X}/R}(\log D)),\]
where $\ol{X}$ is any smooth compactification  of $X$ with simple normal crossing divisor $D$ at infinity.
\end{enumerate}
\end{cor}
\begin{proof}
First note, that $\widetilde{\Omega^q_{/R}}(\sX)=\widetilde{\Omega^q_{/R}}(X_1, \pi^*D)$,  
where $\pi: X_1\to X$ is any blow-up with center in $D$, since
$(X,D)\cong (X_1, \pi^*D)$ in $\MCor$.
Let $\sX=(X,D)$ be a proper modulus pair  with $D_{\red}$ 
a simple normal crossings divisor. Write $D=\sum_i r_i \cdot \ol{\eta}_i$, with $\eta_i\in X^{(1)}$ and set
$L_{\eta_i}:=\Frac(\sO_{X,\eta_i}^h)$. Then  it is direct  to check
that we have $c_L^{\dR}(\rho^*a)\le v_L(D)$, for all henselian dvf points $\rho:\Spec L\to \sX$ if and only if
$c^{\dR}_{L_{\eta_i}}(a)\le r_i$, for all $i$.
Thus the corollary follows from Theorem \ref{thm:cdr-mot}, Theorem \ref{thm:rec-cond}\ref{thm:rec-cond4},
and Corollary \ref{cor:HIc}.
\end{proof}

\subsection{Rank 1 connections and irregularity}

\begin{lemma}\label{lem:dlog-PST}
The homomorphism $\dlog: \sO_X^\times\to \Omega^1_{X/R}$, $X\in \Sm$, induces a
morphism $\dlog: \sO^\times\to \Omega^1_{/R}$ in $\RSC_\Nis$
\end{lemma}
\begin{proof}
The proof is similar to the one of Lemma \ref{lem:d-map-PST}, except that we have to 
replace the formula $f_* d=df_*$ by $f_*\dlog= \dlog \Nm_{Z/X}$, where
$f:Z\to X$ is a finite \'etale map between smooth schemes.
\end{proof}

\begin{para}\label{para:conn1}
Denote by ${\rm Conn}^1(X)$ the group of isomorphism classes of rank 1 connections on $X\in\Sm$, and 
by ${\rm Conn}^1_{\rm int}(X)$ the subgroup of integrable connections.
We have canonical group isomorphisms
\[{\rm Conn}^1(X)\cong H^1(X_{\Zar}, \sO^\times_X\xr{\dlog}\Omega^1_{X/k})\cong
 H^0(X, (\Omega^1_{/k}/\dlog\sO_X^\times)_\Nis)\]
and 
\mlnl{{\rm Conn}^1_{\rm int}(X)\cong H^1(X_{\Zar}, \sO^\times_X\xr{\dlog}Z\Omega^1_{X/k})\\
\cong H^0(X, (Z\Omega^1_{/k}/\dlog\sO^\times)_{\Nis}).}
Indeed, the first isomorphism is well-known (use that the first Zariski cohomology 
can be computed as \v{C}ech cohomology);
we show the  second as follows: Let $\bar{k}^{X}$ 
be the algebraic closure of $k$ in $k(X)$; we consider it as a constant sheaf on $X$.
We obtain the isomorphism  
 \[ [\sO_X^\times/(\bar{k}^{X})^\times\xr{\dlog}\Omega^1_{X/k}] \cong 
 (\Omega^1_{/k}/\dlog\sO_X^\times)_{\Zar}[-1], \]
 in the derived category of abelian sheaves on $X_{\Zar}$; similar with $Z\Omega^1_{/k}$.
Observe that $\Omega^1_{/k}$ and $\sO^\times$ are already Nisnevich sheaves, hence
\[(\Omega^1_{/k}/\dlog\sO_X^\times)_{\Zar}=(\Omega^1_{/k}/\dlog\sO_X^\times)_{\Nis}.\]
Since $H^i(X_{\Zar}, \bar{k}^X)=0$ for all $i\ge 1$, we obtain 
\[H^1(X_{\Zar}, \sO_X^\times\to \Omega^1_{X/k})
= H^1(X_{\Zar}, \sO_X^\times/(\bar{k}^X)^\times\to \Omega^1_{X/k}).\]
Similar with $Z\Omega^1_{/k}$. This yields the second isomorphisms.

By Lemma \ref{lem:dlog-PST} and \cite[Thm 0.1]{Saito-Purity-Rec} we obtain
\[{\rm Conn}^1,\, {\rm Conn}^1_{\rm int}\in \RSC_\Nis.\]
For $E\in {\rm Conn}^1(X)$ we denote by $\omega_E\in H^0(X, (\Omega^1_{/k}/\dlog\sO^\times)_{\Nis})$, 
the element corresponding to $E$ under the above isomorphism.
 
Let $L\in\Phi$ and let $t\in \sO_L$ be a local parameter.
Recall (e.g. from \cite[Def. 1.12]{Kato94}) that the irregularity of 
$E\in {\rm Conn}^1(\Spec L)\cong \Omega^1_{L/k}/\dlog L^\times$ is defined by
\[\irr_L(E)=\min\big\{n\ge 0\mid 
\omega_E\in \Im\big(\tfrac{1}{t^n}\cdot \Omega^1_{\sO_L/k}(\log)\to 
\Omega^1_{L/k}/\dlog L^\times\big)\big\}.\]
\end{para}

\begin{thm}\label{thm:irr-mot}
Notations are as in \ref{para:conn1}.
The motivic conductor of $E\in{\rm Conn}^1(L)$ is given by 
\[c^{{\rm Conn}^1}_L(E)=\begin{cases} 0, &\text{if $E$ extends to an $\sO_L$-connection,}\\
                                               \irr_L(E)+1,& \text{else.}  \end{cases}\]
Moreover, on ${\rm Conn}^1$ the motivic conductor restricts to a level 2 conductor and
on ${\rm Conn}^1_{\rm int}$ it restricts to a level 1 conductor.
\end{thm}
\begin{proof}
Set $\sH^1:= (\Omega^1_{/k}/\dlog \sO^\times)_{\Nis}$, 
$\sH^1_{\rm int}:=  (Z\Omega^1_{/k}/\dlog \sO^\times)_{\Nis}$.
For $a\in \sH^1(L)$ we define
\[c^{\irr}_L(a):=\min\{c^{\dR}_L(\tilde{a})\mid \tilde{a}\in \Omega^1_{L/k} \text{ lift  of }a\},\]
see \ref{para:cond-dR} for the definition of $c^{\dR}$.  It suffices to prove the following identity 
for the motivic conductor of $\sH^1$
\eq{thm:irr-mot0}{c^{\sH^1}=c^{\irr},}
and that $(c^{\irr})^{\le 2}$ and $(c^{\irr})_{|\sH^1_{\rm int}}^{\le 1}$ satisfy \ref{c4}.
It follows directly form Theorem \ref{thm:cdr-mot} and Lemma \ref{lem:im-cond}, 
that $c^{\irr}$ satisfies \ref{c1}-\ref{c6} except maybe \ref{c4} and \ref{c5}.
For \ref{c5}, note that given $X\in \Sm$ we find by resolution of singularities a compactification
$\sX=(\ol{X}, X_\infty)$ with $\ol{X}\in\Sm$. In particular,
for all $x\in \ol{X}$ the local ring $\sO_{\ol{X},x}^h$ is factorial and hence so is any of its localizations.
Therefore, it follows from the exact sequence 
\[H^0(Y,\Omega^1_{Y/k})\to \sH^1(Y)\to \Pic(Y),\]
for any integral scheme $Y$ over $k$, that the condition \eqref{lem:im-cond1} from Lemma \ref{lem:im-cond} is satisfied;
hence $c^{\irr}$ satisfies \ref{c5}.
Next \ref{c4}. Take $a\in \sH^1(\A^1_X)$ with 
\eq{thm:irr-mot1}{c^{\irr}_{k(x)(t)_\infty}(a_x)\le 1,\quad \text{for all } x\in X \text{ with } \td(k(x)/k)\le 1,}
where $a_x$ is the restriction of $a$ to $k(x)(t)_\infty$. Consider the exact sequence (using the $\A^1$-invariance
of $X\mapsto H^i(X, \sO_X^\times)$)
\eq{thm:irr-mot2}{H^0(X, \sO_X^\times)\xr{\dlog} H^0(\A^1_X, \Omega^1_{\A^1_X/k})\to \sH^1(\A^1_X)\to 
H^1(X,\sO_X^\times).}
Let $\pi:\A^1_X\to X$ be the projection and $i:X\inj \A^1_X$  a section.
By \eqref{thm:irr-mot2} there exists an $\tilde{a}\in H^0(\A^1_X, \Omega^1_{\A^1_X/k})$ mapping to $a-\pi^*i^*a$
 and any such lift
satisfies \eqref{thm:irr-mot1} with $c^{\irr}$ replaced by $c^{\Omega^1_{/k}}$. 
Thus $\tilde{a}\in H^0(X,\Omega^1_{X/k})$,
by \ref{c4} for $(c^{\Omega^1_{/k}})^{\le 2}$;
hence $(c^{\irr})^{\le 2}$ satisfies \ref{c4}. Similarly, one proves \ref{c4} and \ref{c5} 
for $(c^{\irr})_{|\sH^1_{\rm int}}^{\le 1}$.

Hence $c^{\irr}$ is a semi-continuous conductor and we obtain $c^{\sH^1}\le c^{\irr}$. We show the other inequality.
Let $L\in\Phi$ and let $\sigma: K\inj \sO_L$ be a coefficient field. 
Denote by $\ol{\fil{n}}\subset \sH^1(L)$ the image of $\fil{n}=\frac{1}{t^{n-1}}\Omega^1_{\sO_L/k}(\log)$. 
Take $a\in \ol{\fil{r+1}}$. 
Similar as in the proof of Theorem \ref{thm:cdr-mot}
(around \eqref{thm:cdr-mot3}, and with the notation from there) it suffices to show the implication
\eq{thm:irr-mot3}{(a, 1-xt^r)_{L_x,\sigma}=0 \text{ in } \sH^1(K(x))\,\Rightarrow \, a\in \ol{\fil{r}}.}
Let $\tilde{a}\in \fil{r+1}$ be a lift of $a$; write
\[\tilde{a}= \frac{1}{t^r}\alpha +\beta \frac{dt}{t^{r+1}} \quad \text{mod }\fil{r} \]
with $\alpha\in \Omega^1_{K/k}$ and $\beta\in K$. 
Then the left hand side of \eqref{thm:irr-mot3} is equivalent to 
\[\Res_t(\tilde{a}\dlog(1-xt^r))= \dlog f, \quad \text{for some }f\in K(x)^\times.\]
Computing the residue symbol yields
\eq{thm:irr-mot4}{-r x\alpha +\beta dx =\dlog f \quad\text{in }\Omega^1_{K(x)/k}.}
We claim this can only happen if $\alpha=\beta=0$. Indeed, 
first observe that if $h\in K((x))^\times$ is a formal Laurent series such that there exists a form 
$\gamma\in \Omega^1_{K/k}$
with $\dlog(h)= x\cdot \gamma\quad \text{in }\widehat{\Omega}^1_{K((x))/k}$,
then $\gamma=0=\dlog(h)$. Thus \eqref{thm:irr-mot4} implies that $\dlog(f\cdot \exp(-\beta x))=0$ in 
$\widehat{\Omega}^1_{K((x))/k}$. Hence there exists an element $\lambda\in k_1$ the algebraic closure of $k$ in $K$ such that
\[\lambda\cdot \exp(\beta x)= f\in K(x)^\times,\]
which is only possible if $\beta=0$; it follows $\alpha=0$. Thus $a\in \ol{\fil{r}}$, which proves \eqref{thm:irr-mot3} 
and completes the proof.
\end{proof}

\begin{cor}\label{cor:h0-reg}
Let $X\in \Sm$.
Then $h^0_{\A^1}({\rm Conn}^1_{\rm int})(X)$ is the group of isomorphism classes of regular singular rank 1 connections on $X$
(see \ref{para:HIsub} for notation).
\end{cor}
\begin{proof}
Let $E\in {\rm Conn}^1_{\rm int}(X)$. Then by definition (see \cite[II, Def 4.5]{Deligne-ED}) $E$ is regular singular 
if and only if $\irr(\rho^*E)=0$, for all henselian dvf points $\rho:\Spec L\to X$ with $\td(L/k)=1$.
By Theorem \ref{thm:irr-mot} and Corollary \ref{cor:extending-cond}, this is equivalent 
$c_L^{{\rm Conn}^1_{\rm int}}(\rho^*a)\le 1$, for all $L$. Thus the statement follows from Corollary \ref{cor:HIc}.
\end{proof}

\section{Witt vectors of finite length}\label{subsec:Witt}
In this section we assume that $k$ is a perfect field of characteristic $p>0$.
Denote by $W_n$ the ring scheme of $p$-typical Witt vectors of length $n$. 
We will denote by $W_n\sO_X$ the (Zariski-, Nisnevich-, \'etale-) sheaf on $X$ defined by $W_n$.
The restriction of $W_n$ to $k$-schemes, which - by abuse of notation - we will again denote by $W_n$,
is in particular a smooth commutative group over $k$. Hence $W_n\in \RSC_{\Nis}$ (see \ref{para:alg}).

\begin{para}\label{para:Witt-trace}
Let $A$ be a ring. Recall, that there is an isomorphism of groups
\eq{para:Witt-trace1}{W_n(A)\xr{\simeq} 
(1+TA[[T]])^\times/ \{\prod_{s\not\in \{1,p, \ldots, p^{n-1}\}} (1-b_sT^s)\mid b_s\in A\},}
\[ (a_0,\ldots, a_{n-1})\mapsto \prod_{i=0}^{n-1} (1-a_i T^{p^i}).\]
Assume $A$ is normal and we have an inclusion of rings $A\inj B$ making $B$ a finite $A$-module.
Then $B[[T]]$ is finite over the normal ring $A[[T]]$ and hence the norm map,
$\Nm: B[[T]]^\times\to A[[T]]^\times$ induces a trace
$\Tr: W_n(B) \to W_n(A)$, see e.g. \cite[Prop A.9]{Ru}. 

Now assume $f: Y\to X$ is a finite and surjective $k$-morphism, where  $X$ is a normal $k$-scheme.
Then the local traces above glue to give 
\[\Tr_f: W_n(Y)\to W_n(X).\]
\end{para}
\begin{lemma}\label{lem:Witt-trace}
In the situation above, $\Tr_f$ equals $f_*:W_n(Y)\to W_n(X)$, the map used to define the transfer structure on
the group scheme $W_n$.
\end{lemma}
\begin{proof}
Let $a\in W_n(Y)$ and $d=\deg(f)$. Recall the element $f_*(a)$ is defined by the composition 
\[X\to \Sym^d_X Y\xr{\sum^d a} W_n.\]
It suffices to check that $\Tr_f(a)$ and $f_*(a)$ coincide on a dense open subset.
Thus we can assume that $X$ is affine integral and $f: Y\to X$ is finite free.
Furthermore $W_n$ is a direct factor of the scheme of big Witt vectors $\mathbb{W}_{p^n}$
and $\Tr$ and $f_*$ extend to the big Witt vectors.
Thus it suffices to show the equality on the big Witt vectors $\mathbb{W}_r$, for $r\ge 1$.
Let $S_r=\Spec k[t]/(t^{r+1})$ and denote by $\epsilon: S= \Spec k\inj S_{r}$ the $S$-section.
We have the following isomorphism of $S$-group schemes (cf. \eqref{para:Witt-trace1})
\[\mathbb{W}_r\cong \Ker(\Res_{S_r/S}(\G_m)\xr{\epsilon^*} \G_m),\]
where $\Res_{S_r/S}(\G_m)$ denotes the Weil restriction.
Denote by $f_r: Y_r\to X_r$ the base change of $f$ along $S_r\to S$.
Let $b\in \mathbb{W}_r(Y)$ which we can view as an element in $\Res_{S_r/S}(\G_m)(Y)$.
Then the image of $f_*(b)$ in $\mathbb{W}_r(X)\subset \Res_{S_r/S}(\G_m)(X)$
is equal to the $S_r$-morphism
\[f_{r*}(b) : X_r\to \Sym^d_{X_r} Y_r=X_r\times_X \Sym^d_X (Y)\xr{\prod^d b} \G_{m, S_r}. \]
Now the statement follows from the fact that $f_*=\Nm$ on $\G_m$, see \cite[Exp. XVII, Ex 6.3.18 ]{SGA4III}.
\end{proof}

\begin{para}\label{para:BMK}
Let $L\in\Phi$.  Denote by  $\fillog{j} W_n(L)$, $j\ge 0$, 
the Brylinski-Kato filtration (see \cite{Brylinski}, \cite{Kato-Swan}), i.e.,
\begin{align*}
\fillog{j} W_n(L) &=\{a\in W_n(L)\mid [x]\cdot F^{n-1}(a)\in W_n(\sO_{L}),\text{ all }x\in \fm_{L}^j\}\\
                      &= \{(a_0,\ldots, a_{n-1})\in W_n(L)\mid p^{n-1-i} v(a_i)\ge -j, \, \text{all } i\},
\end{align*}
where $[x]$ denotes the Teichm\"uller lift of $x$ and $F: W_n(L)\to W_n(L)$ is the Frobenius, which by 
contravariant functoriality is induced by the Frobenius of $L$ (or by covariant functoriality by
the base change over $\Spec k$ of the Frobenius on the $\Spec(\F_p)$-ring scheme $W_n$).
 We observe that  for $s\ge 0$ we have
\eq{para:BKM0}{\wV^s(\fillog{j}W_n(L))\subset \fillog{j}W_{n+s}(L),}
where $\wV$ is the Verschiebung on the Witt vectors.
The non-log version introduced by Matsuda in \cite[3.1]{Matsuda}, 
is given by (with the conventions from \cite[2.1]{KeS})
\[\fil{j}W_n(L)= \fillog{j-1} W_n(L)+ \wV^{n-r}(\fillog{j}W_r(L)),\quad j\ge 1,\]
where $r=\min\{n, {\rm ord}_p(j)\}$.
(This is equal to Matsuda's ${\rm fil}'_{j-1}W_n(L)$.)
Assume $r={\rm ord}_p(j)<n$, then  $(a_0,\ldots, a_{n-1})\in \fil{j}W_n(L)$ 
\[\Longleftrightarrow 
p^{n-1-i}v(a_i) \begin{cases}  
\ge -j,  &\text{if } i\neq n-1-r,\\
> -j, & \text{else.} 
\end{cases}\]
This is the description given in \cite[4.7]{Kato-Russell}.
(They denote by ${}^\flat\fil{j}W_n(L)$ what we call $\fil{j}W_n(L)$.)
One directly checks that 
\eq{para:BKM2}{\wF^{n-1}d(\fil{j}W_n(L))\subset \fm_L^{-j}\cdot \Omega^1_{\sO_L},}
where $\wF^{n-1}d$ is the map
\eq{para:BKM3}{\wF^{n-1}d : W_n(L)\to \Omega^1_L, 
\quad (a_0,\ldots, a_{n-1})\mapsto \sum_{i=0}^{n-1} a_i^{p^{n-1-i}-1}d a_i.}
\end{para}

\begin{para}\label{para:KR-cond}
Let $L\in\Phi$. The $F$-saturation of $\fillog{j}W_n(L)$ and $\fil{j}W_n(L)$
is introduced in \cite{Kato-Russell}:
\eq{para:KR-cond1}{{\rm fil}^{{\rm log}, F}_{j}W_n(L)=
\sum_{r\ge 0} F^r (\fillog{j}W_n(L)), \quad j\ge 0,}
and 
\eq{para:KR-cond2}{{\rm fil}^F_jW_n(L)=\sum_{r\ge 0} F^r(\fil{j}W_n(L)), \quad j\ge 1.}
Let $\kappa$ be the residue field of $\sO_L$. Denote by 
$\kappa[F]$ the non-commutative polynomial ring in the variable $F$ and with coefficients in $\kappa$  
with relation $F a= a^p F $ in $\kappa[F]$, for $a\in \kappa$.
By \cite[4.7]{Kato-Russell}, there is an injective homomorphism for $j\ge 1$
\eq{para:KR-cond3}{\bar{\theta}_j: \frac{{\rm fil}^{{\rm log},F}_jW_n(L)}{ {\rm fil}^{F}_jW_n(L)}\inj
 \kappa[F]\otimes_\kappa \left(\frac{\Omega^1_{\sO_L}(\log)\otimes_{\sO_L} \fm_L^{-j}/\fm_L^{-j+1}}
 {\Omega^1_{\sO_L}\otimes_{\sO_L} \fm_L^{-j}/\fm_L^{-j+1}}\right)}
 induced by (cf. \ref{para:BKM2})
\[\sum_{r\ge 0} F^r(a_r)\mapsto \sum_{r\ge 0}( F^r\otimes \wF^{n-1}da_r). \]

For $a\in W_n(L)$, we define the Brylinski-Kato-Matsuda-Russell conductor $\gamma_{n,L}(a)$ 
(cf. \cite[Thm 8.7]{Kato-Russell}) by 
\[\gamma_{n,L}(a):=
\begin{cases} 
0, & \text{if } a\in W_n(\sO_L),\\ \min\{j\ge 1\mid a\in {\rm fil}^{F}_jW_n(L)\}, & \text{else.}
\end{cases}
\]
Note that ${\rm fil}^{F}_1W_n(L)=W_n(\sO_L)$. Thus $\gamma_{n,L}(a)=0$ or $\ge 2$.
\end{para}

\begin{proposition}\label{prop:cond-Witt}
The collection 
\[\gamma_n=\{\gamma_{n,L}: W_n(L)\to \N_0\mid L \in\Phi\}\]
 is a semi-continuous conductor on $W_n$, as is its restriction $\gamma_n^{\le 1}$.
\end{proposition}
\begin{proof} Set $\gamma:=\gamma_n$.
Conditions \ref{c1} and \ref{c2} of Definition \ref{defn:cond} are clear.
\ref{c3}. Let $L'/ L$ be a finite extension  of henselian dvf's.
Let $e=e(L'/L)$ be the ramification index. Let $a\in W_n(L')$ and set $r:=\gamma_{L'}(a)$.
We have to show 
\eq{prop:cond-Witt2}{\Tr(a)\in {\rm fil}^{F}_sW_n(L), \quad \text{with } s:=\left\lceil\frac{r}{e}\right\rceil,}
where $\Tr=\Tr_{L'/L}$, see Lemma \ref{lem:Witt-trace}.
This is immediate if $r=0$. Thus we can assume $r\ge 2$ and
write $a=\sum_{j\ge 0} F^j(a_j)$, with $a_j\in \fil{r}W_n(L')$.
We have $\Tr(a_j)\in \fillog{s}W_n(L)$. Indeed, this follows from 
\begin{align*}
[\fm_L^s]\cdot F^{n-1}(\Tr(a_j)) &\subset  \Tr([\fm_{L'}^{se}]\cdot F^{n-1}(a_j)) \\
                                                 &\subset \Tr([\fm_{L'}^r]\cdot F^{n-1}(a_j))\\ 
                                                 &\subset \Tr(W_n(\sO_{L'}))\subset W_n(\sO_L),
\end{align*}
where for $b\in W_n(L)$ we denote $[\fm^j_L]\cdot b:=\{[x]\cdot b \mid x\in \fm^j_L\}$.
Hence
\[\Tr(a)=\sum_j F^j(\Tr(a_j))\in {\rm fil}^{\log, F}_sW_n(L).\]
By the injectivity of $\bar{\theta}_s$ in \eqref{para:KR-cond3} it suffices  to show
\eq{prop:cond-Witt3}{\fm^s_L\cdot F^{n-1}d\Tr(a_j)\in \Omega^1_{\sO_L},\quad \text{all } j\ge 0.}
By \cite[Thm 2.6]{Ru} the trace $\Tr$ extends to a trace between the de Rham-Witt complexes 
$\Tr: W_n\Omega^\cdot_{L'}\to W_n\Omega^\cdot_{L}$ which is compatible with the differential and Frobenius, is
$W_n\Omega^\cdot_L$-linear, and  equals the classical trace on K\"ahler differentials for $n=1$.
We obtain
\begin{align*}
\fm_L^s\cdot F^{n-1}d\Tr(a)  & = \fm^{s}_L\cdot \Tr(F^{n-1}d a) \\
                             & \subset \fm^{s}_L\cdot \Tr(\fm_{L'}^{-r}\cdot \Omega^1_{\sO_{L'}}),& 
                                                                                    &a\in \fil{r}W_n(L'), \text{ see } \eqref{para:BKM2}\\
                                        &\subset \Tr(\fm_{L'}^{es-r}\cdot \Omega^1_{\sO_{L'}})\\
                                       &\subset \Tr(\Omega^1_{\sO_{L'}})\subset \Omega^1_{\sO_L}.
\end{align*}
This completes the proof of \ref{c3}.

Next we show that the restriction of $\gamma$ to $\Phi_{\le 1}$ satisfies \ref{c4}. 
Let $X\in \Sm$ and $a\in W_n(\A^1_X)$ with 
\eq{thm:cond-Witt1}{\gamma_{k(x)(t)_\infty}(\rho_x^*a) \le 1,}
for  closed points $x\in X$, where $k(x)(t)_\infty=\Frac(\sO_{\P^1_x,\infty}^h)$.
 We have to show $a\in W_n(X)$.
We may assume $X=\Spec A$, and thus $a\in W_n(A[t])$.
If $a$ is not constant, then 
we find a closed point $x\in X$ such that the image of $a$ in $W_n(k(x)[t])$ is not constant.
Hence $a_{k(x)(t)_\infty}\not\in W_n(\sO_{k(x)(t)_\infty})$, i.e.,
$\gamma(a_{k(x)(t)_\infty})\ge 2$, contradicting our assumption \eqref{thm:cond-Witt1}.

\ref{c5}. Let $X\in\Sm$ and  $a\in W_n(X)=H^0(X, W_n\sO_X)$.
Let $\sX=(\ol{X}, X_\infty)$ be a proper modulus pair with $X=\ol{X}\setminus |X_\infty|$.
For an effective Cartier divisor $E$ on $\ol{X}$ denote by $W_n\sO_{\ol{X}}(E)$ the 
invertible subsheaf of $j_*W_n\sO_{\ol{X}\setminus|E|}$ corresponding to
the image of $[\sO_{\ol{X}}(E)]\in H^1_{\et}(\ol{X}, \sO_{\ol{X}}^\times)$ in 
$H^1_{\et}(\ol{X}, W_n\sO_{\ol{X}}^\times)$ under the map  induced by the Teichm\"uller lift.
If $e$ is an equation for $E$ at $x\in\ol{X}$, then $W_n\sO_{\ol{X},x}(E)=W_n\sO_{\ol{X},x}\cdot \frac{1}{[e]}$.
There exists an integer $N$ such that $a\in H^0(\ol{X}, W_n\sO_{\ol{X}}(N\cdot X_\infty))$. 
\begin{claim}\label{thm:cond-Witt:claim1}
 $(\ol{X}, r X_\infty)$ satisfies \ref{c5} for any $r>p^{n-1}N$.
\end{claim}
Indeed, let $\rho: \Spec L\to \sX$ be a henselian dvf point.
Assume that the closed point $s\in S_L$ maps into $X_\infty$ and let 
$f\in \sO_{\ol{X}, \ol{\rho}(s)}$ be a local equation for $X_\infty$.
Let $m=v_L(f)$. For $r> p^{n-1}N$ we find $[\fm_L^{rm-1}]\cdot F^{n-1}(a)\in W_n(\sO_L)$; 
hence (see \ref{para:BMK})
\[a\in\fillog{rm-1}W_n(L)\subset \fil{rm}W_n(L)\subset {\rm fil}^{F}_{rm}W_n(L),\]
i.e., $\gamma_L(\rho^*a)\le rm =v_L(r\cdot X_\infty)$,  proving Claim \ref{thm:cond-Witt:claim1}.
 
Finally, \ref{c6}. Let $X\in\Sm$ and $Z\subset X$ a smooth prime divisor with generic point $z$. 
Set $K=\Frac(\sO_{X,z}^h)$.
Let $a\in W_n(X\setminus Z)$. Assume $a_K\in {\rm fil}_{j}^{F}W_n(K)$, $j\ge 2$.
Then there exists an affine Nisnevich neighborhood $U=\Spec A\to X$ of $z$ such that $Z_U= \divi(t)$ on $U$
and $a_U=\sum_{s\ge 0} F^s(a_s+ \wV^{n-r}(b_s))$, where 
$r=\min\{{\rm ord}_p(j), n\}$ and  $a_s\in W_n(A[1/t])$, $b_s\in W_r(A[1/t])$ with 
\eq{prop:cond-Witt4}{[t]^{j-1}\cdot F^{n-1}(a_s)\in W_n(A), \quad [t]^j\cdot F^{r-1}(b_s)\in W_r(A).}
Let $(\ol{Y}, \ol{Z}+\Sigma)$ be a compactification of $(U,Z)$ with $\ol{Z}_{|U}=Z$ and $\ol{Y}$ normal.
Let $\ol{Y}=\cup V_i$ be an open covering
such that $V_i=\Spec B_i$, $\Sigma_{|V_i}=\Div (f_i)$, and $\ol{Z}_{|V_i}=\Div(\tau_i)$, with
$\tau_i, f_i\in B_i$. Note that $\Spec B_{i}[1/f_i]\subset U$ is open,  for all $i$.
Hence, in $B_{i}[1/f_i]$ we can write $t=\tau_i e_i$, with $e_i\in (B_i[1/f_i])^\times$. 
Let $E_i$ be the Cartier divisor on $V_i$ defined by $e_i$. We have $|E_i|\subset |\Sigma_{|V_i}|$.
By Lemma \ref{lem:Cart-div},  there exists $N_1\ge 0$, such that $f_i^{N_1}/e_i\in B_i$, for all $i$.
By \eqref{prop:cond-Witt4}, there exists an $N_2\ge 0$ such that  for all $i$ and all $s$
\[[f_i]^{N_2}  [t]^{j-1}\cdot F^{n-1}(a_s)\in W_n(B_i), \quad
[f_i]^{N_2} [t]^{j} \cdot F^{r-1}(b_s)\in W_r(B_i).\] 
Choose $N\ge j\cdot N_1+N_2$, such that $p^n\mid N$. We obtain  for all $i$
\[[\tau_i]^{j-1} [f_i]^{N-1}\cdot  F^{n-1}(a_s)\in W_n(B_i),\quad
   [\tau_i]^j [f_i]^{N}\cdot F^{r-1}(b_s)\in W_r(B_i).\]
Let $\rho: \Spec L\to U$, $L\in \Phi$. 
Assume the closed point of $\Spec \sO_L$ maps into $|\ol{Z}+\Sigma|$.
Then it follows from the above formula that 
\[\rho^*a_s\in \fillog{v_L((j-1)\cdot \ol{Z} +(N-1) \cdot\Sigma)}W_n(K)\subset \fil{v_L(j\cdot \ol{Z}+ N\cdot \Sigma)}W_n(K)\]
and 
\[\rho^*b_s\in \fillog{v_L(j\cdot \ol{Z}+ N\cdot \Sigma)}W_r(K).\]
By the choice of $N$ we have 
\[r_0:=\min\{{\rm ord}_p(v_L(j\cdot \ol{Z}+ N\cdot \Sigma)), n\}\ge r=\min\{{\rm ord}_p(j),n\};\]
hence 
\[\wV^{n-r}(\rho^*b_s)\in \wV^{n-r_0}\fillog{v_L(j\cdot \ol{Z}+ N\cdot \Sigma)}W_{r_0}(K)\subset 
\fil{v_L(j\cdot \ol{Z}+ N\cdot \Sigma)}W_n(K).\]
Running over all $\rho: \Spec L\to U$ yields
\[\gamma_{\ol{Y}}(a)\le j\cdot \ol{Z}+ N\cdot \Sigma.\]
This proves \ref{c6} and completes the proof of the proposition.
\end{proof}

The above proposition gives $c^{W_n}\le \gamma_n$ by Corollary \ref{cor:mot-cond-mini}. 
We show in Theorem \ref{thm:filF-motivic} below, that equality holds using symbol computations.
If we restrict to $\td(L/k)=1$ and $k$ is infinite, this follows, e.g., from \cite[Prop 6.4, (3)]{Kato-Russell}.
To handle the case of higher transcendence degree we need some preparations.
We start by identifying the local symbol for $W_n$ on regular projective curves over function fields.

\begin{para}\label{para:DRW}
Let $X\in \Sm$.
We denote by $W_n\Omega_X^\bullet$ the de Rham-Witt complex of  length $n$  on $X$ 
(see \cite{IlDRW}). 
By \cite[Cor 3.2.5]{KSY-RecII}  we have  $W_n\Omega^q\in \RSC_{\Nis}$.
See also \cite{Gros} and \cite{CR12} for details on how to define the transfers structure.
If $f: X\to Y$ is a morphism in $\Sm$, then the morphism 
\[\Gamma_f^*=f^*:  W_n\Omega^q(Y)\to  W_n\Omega^q(X)\]
induced by its graph $\Gamma_f\in \Cor(X,Y)$, is the natural pullback morphism induced 
by the functoriality of the de Rham-Witt complex. If $f$ is finite and surjective, then the transpose of the graph
defines an element $\Gamma^t_f\in \Cor(Y,X)$ and  $\Gamma_f^{t*}=f_*$, where $f_*$ is the pushforward defined using 
duality theory.
\end{para}

\begin{lemma}\label{lem:DRW-comp}
\begin{enumerate}[label= (\arabic*)]
\item\label{lem:DRW-comp1} The restriction, Verschiebung, Frobenius, and the differential
(which are part of the structure of the de Rham-Witt complex)
define morphisms in $\RSC_{\Nis}$
\[R: W_{n+1}\Omega^q\to W_n\Omega^q, \quad \wV: W_n\Omega^q\to W_{n+1}\Omega^q,\]
\[ F: W_{n+1}\Omega^q\to W_n\Omega^q, \quad d: W_n\Omega^q\to W_n\Omega^{q+1}.\]
\item\label{lem:DRW-comp2}
Let $W_n$ be the algebraic group of Witt vectors of length $n$ considered as a presheaf on $\Sm$.
Then there is a unique structure of presheaf with transfers on $W_{n}$, for all $n$,
which is unique with the following  properties
\begin{enumerate}
\item\label{lem:DRW-comp2.0} the restriction $R: W_{n+1}\to W_n$ is compatible with the transfer structure, for all $n$;
\item\label{lem:DRW-comp2.1} 
if $f:X\to Y$ is a morphism in $\Sm$ with graph $\Gamma_f\in \Cor(X,Y)$, then $\Gamma_f^*: W_n(Y)\to W_n(X)$
is the pullback from the presheaf structure.
\end{enumerate}
In particular, the Nisnevich sheaf with transfers $W_n\Omega^0=W_n\sO$ from \ref{para:DRW}
coincides with the Nisnevich sheaf with transfers
defined by the algebraic group $W_n$ (see \cite[Cor 3.2.5]{KSY-RecII}).
\end{enumerate}
\end{lemma}
\begin{proof}
\ref{lem:DRW-comp1}. We have to show, that if $\alpha\in\Cor(X, Y)$ is a finite correspondence, 
then the following morphisms are equal on $H^0(Y, W_n\Omega^{q}_Y)$
\[\alpha^* R=R\alpha^*,\quad \alpha^*V=V\alpha^*, \quad \alpha^*F=F\alpha^*,\quad  \alpha^*d=d\alpha^*.\]
This follows from \cite[Proof of Prop 3.5.4]{CR12}.  

\ref{lem:DRW-comp2}. The existence of such a transfer structure follows, e.g., from \ref{para:DRW}.
The last part of the statement follows since the two transfer structures satisfy \ref{lem:DRW-comp2.0}, \ref{lem:DRW-comp2.1}.

It remains to prove the uniqueness. Assume we have two transfer actions on $W_n$ 
with \ref{lem:DRW-comp2.0}, \ref{lem:DRW-comp2.1}. For $\alpha\in \Cor(X,Y)$ a finite correspondence denote
by  $\alpha^*, \alpha^\bigstar : W_n(Y)\to W_n(X)$ the two actions. We have to show they are equal.
Let $f:X\to Y$ be a morphism. By assumption we have $\Gamma_f^*=\Gamma_t^\bigstar=:f^*$;
if $f$ is finite and and surjective we set $f_*:=(\Gamma_f^t)^*$ and $f_\bigstar:=(\Gamma_f^t)^\bigstar$.
In general for $\alpha$ as above we want to show $\alpha^*=\alpha^\bigstar$.
It suffices to check this after shrinking $X$ around its generic points. Hence we can assume, that
$X$ is connected and $\alpha=Z\subset X\times Y$ with $Z$ smooth, integral, and finite free over $X$.
Denote by $f:Z\to X$ and $g:Z \to Y$ the maps induced by the projections.  Then $\alpha^\bigstar=f_\bigstar g^*$
and $\alpha^*=f_* g^*$. It remains to show $f_\bigstar=f_*$.
We may shrink $X$ further and hence assume that $f:Z=\Spec L\to X=\Spec K$ is induced by a finite field extension
$L/K$ of function fields over $k$. By transitivity it suffices to consider the two cases where $L/K$ is either separable or purely
inseparable of degree $p$.

{\em 1st case: $L/K$ separable.} Let $K'/K$ be a Galois hull of $L/K$ and set $X'=\Spec K'$.
We obtain the cartesian diagram
\[\xymatrix{\coprod_{i=1}^n X'\ar[d]\ar[r]^-{\coprod_i \sigma_i}& Z\ar[d]^f \\
                 X'\ar[r]^u & X,
}\]
where the vertical map on the left is induced by the universal property of the coproduct from the identity on $X'$,
$u$ is induced by the inclusion $K\inj K'$, and the $\sigma_i: X'\to Z$, $i=1,\ldots, n$,
are induced by be all the $K$-embeddings $L\inj K'$. For $a\in W_n(L)$ we obtain
\[u^*f_*a= (\Gamma_f^t\circ \Gamma_u)^*= \sum_i \Gamma_{\sigma_i}^*\]
and similar with $u^*f_\bigstar$. Thus $u^*f_*=u^*f_\bigstar$ and since $u^*: W_n(K)\inj W_n(K')$ is injective
we have proven the claim in this case.

{\em 2nd case: $L/K$ purely inseparable of degree $p$.}
In this case we have 
\eq{lem:DRW-comp3}{f_*f^*(-)=[L:K]\cdot(-) =p\cdot(-) =f_\bigstar f^\bigstar(-)\quad \text{on }W_n(X).}
Let $\ul{p}: W_n\to W_{n+1}$ be the map {\em lift-and-multiply-by-$p$}; thus it sends a Witt vector
$(a_0,\ldots, a_{n-1})$ in $W_n(A)$, where $A$ is some $\F_p$-algebra, to $(0,a_0^p,\ldots, a_{n-1}^p)$.
Let $b\in W_n(L)$. Clearly we find an element $a\in W_{n+1}(K)$ such that $f^*a =\ul{p}(b)$.
We obtain
\[
\ul{p}(f_* b)  \stackrel{\ref{lem:DRW-comp2.0}}{=} f_*\ul{p}(b)
                = f_*(f^*a)
                  \stackrel{\eqref{lem:DRW-comp3}}{=}p\cdot a  = \ul{p}R(a).
\]
The same computation works for $f_\bigstar b$.
Thus $\ul{p}(f_* b)= \ul{p}(f_\bigstar b)$,
and the claim follows the injectivity of  $\ul{p}$.
\end{proof}

\begin{lemma}\label{lem:trace-dlog}
Let $f: Y\to X$ be a finite  and surjective morphism in $\Sm$.
Then for all $u\in H^0(Y, \sO_Y^\times)$ and all $n\ge 1$ we have 
\[f_*\dlog[u]= \dlog [\Nm_{Y/X}(u)] \quad \text{in } H^0(X, W_n\Omega^1_X),\]
where $\Nm_{Y/X}: f_*\sO^\times_Y\to \sO_X^\times$ is the usual norm.
\end{lemma}
\begin{proof}
Note that $f$ is flat by \cite[Thm 23.1]{Matsumura}, hence also finite locally free, so that $\Nm_{Y/X}$ is defined.
It suffices to prove the equality after shrinking $X$ around its generic points.
Thus we can assume that $f$ corresponds to a finite field extension $L/K$.
By transitivity it suffices to consider the cases where $L/K$ is separable or purely inseparable of degree $p$.

{\em 1st case: $L/K$ finite separable.} We have 
$W_n\Omega^q_{L}=W_n(L)\otimes_{W_n(K)}W_n\Omega^q_K$ (see \cite[I, Prop 1.14]{IlDRW}).
 By the projection formula and 
Lemma \ref{lem:DRW-comp}\ref{lem:DRW-comp2}, we have $f_*= \Tr_{L/K}\otimes \id$.
Let $K^{\rm sep}$ be a separable closure of $K$. Note that $W_n(K)\to W_n(K^{\rm sep})$
is faithfully flat (since it is ind-\'etale and $\Spec W_n(K)$ is one point). Hence by
\'etale base change and fppf descent the natural map $W_n\Omega^1_K\to W_n\Omega^1_{K^{\rm sep}}$
is injective. Thus it suffices to check the equality in $W_n\Omega^1_{K^{\rm sep}}$.
Let $\sigma_1,\ldots, \sigma_r: L\inj  K^{\rm sep}$ be all $K$-embeddings, then
by the above we have in $W_n\Omega^1_{K^{\rm sep}}$
\[f_*\dlog[u]=\sum_{i=1}^r \sigma_i(\dlog[u])= \dlog\left[\prod_{i=1}^r \sigma_i(u)\right]= \dlog[\Nm_{L/K}(u)].\]

{\em 2nd case: $L/K$ is purely inseparable of degree $p$.}
We have $\Nm_{L/K}(u)=u^p\in K$. Since the map {\em lift-and-multiply-by-$p$}, 
 $\ul{p}: W_n\Omega^1_K \to W_{n+1}\Omega^1_K$ is injective by \cite[I, Prop 3.4]{IlDRW} and commutes with $f_*$
the statement follows from the following  equality in $W_{n+1}\Omega^1_{K}$:
\mlnl{\ul{p}(f_*\dlog[u]_n)= f_*\dlog[u^p]_{n+1}= f_*(1)\cdot \dlog[u^p]_{n+1} \\ = \ul{p} \dlog[\Nm_{L/K}(u)]_n.}
This completes the proof of the lemma.
\end{proof}

\begin{para}\label{para:Res}
Let $A$ be a ring of characteristic $p$ and set $B:=A[[t]][\frac{1}{t}]$.
Recall from \cite[\S 2.2, Prop 3]{Kato-GLCFII} and \cite[Prop 2.12]{Ru}
that there is a residuum map
\eq{para:Res1}{\Res_t: W_n\Omega^*_{B}\to W_n\Omega^{*-1}_A}
which is $W_n\Omega^*_A$-linear (where we consider the left-module structures), commutes with 
$R$, $F$, $V$, and  $d$, is zero on $W_n\Omega^*_{A[[t]]}$, 
and satisfies the equality $\Res_t(\alpha \dlog [t])= \alpha(0)$, for  $\alpha\in W_n\Omega^*_{A[[t]]}$.

Let $K$ be a function field over $k$ and $C$ a regular projective connected curve over $K$ with function field $E=K(C)$.
Recall from \cite[Def-Prop 1]{RuErr} that the residue map 
\[\Res_{C/K,x}: W_n\Omega^*_{E}\to W_n\Omega^{*-1}_K\]
at a closed point $x\in C$ is defined as follows:
by a result of H\"ubel-Kunz we find an integer  $m_0\ge 0$ such that for all $m\ge m_0$ 
the curve $C_m:= \Spec (\sO_C\cap K(E^{p^m}))$ is smooth over $K$ and, 
if $x_m$ denotes the image of $x$ under the finite homeomorphism $C\to C_m$, then 
the residue field $K_m:=K(x_m)$ is separable over $K$. Hence $\sO_{C_m, x_m}^h$ has a unique coefficient field
containing $K$,  which we identify with $K_m$. Set $E_m:=K(C_m)= K(E^{p^m})$.
The  choice of a local parameter  $t\in \sO_{C_m, x_m}$ yields a canonical inclusion $E_m\inj  K_m((t))$.
We define $\Res_{C/K,x}$  as the composition
\[W_n\Omega^*_E\xr{\Tr_{E/E_m}} W_n\Omega^*_{E_m}\inj W_n\Omega^*_{K_m((t))}\xr{\eqref{para:Res1}}
W_n\Omega^*_{K_m}\xr{\Tr_{K_m/K}} W_n\Omega^{*-1}_K.\]
(Here we should observe that if $\pi: \Spec L\to \Spec K$ is a finite extension, 
then the trace $\Tr_{L/K}:W_n\Omega^q_L\to W_n\Omega^q_{K}$ from \cite[Thm 2.6]{Ru} 
is equal to the pushforward $\pi_*$ from \ref{para:DRW}.
Indeed in the case $q=0$ this follows from Lemma \ref{lem:DRW-comp}\ref{lem:DRW-comp2} and 
Lemma \ref{lem:Witt-trace}; by transitivity, the general case is reduced to a simple extension
$L=K[a]$ in which case it follows from the fact that both maps commute with $V$, $F$, $d$, satisfy a projection formula, 
and the equality $[a]^{i-1}d[a]= i_0^{-1}F^ed[a]^{i_0}$, where $i=p^e i_0\ge 1$ with $(i_0,p)=1$.)
\end{para}

\begin{remark}\label{rmk:Res-char-2}
In \cite[2.]{Ru} and \cite{RuErr}, where the trace and the residue symbol mentioned above are constructed it
is always assumed that the characteristic is not $2$. The reason for this that the structure theorem
by Hesselholt and Madsen which in {\em loc. cit.} is cited as Theorem 2.1 was only known for $\Z_{(p)}$-algebras,
with $p$ odd at that time. This theorem is used in Proposition 2.4 and Lemma 2.9 of {\em loc. cit.}
which are needed to define the trace and the formal residue symbol, respectively.
However, the Theorem 2.1 of {\em loc. cit.} is also available for $\Z_{(2)}$-algebras
by \cite[4.2]{Costeanu} hence all the results from {\em loc. cit.} extend to the case $p=2$.
\end{remark}

\begin{lemma}\label{lem:LS-DRW}
Let $C/K$ and $x\in C$ be as in \ref{para:Res}. 
Then the corresponding local symbol of  $W_n\Omega^q$ (see \ref{para:symbol}) is  given by
\[(a, f)_{C/K,x}= \Res_{C/K,x}(\alpha\cdot \dlog [f]),\quad \alpha\in W_n\Omega^q_{K(C)}, \, f\in K(C)^\times,\]
where $[f]=(f,0,\ldots, 0)\in W_n(K(C))$. 

In particular, if $L\in\Phi$ with coefficient field $\sigma: K\inj \sO_L$ and local parameter $t\in \sO_L$,
then the local symbol $( -, -)_{L,\sigma}: W_n\Omega^q_L\times L^\times \to W_n\Omega^q_K$ 
(see \ref{para:symbol-hdvf}) is given by the composition
\[W_n\Omega^q_L\times L^\times\xr{\hat{\sigma}\wedge \dlog\circ[-]\circ\hat{\sigma}}  
       W_n\widehat{\Omega}^{q+1}_{K((t))}\xr{\Res_{t}} W_n\Omega^q_K,\]
where we denote by $\hat{\sigma} : L\inj K((t))$ the canonical inclusion.
\end{lemma}
\begin{proof}
We have to show that the family of maps $\{\Res_{C/K,x}(-\cdot \dlog[-])\}_x$ with $x$ running through all the closed points of $C$,
satisfies the properties \ref{LS1} - \ref{LS4} from \ref{para:symbol}.
\ref{LS1} (linearity) is clear and since we can choose the modulus $D$ for \ref{LS3} as large as we want  
this condition is clear from Lemma \ref{lem:Witt-dlog} below; \ref{LS4} (the reciprocity law) holds by \cite[Thm 2]{RuErr}
(see also Remark \ref{rmk:Res-char-2}).
It remains to show \ref{LS2}, i.e., 
\[\Res_{C/K,x}(\alpha \dlog (f))= v_x(f)\Tr_{K(x)/K}(\alpha(x)),\quad \alpha\in W_n\Omega^q_{C,x}.\]
To this end choose $m$ as in \ref{para:Res} above. Then  $K(x)/K(x_m)$ is purely inseparable of degree, say, $p^s$ and 
we can write 
\[[E:E_m]= p^{s+e},\]
where $p^e$ is the ramification index of $x/x_m$.
Denote by $\ul{p}^s: W_n\Omega^q\to W_{n+s}\Omega^q$ the map {\em lifting-and-multiplying by $p^s$}; it is injective,
by \cite[I, Prop 3.4]{IlDRW}.
Denote by $\sigma : K_m:=K(x_m)\inj \sO_{C_m,x_m}^h\inj \sO_{C,x}^h$ the inclusion of the coefficient field.
By \cite[Thm 2.6(iii)]{Ru} there exists a $\beta\in W_{n+s}\Omega^q_{K_m}$ mapping
to $\ul{p}^s \alpha(x)\in W_{n+s}\Omega^{q}_{K(x)}$ and we have 
\eq{lem:LS-DRW0}{\Tr_{K(x)/K_m}(\alpha(x))= R^s(\beta).}
By the choice of $\beta$, we have
\eq{lem:LS-DRW1}{\ul{p}^s (\alpha) - \sigma(\beta)\in \Ker(W_{n+s}\Omega^q_{\sO_{C,x}^h}\to W_{n+s}\Omega^q_{K(x)}).}
Since the kernel is the differential graded ideal generated by $W_{n+s}(\fm_{x})$ we obtain
in $W_{n+s}\Omega^q_{K}$
\begin{align*}
\ul{p}^s\Res_{C/K,x}(\alpha\dlog[f]) & = \Res_{C/K,x}(\ul{p}^s(\alpha\dlog[f]))\\
                                          & = \Res_{C/K,x}(\sigma(\beta)\dlog[f]), &\eqref{lem:LS-DRW1}\\
                                          & =\Res_{C_m/K,x_m}(\Tr_{E/E_m}(\sigma(\beta)\dlog[f]))), & \text{defn.}\\
                                        &=\Res_{C_m/K,x_m}(\beta \dlog\Nm_{E/E_m}[f]), & \text{\ref{lem:trace-dlog}}\\
                                       &=v_{x_m}(\Nm_{E/E_m}(f)))\cdot \Tr_{K_m/K}(\beta), & \text{defn.} \\
                                       & =[K(x):K(x_m)]\cdot v_x(f)\cdot \Tr_{K_m/K}( \beta )\\
                                       &=v_x(f) \cdot\ul{p}^s \Tr_{K_m/K}(R^s(\beta))\\
                                      &=v_x(f)\cdot \ul{p^s}\Tr_{K(x)/K}(\alpha(x)), & \eqref{lem:LS-DRW0}.
\end{align*}
Here the first equality follows from the fact that $\Res_{C/K, x}$ commutes with the restriction $R$.
(This follows from the definition and the fact that $\Res_t$ from \eqref{para:Res1} and $\Tr$
commute with $R$, for the latter see, e.g., Lemma \ref{lem:DRW-comp}\ref{lem:DRW-comp1}.)
The statement follows from the injectivity of $\ul{p}^s$.
\end{proof}

\begin{para}\label{para:relDRW}
Let $A$ be a $\Z_{(p)}$-algebra. For an $A$-algebra $B$ we denote by 
$W_n\Omega^\bullet_{B/A}$ the relative de Rham-Witt complex of Langer-Zink (see \cite{LZ}).
It is equipped with $\wR$, $\wF$, $\wV$, $d$ as usual.
If $B[x]$ is the polynomial ring with coefficients in $B$, we denote
by $I_r\subset W_n\Omega^{\bullet}_{B[x]/A}$ the differential graded ideal generated by
$W_n(x^rB[x])$. We define the $x$-adic completion of $W_n\Omega^{\bullet}_{B[x]/A}$ to be
\[W_n\widehat{\Omega}^\bullet_{B[[x]]/A}:=\varprojlim_{r} W_n\Omega^\bullet_{B[x]/A}/I_r.\]
Note that $W_n\Omega^\bullet_{B[x]/A}/I_r= W_n\Omega^\bullet_{(B[x]/(x^r))/A}$ (see \cite[Lem 2.4]{GH}). 
In particular, $W_n\widehat{\Omega}^\bullet_{B[[x]]/A}$ is a $W_n(B[[x]])=\varprojlim_r W_n(B[x]/(x^r))$-module.
\end{para}

\begin{lemma}\label{lem:Witt-dlog}
The following equalities hold in  $W_n\widehat{\Omega}^1_{\Z_{(p)}[[x]]/\Z_{(p)}}$:
\[-\dlog [1-x]= \sum_{i\ge 0} [x]^i d[x] + 
                  \sum_{s=1}^{n-1}\sum_{(j,p)=1}  \tfrac{1}{j} d\wV^s( [x]^j  ).\]
\end{lemma}
\begin{proof}
We prove this by induction over $n$. The case $n=1$ is clear.
Assume $n\ge 2$. 
By \cite[Cor 2.13]{LZ} we find unique elements  $a_i\in W_n(\Z_{(p)})$ and $b_{s,j}\in W_{n-s}(\Z_{(p)})$ such that 
\[-\dlog[1-x]= \sum_{i\ge 0} a_i [x]^i d[x]+ \sum_{s=1}^{n-1}\sum_{(j,p)=1}d\wV^s(b_{s,j} [x]^j).\]
Applying $\wF^{n-1}$ we obtain in $\widehat{\Omega}^1_{\Z_{(p)}[[x]]/\Z_{(p)}}$
\mlnl{-\dlog(1-x)=\sum_{k\ge 0}  x^k dx \\ = \sum_{i\ge 0} F^{n-1}(a_i) x^{(i+1)p^{n-1}-1} dx 
                  +\sum_{s=1}^{n-1} \sum_{(j,p)=1} F^{n-1-s}(b_{s,j})j x^{jp^{n-1-s}-1}dx.}
By induction hypothesis we have for all $i$, $j$, and for $s=1,\ldots, n-2$
\[a_i=1+\wV^{n-1}(e_i), \quad b_{s,j }=\tfrac{1}{j}+ \wV^{n-s-1}(f_{s,j}), \]
with $e_i, f_{s,j}\in \Z_{(p)}$.
Comparing coefficients we obtain in $\Z_{(p)}$
\[1= F^{n-1}(a_i)=1+p^{n-1}e_i, \]
and for $s= 1,\ldots, n-2$ 
\[\tfrac{1}{j}= F^{n-s-1}(b_{s,j})=\tfrac{1}{j} + p^{n-s-1}f_{s,j},\]
hence $e_i=f_{s,j}=0$; further we find $b_{n-1,j}=1/j\in W_1(\Z_{(p)})$. 
\end{proof}

\begin{para}\label{para:res-formula}
Let $K$ be a field and $\Res_t: W_n\Omega^*_{K((t))}\to W_n\Omega^{*-1}_{K((t))}$ the resiude map from \ref{para:Res1}.
Then for all $r, s\ge 0$, $i,j\in\Z$,  $a\in W_{n-r}(K)$ and $b\in W_{n-s}(K)$
the following equality holds in  $W_n(K)$
\mlnl{\Res_t(\wV^r([a][t]^i)d \wV^s([b][t]^j))= \\
\begin{cases} {\rm sgn}(j) {\rm gcd}(i,j) \wV^{r+s-c}([a]^{p^{s-c}}[b]^{p^{r-c}}), & \text{if } jp^r+i p^s=0,\\ 
                   0,& \text{else},  \end{cases}}
where ${\rm sgn}(j):= j/|j|$, if $j\neq 0$, and ${\rm sgn}(0):=0$, and $c=\min\{r,s\}$ (see \cite[Prop 2.12]{Ru})
\end{para}

\begin{lemma}\label{lem:LSW-form1}
Let $L\in\Phi$ and let $\sigma: K\inj \sO_L$ be a coefficient field. Let $t\in \sO_L$ be a local parameter,
 and $c\in K$. 
\begin{enumerate}[label = (\arabic*)]
\item\label{lem:LSW-form1.1} Let $ r\ge 1$ and write $r=p^e r_0$, with $(r_0, p)=1$, $e\ge 0$. Then 
\[([t]^{-r_0}, 1-t^r c)_{L,\sigma}= -r_0 \wV^e( [c]),\quad \text{in }W_{e+1}(K).\]
\item\label{lem:LSW-form1.2} Let $r\ge 1$ with $(r,p)=1$ and $m=p^u m_0$, with $(m_0, p)=1$, $u\ge 1$. 
 Assume $r> m_0$. Then for all $n\ge 1$
\[([t]^{-m}, 1-t^{r}c)_{L,\sigma}= 0,\quad \text{in } W_n(K).\] 
\end{enumerate}
\end{lemma}
\begin{proof}
\ref{lem:LSW-form1.1}. By the Lemmas \ref{lem:LS-DRW} and \ref{lem:Witt-dlog} we have 
\begin{align*}
([t]^{-r_0}, 1-t^r c)_{L,\sigma} =& \Res_t([t]^{-r_0}\dlog[1-t^r c])\\
                                              =& -\sum_{i\ge 0}\Res_t([c]^i [t]^{ir-r_0} d[t^r c])\\
                                               &  -\sum_{s=1}^{e} \sum_{(j,p)=1} \tfrac{1}{j}\Res_t([t]^{-r_0}d\wV^s([c]^j[t]^{jr})).
\end{align*}
Now the claim follows from \ref{para:res-formula}. The proof of \ref{lem:LSW-form1.2} is similar.
\end{proof}

\begin{lemma}\label{lem:fillog-explicit}
Let $L\in\Phi$ and let $t\in \sO_L$ be a local parameter. 
Let $K\inj \sO_L$ be a coefficient field.  Then, for $r\ge 1$, any element  
$ a\in \fillog{r}W_n(L)/W_n(\sO_L)$ can be  written uniquely in the following way 
\[ a= \sum_{  0>i p^{n-1} \ge -r } a_i [t]^i 
+\sum_{s=1}^{n-1}\sum_{\substack{0>j p^{n-1-s} \ge -r\\ (j,p)=1}} \wV^s(b_{s,j}[t]^j),\]
where $a_i\in W_n(K)$ and $b_{s,j}\in W_{n-s}(K)$. 
\end{lemma}
\begin{proof}
We can assume $L$ is complete and hence have $L=K((t))$.
By \cite[Lem 4.1.1]{HeMa04} (see also \cite[Lem 2.9]{Ru})
we can write any element $a$ in $W_n(K((t)))/W_n(K[[t]])$ uniquely in the form
\[a= \sum_{  0>i } a_i [t]^i 
+\sum_{s=1}^{n-1}\sum_{\substack{0>j \\ (j,p)=1}} \wV^s(b_{s,j}[t]^j),\]
with $a_i\in W_n(K)$ and $b_{s,j}\in W_{n-s}(K)$. 
Now, $a\in \fillog{r}W_n(L)/W_n(\sO_L)$ is equivalent to the following equality in $W_n(K((t)))/W_n(K[[t]])$
\mlnl{0=[t]^r F^{n-1}(a)\\ = \sum_{0>i} F^{n-1}(a_i) [t]^{ip^{n-1}+r} + 
        \sum_{s=1}^{n-1}\sum_{\substack{0> j\\ (j,p)=1}} \wV^s(F^{n-1}(b_{s,j}))\cdot [t]^{jp^{n-1-s}+r}.}
This yields the statement.
\end{proof}

\begin{corollary}\label{cor:gr-Wfillog}
Let  $r=p^e r_0\ge 1$ with $e\ge 0$ and $(r_0, p)=1$.  Let $L\in\Phi$ have local parameter $t\in \sO_L$
and let $\sigma: K\inj \sO_L$  be a coefficient field. Set 
$\gr_r^{\rm log} W_n(L):=\fillog{r}W_n(L)/\fillog{r-1}W_n(L)$, $n\ge 1$.
\begin{enumerate}[label=(\arabic*)]
\item\label{cor:gr-Wfillog1} Assume $e\in [0,n-1]$. There is a group isomorphism
\[W_{e+1}(K)\xr{\simeq} \gr_r^{\rm log}W_n(L),\quad
b\mapsto \wV^{n-1-e}(b[t]^{-r_0}) \text{ mod } \fillog{r-1}W_n(L).\]
\item\label{cor:gr-Wfillog2} Assume $e\ge n$. Then there is a group isomorphism
\[W_n(K)\xr{\simeq} \gr_r^{\rm log}W_n(L), \quad b \mapsto  b[t]^{-p^{e-n+1}r_0}\text{ mod } \fillog{r-1}W_n(L).\]
\end{enumerate}
\end{corollary}
\begin{proof}
This follows directly from Lemma \ref{lem:fillog-explicit}.
\end{proof}

\begin{corollary}\label{cor:gr-Wfil}
Let  $r=p^e r_0\ge 1$ with $e\ge 0$ and $(r_0, p)=1$.  Let $L\in\Phi$ have local parameter $t\in \sO_L$
and let $\sigma: K\inj \sO_L$  be a coefficient field. Set 
$\gr_r W_n(L):=\fil{r}W_n(L)/\fil{r-1}W_n(L)$, $n\ge 1$.
\begin{enumerate}[label=(\arabic*)]
\item\label{cor:gr-Wfil1} Assume $e=0$. Write $r-1=p^{e_1}r_1$ with $e_1\ge 0$ and $(r_1, p)=1$. Then
$\gr_rW_n(L)=0$, if $e_1\ge n$ and, if $e_1\in [0,n-1]$ there is a group isomorphism
\[K\xr{\simeq}\gr_r W_n(L), \quad b\mapsto \wV^{n-1-e_1}([b t^{-r_1}]) \text{ mod } \fil{r-1}W_n(L).\] 
\item\label{cor:gr-Wfil2} Assume $e\in [1,n-1]$. There is a group isomorphism
\[K\oplus W_e(K)\xr{\simeq} \gr_rW_n(L),\]
\[(b,c)\mapsto \wV^{n-1}(bt^{-(r-1)})+\wV^{n-e}(c[t]^{-r_0p}) \text{ mod } \fil{r-1}W_n(L).\]
\item\label{cor:gr-Wfil3} Assume $e\ge n$. Then there is a group isomorphism
\[K\oplus W_n(K)\xr{\simeq} \gr_rW_n(L),\]
\[(b,c)\mapsto \wV^{n-1}(bt^{-(r-1)})+ c[t]^{-p^{e-n+1}r_0}\text{ mod } \fil{r-1}W_n(L).\]
\end{enumerate}
\end{corollary}
\begin{proof}
Let $e':=\min\{e, n\}$ and  recall 
\[\fil{r}W_n(L)= \fillog{r-1}W_n(L)+ \wV^{n-e'}\fillog{r}W_{e'}(L).\]
Thus \ref{cor:gr-Wfil2} and \ref{cor:gr-Wfil3} follow directly from Lemma \ref{lem:fillog-explicit}.
(For the injcetivity in \ref{cor:gr-Wfil2} use that $\wV^{n-e}(c [t]^{-r_0p})=\wV^{n-e-1}(\wV(c)[t]^{-r_0})$.)
Furthermore, it is immediate from Lemma \ref{lem:fillog-explicit}, that there is an injective map as in
\ref{cor:gr-Wfil1} and that any element in the target has a representative of the form $\wV^{n-1-e_1}(\beta[t]^{-r_1})$
with $\beta\in W_{e_1}(K)$.  Thus the statement follows if we show 
$\wV^{n-1-e_1}(\wV(\beta_1) [t]^{-r_1})\in \fil{r-1}W_n(L)$. 
But by Lemma \ref{lem:fillog-explicit} the element
 $\wV^{n-1-e_1}(\wV(\beta_1) [t]^{-r_1})=\wV^{n-e_1}(\beta_1 [t]^{-p r_1})$ 
lies in $\wV^{n-e_1}\fillog{r-1}W_{e_1}(L)\subset \fil{r-1}W_n(L)$. Hence the statement.
\end{proof}

\begin{proposition}\label{prop:Wfil-symb}
Let $L\in\Phi$ have residue field $\kappa_L$ and local parameter $t\in\sO_L$.
 Let $z_1,\ldots, z_m\subset \sO_L$ be a lift of some $p$-basis of
$\kappa/k$. Let $\sigma_0: K_0\inj \sO_L$ be the unique coefficient field with $z_i\in K_0$, $i=1,\ldots, m$.
Let $x$ be an indeterminate and   set $L_x:=\Frac(\sO_L[x]_{(t)}^h)$.
Denote also by $\sigma_0: K_0(x) \inj  L_x$ the canonical extension of $\sigma_0$.
Let $r\ge 1$  and $a\in \filF{r} W_n(L)$. 
Assume one of the following:
\begin{enumerate}[label=(\arabic*)]
\item\label{prop:Wfil-symb01} $(r,p)=1$  or $r=p=2$  or $m=0$, and  $(a, 1-x t^{r-1})_{L_x,\sigma_0}=0$. 
\item\label{prop:Wfil-symb02} $r>2$, $p| r$, $m\ge 1$,  and $(a, 1-x t^{r-1})_{L_x,\sigma_j}=0$,  for $j=0, 1$,
                                    where  $\sigma_1: K_1\inj \sO_{L}$ is the unique coefficient field with 
                                       $z_i/(1+z_i^{p^e}t)\in K_1$, for all $i$, with  $e={\rm ord}_p(r)$, and we denote also by 
                        $\sigma_1: K_1(x)\inj \sO_{L_x}$ the canonical extension.                                           
\end{enumerate}
Then $a\in \filF{r-1} W_n(L)$.
\end{proposition}
\begin{proof}
Since $k$ is perfect, a $p$-basis over $k$ is the same as a separating transcendence basis over $k$,
(e.g., \cite[Thm 0.21.4.5]{EGAIV1}), hence there are unique coefficient fields $K_0$ and $K_1$ as in the statement
(see \cite[IX, \S 3, No. 2]{BourbakiCA89}).
By Proposition \ref{prop:cond-Witt} and Corollary \ref{cor:mot-cond-mini} we know 
$\filF{r-1}W_n(L)\subset W_n(\sO_L, \fm_L^{-r+1})$; furthermore for all $b\in W_n(\sO_L, \fm_L^{-r+1})$
we have $(b, 1-xt^{r-1})_{L_x,\sigma}=0$ for all coefficient fields $\sigma$ (by Corollary \ref{cor:localLS3}). 
Thus in the following we may replace $a$ by $a+b$ with $b\in \filF{r-1}W_n(L)$. 
We will use $\sigma_0$ to identify $\hat{L}= K_0((t))$.

Write $r=p^e r_0$ with $e\ge 0$ and $(r_0, p)=1$. We distinguish four cases.

{\em 1st case: $e=0$.} Write $r-1=p^{e_1}r_1$ with $(r_1,p)=1$ and $e_1\ge 0$. 
By Corollary \ref{cor:gr-Wfil}\ref{cor:gr-Wfil1} we have $\gr_r W_n(L)=0$, if $e_1\ge n$, and there is nothing to show;
if  $e_1\in [0,n-1]$ we have
\[a\equiv \sum_{h\ge 0}F^h \wV^{n-1-e_1}([b_h][t]^{-r_1}) \quad\text{mod }\filF{r-1}W_n(L), \]
with $b_h\in K_0$. We compute in $W_n(K_0(x))$:
\begin{align*}
0 & = \bigg(\sum_h F^h(\wV^{n-1-e_1}([b_h][t]^{-r_1})), 1-xt^{r-1}\bigg)_{L_x, \sigma_0}, & & 
          \text{by \ref{prop:Wfil-symb01}},\\
    & =\sum_h F^h\wV^{n-1-e_1}\bigg([b_h] \cdot ([t]^{-r_1}, 1-xt^{r-1})_{L_x,\sigma_0}\bigg),& &
           \text{by Lem \ref{lem:LS-DRW}},\\
    &= -r_1\sum_h F^h\wV^{n-1-e_1}([b_h] \wV^{e_1}([x])), & &
                \text{by Lem \ref{lem:LSW-form1}\ref{lem:LSW-form1.1}},\\
   & = -r_1V^{n-1}(\sum_{h} b_h^{p^{e_1+h}} x^{p^h}).
\end{align*}
Hence $b_h=0$, for all $h\ge 0$, which completes the proof of the first case.

{\em 2nd case: $r=p=2$.} By Corollary \ref{cor:gr-Wfil}\ref{cor:gr-Wfil2}, \ref{cor:gr-Wfil3} we have
\[a\equiv \sum_h F^h\wV^{n-1}(b_h t^{-1}+c_h t^{-2})\quad \text{mod }W_n(\sO_L),\]
with $b_h, c_h\in K_0$. Note
\[\Res_t(t^{-1}\dlog(1-xt))= x,\quad \Res_t(t^{-2}\dlog(1-xt))=x^2.\]
Hence by \ref{prop:Wfil-symb01}
\[0=(a,1-xt)_{L_x,\sigma_0}= \wV^{n-1}\big(\sum_{h} b_h^{2^h}x^{2^h}+c_h^{2^h}x^{2^{h+1}}\big).\]
We obtain 
\[b_0=0\quad \text{and} \quad c_h=b_{h+1}^2,\quad \text{all } h\ge 0.\]
Thus reshuffling the sum defining $a$ we obtain
\[a=\sum_{h} F^h\wV^{n-1}(b_h t^{-1}+ F(b_{h+1} t^{-1}))= 2\sum_h F^h\wV^{n-1}(b_h t^{-1})=0.\]

{\em 3rd case: $1\le e\le n-1$ and $r>2$.} By Corollary \ref{cor:gr-Wfil}\ref{cor:gr-Wfil2} we have 
\[a\equiv \sum_{h\ge 0} F^h\bigg(\wV^{n-1}(b_h [t]^{-(r-1)})+ \wV^{n-e}(c_h [t]^{-r_0p})\bigg)\quad 
 \text{mod }\filF{r-1}W_n(L),\]
where $b_h\in K_0$ and $c_h\in W_e(K_0)$. By a similar computation as in the first case, 
the vanishing $(a, 1-xt^{r-1})_{L_x,\sigma_0}=0$ together with $r-1>r_0$ and 
Lemma \ref{lem:LSW-form1}, \ref{lem:LSW-form1.1} and \ref{lem:LSW-form1.2}, imply
$b_h=0$, for all $h\ge 0$.  Thus
\[a\equiv \sum_{h\ge 0} F^h\big( \wV^{n-e}(c_h [t]^{-r_0p})\big)\quad \text{mod }\filF{r-1}W_n(L).\]
It suffices to show
\eq{prop:Wfil-symb4}{c_h\in FW_e(K_0), \quad \text{all }h\ge 0.}
Indeed, then $V^{n-e}(c_h [t]^{-r_0p})=F V^{n-e}(c_h' [t]^{-r_0})$, for some $c_h'\in W_e(K_0)$,
which lies in $F\fillog{r/p}W_n(L)\subset F\fillog{r-2}W_n(L)$ (use $r\ge 3$ for the last inclusion).

If $m=\td(\kappa/k)=0$, then $K_0$ is perfect and \eqref{prop:Wfil-symb4} holds.
This completes the proof of the implication: \ref{prop:Wfil-symb01} $\Rightarrow$ $a\in \fil{r-1}W_n\sO_L$.

Now assume $m\ge 1$.
We prove \eqref{prop:Wfil-symb4} by contradiction using $(a, 1-xt^{r-1})_{L_x, \sigma_1}=0$ with 
$\sigma_1: K_1(x)\inj \sO_{L_x}$ as in \ref{prop:Wfil-symb02}.
Thus assume not all $c_h$ are in $FW_e(K_0)$. Let $h_0$ be the minimal $h$ with $c_h\not\in FW_e(K_0)$.
Hence modulo $\filF{r-1} W_n(L)$ we can write $a$ as $F^{h_0}(V^{n-e}(a'))$, with 
$a'=\sum_{h\ge h_0} F^{h-h_0}(c_h[t^{-r_0p}])$. Since $F: W_n(K_j(x))\to W_n(K_j(x))$ and 
$V^{n-e}: W_e(K_j(x))\to W_n(K_j(x))$,
 $j=0,1$, are injective, the element $a'$ also satisfies $(a', 1-xt^{r-1})_{L_x,\sigma_j}=0$, $j=0,1$.
Thus we can assume $n=e$ and $h_0=0$, i.e., $c_0\not\in FW_e(K_0)$ and we want to find a contradiction.
Since the elements $z_1,\ldots, z_m\in K_0$ from the statement form a $p$-basis we can write 
$c_0$ as follows:
\[c_0= \sum_{j=0}^{e-1} V^j\bigg(\sum_{I\subset [0,p-1]^m} [a_{I,j}]^p[z]^I\bigg),\]
where $a_{I,j}\in K_0$ and $[z]^I=[z_1]^{i_1}\cdots [z_m]^{i_m}$, for $I=(i_1,\ldots, i_m)$.
Therefore, $c_0\not\in FW_e(K_0)$ translates into
\eq{prop:Wfil-symb4.1}{\exists\, j\in[0,e-1], \,I\in[0,p-1]^m\setminus\{(0,\ldots, 0)\} \quad \text{such that}\quad a_{I,j}\neq 0.}
Since we want to compute the local symbol with respect to the coefficient field $\sigma_1: K_1(x)\inj \sO_{L_x}$, 
we have to rewrite $c_0$ as an element in $W_n(K_1[[t]])$. 
Set 
\[y_i:= \frac{z_i}{1+z_i^{ p^e} t}\in K_1, \quad i=1,\ldots, m.\]
Then 
\[c_0=\sum_{j=0}^{e-1}V^j\bigg(\sum_{I\subset [0,p-1]^m} [a_{I,j}]^p [y(1+z^{p^e}t)]^I\bigg),\]
where 
$[y(1+z^{ p^e}t)]^I:= \prod_{h=1}^m [y_h(1+z_h^{ p^e}t)]^{i_h}$. 
 Note that $a_{I,j}, z_h\in K_0\subset  K_1[[t]]$ are not constant.
The composition  $\wF^{e-1}d: W_e(-)\to \Omega^1$ is a morphism of reciprocity sheaves (see Lemma \ref{lem:DRW-comp}).
 Hence $\wF^{e-1}d$ commutes with the local symbol,
which on $\Omega^1$ is given by $(\alpha, f)_{L_x, \sigma_1}=\Res_{K_1((t))}(\alpha\wedge \dlog f)$
(see Lemma \ref{lem:LS-DRW}).
Using $\wF^{e-1}d F=0$ on $W_e$, we obtain the following equalities in $\Omega^1_{K_1(x)}$:
\begin{align*}
0= & \wF^{e-1}d\big(a, 1-xt^{r-1}\big)_{L_x,\sigma_1}\\
 = & (\wF^{e-1}d \big(c_0[t]^{-r_0p}), 1-xt^{r-1}\big)_{L_x,\sigma_1}\\
 = & \sum_{j=0}^{e-1}\sum_{I} 
            \bigg(\wF^{e-1-j}d \big([a_{I,j}]^p[y(1+z^{p^e}t)]^I[t]^{-r_0p^{j+1}}\big) , 1-xt^{r-1}\bigg)_{L_x,\sigma_1}\\
  = &\sum_{j=0}^{e-1}\sum_{I} 
                \Res_t\bigg(a_{I,j}^{p^{e-j}}t^{-r}\wF^{e-1-j}d([y(1+z^{p^e}t)]^I)\dlog(1-xt^{r-1})\bigg).
\end{align*}
Write 
\[a_{I,j}= \bar{a}_{I,j} + t b_{I,j}, \quad \bar{a}_{I,j}\in K_1, \, b_{I,j}\in K_1[[t]].\]
Denote  by 
\[\bar{\sigma}_j:K_j\xr{\simeq} \kappa_{L}, \quad j=0,1\]
the isomorphisms induced by $\sigma_j: K_j\inj \sO_L$.
Then  $\bar{\sigma}_1(\bar{a}_{I,j})=\bar{\sigma}_0(a_{I,j})$; in particular,
\eq{prop:Wfil-symb4.5}{a_{I,j}=0 \Longleftrightarrow  \bar{a}_{I,j}=0.}
For $j\in [0, e-1]$ we have $a_{I,j}^{p^{e-j}}\equiv\bar{a}_{I,j}^{p^{e-j}}$ mod $t^2$ and 
thus we obtain from the computation above 
\eq{prop:Wfil-symb5}{
0 =-\sum_{j=0}^{e-1}\sum_{I\subset [0,p-1]^m} 
                \Res_t\bigg(\bar{a}_{I,j}^{p^{e-j}}t^{-r}\wF^{e-1-j}d([y(1+z^{p^e}t)]^I) d(xt^{r-1})\bigg). }
We have 
\ml{prop:Wfil-symb6}{\wF^{e-1-j}d[y(1+z^{p^e}t)]^I\\ 
=\sum_{h=1}^m i_h (y(1+z^{p^e}t))^{I p^{e-1-j}} \dlog(y_h(1+ z_h^{p^e}t)).}
Note 
\[z_h= y_h+t\zeta_h,\quad \zeta_h\in K_1[[t]],\, h=1,\ldots, m.\]
Thus, $z_h^{p^e}\equiv y_h^{p^e}$ mod $t^{p^e}$.
Hence the coefficient of $\wF^{e-1-j}d[y(1+z^{p^e}t)]^I$ in $K_1$ in front of $dt$ is equal to
\eq{prop:Wfil-symb7}{ f_{I,j}:= q_I^{p^e} y^{Ip^{e-1-j}}, \text{ with } q_I= \sum_{h=1}^m i_h y_h;}
the coefficient of $\wF^{e-1-j}d[y(1+z^{p^e}t)]^I$ in $\Omega^1_{K_1}$ in front of $t$ is equal to 
$df_{I,j}$. (This is zero if  $j\in [0, e-2]$.) Thus by \eqref{prop:Wfil-symb5} we have
\[0= \sum_{j=0}^{e-1}\sum_I (\bar{a}_{I,j}^{p^{e-j}}(f_{I,j}dx+x df_{I,j}))=
d\left(\sum_{j=0}^{e-1}\sum_{I} \bar{a}^{p^{e-j}}_{I,j} f_{I,j}\cdot x\right).\]
Hence the element in the brackets has to be a $p$-th power, i.e., by \eqref{prop:Wfil-symb7}
\[K_1^p\ni \sum_{j=0}^{e-2} \bigg(\sum_{I\subset [0,p-1]^m} (\bar{a}_{I,j}q_{I}^{p^{j}})^p y^I\bigg)^{p^{e-1-j}}\cdot x
+\sum_{\stackrel{I\subset [0, p-1]^m}{I\neq 0}} (\bar{a}_{I,e-1}q_I^{p^{e-1}})^p y^I x.\]
Note
\[q_I=0 \Longleftrightarrow I=0.\]
Since $y_1,\ldots, y_m,x$ form a $p$-basis of $K_1(x)$ over $k$  we obtain
\[\bar{a}_{I,e-1}=0, \quad \text{for all } I\subset [0,p-1]^m\setminus \{(0,\ldots,0)\},\]
and 
\[\sum_{j=0}^{e-2} \sum_{I\subset [0,p-1]^m} (\bar{a}_{I,j}q_{I}^{p^{j}})^{p^{e-1-j}} y^{Ip^{e-2-j}}=0.\]
Since $y_1,\ldots, y_m\in K_1$ form a $p$-basis over $k$, we obtain, similar as above, 
$\bar{a}_{I,e-2}=0$, for all $I\neq 0$. We may proceed in this way  and obtain
\[\bar{a}_{I,j}=0,\quad \text{for all } I\neq 0, j\ge 0.\]
By \eqref{prop:Wfil-symb4.5}  this contradicts \eqref{prop:Wfil-symb4.1} and proves the statement in this case.

{\em 4th case: $e\ge n$ and $r>2$.} By Corollary \ref{cor:gr-Wfil}\ref{cor:gr-Wfil3} we have
\[a\equiv \sum_{h\ge 0} F^h\bigg(c_h[t]^{-p^{e-n+1}r_0}+ V^{n-1}(b_h[t]^{-(r-1)})\bigg)\quad \text{mod }\filF{r-1}W_n(L),\]
where $c_h\in W_n(K_0)$ and $b_h\in K_0$. As before it follows from $(a, 1-xt^{r-1})_{L_x,\sigma_0}=0$ and 
Lemma \ref{lem:LSW-form1} that $b_h=0$, for all $h\ge 0$.
Thus 
\[a\equiv \sum_{h\ge 0} F^h(c_h[t]^{-p^{e-n+1}r_0})\quad \text{mod }\filF{r-1}W_n(L).\]
Applying $\wV^{e-n+1}$ we obtain
\[\wV^{e-n+1}(a)\equiv \sum_{h\ge 0} F^h \wV(c_h' [t]^{-r_0p}) \quad \text{mod } \filF{r-1}W_{e+1}(L),\]
where $c_h'=V^{e-n}(c_h)\in W_e(K_0)$. Since $V^{e-n+1}(a)\in \fil{r}W_{e+1}(L)$ we can apply the third case, in particular
\eqref{prop:Wfil-symb4} to conclude $c_h\in FW_n(K_0)$, and then also $a\in \filF{r-1}W_n(L)$.
This completes the proof of the proposition.
\end{proof}

\begin{thm}\label{thm:filF-motivic}
Let $L\in\Phi$  and $r\ge 0$.
Then
\[\filF{r}W_n(L)=\widetilde{W}_n(\sO_L,\fm^{-r}),\]
i.e., the Brylinski-Kato-Matsuda-Russell conductor is motivic. 
\end{thm}
\begin{proof}
We have $\filF{r}W_n(L)\subset \widetilde{W}_n(\sO_L,\fm^{-r})$, by
 Proposition \ref{prop:cond-Witt} and Theorem \ref{thm:rec-cond}\ref{thm:rec-cond4}, 
 and we know this is an equality for $r=0$.
Let $t\in \sO_L$ be a local parameter.
 By Corollary \ref{cor:localLS3} we have 
\[a\in \widetilde{W}_n(\sO_L,\fm^{-r})\Rightarrow (a, 1-xt^m)_{L_x,\sigma}=0, 
\quad \text{for all }m\ge r \text{ and all } \sigma,\]
where $L_x=\Frac(\sO_L[x]_{(t)}^h)$ and $\sigma$ is running through all coefficient fields $\sigma: K\inj \sO_L$.
Furthermore, we know for any $a\in \widetilde{W}_n(\sO_L,\fm^{-r})$ there exists some $m\ge r$ such that  
$a\in\filF{m}W_n(L)$.  Hence the statement follows from Proposition \ref{prop:Wfil-symb}.
\end{proof}

\section{Lisse sheaves of rank 1 and the Artin conductor}\label{subsec:art}
In this section $k$ is a perfect field of characteristic $p>0$.
\subsection{The case of finite monodromy}
\begin{para}\label{para:art}
Consider the constant presheaf with transfers $\Q/\Z$, i.e., an elementary correspondence $V\in \Cor(X,Y)$, with
$X, Y$ smooth and connected, acts by multiplication with $[V:X]$. By \cite[Lem 6.23]{MVW}
\[X\mapsto H^1(X):=H^1_{\et}(X, \Q/\Z)= \Hom_{\cont}(\pi_1(X)^{\rm ab}, \Q/\Z)\]
is a presheaf with transfers, which  we denote by $H^1$ in the following. 
\end{para}
Note that $H^1\in\NST$ as follows from the following Lemma.
\begin{lemma}\label{lem:H1Nis}
Let $A$ be an abelian group. It defines a constant \'etale sheaf on $\Sm$.
Then the presheaf  $X\mapsto H^1_{\et}(X, A)$ is a Nisnevich sheaf on $\Sm$.
\end{lemma}
\begin{proof}
Let $\sH^i$ be the  Nisnevich sheafification of $X\mapsto H^i_{\et}(X, A)$. 
Then for any $X\in \Sm$ we have an exact sequence 
\[ H^1_{\Nis}(X, \sH^0)\to H^1_{\et}(X, A)\to H^0_{\Nis}(X, \sH^1)\to H^2_{\Nis}(X, \sH^0).\]
But $\sH^0=A$ is constant and hence by \cite[Thm 3.1.12]{VoDM} we have 
$H^i_{\Nis}(X, \sH^0)=H^i_{\Zar}(X, \sH^0)=0$, for all $i\ge 1$. Thus the presheaf from the statement is
equal to $\sH^1$. 
\end{proof}

\begin{lemma}\label{lem:ASW-EST}
The Artin-Schreier-Witt sequence 
\eq{ASW-seq}{0\to \Z/p^n\Z\to W_n\xr{F-1} W_n\to 0}
is an exact sequence of \'etale sheaves with transfers on $\Sm$,
where $F: W_n\to W_n$ is the base change over $\Spec k$ of the Frobenius on the $\F_p$-group scheme $W_n$.
\end{lemma}
\begin{proof}
The exactness of the sequence \eqref{ASW-seq} on $X_{\et}$ is classical. The map
$F-1:W_n\to W_n$ is a morphism of $k$-group schemes hence is compatible with transfers;
for the inclusion $\Z/p^n\Z\inj W_n$ this follows directly from Lemma \ref{lem:Witt-trace}.
\end{proof}

\begin{para}\label{para:Art-cond}
We denote by $\delta_n$ the composition
\[\delta_n : W_n(L)\to W_n(L)/(F-1)W_n(L)\cong H^1_\et(L, \Z/p^n\Z):=H^1_{p^n}(L),\]
which is the connecting homomorphism stemming from the Artin-Schreier-Witt sequence \eqref{ASW-seq}.
Then we set 
\[\fil{j}H^1_{p^n}(L):= \delta_n(\fil{j}W_n(L))=\delta_n({\rm fil}_j^{F} W_n(L)).\]

For $j\ge 0$, we set
\eq{para:Art-cond2}{ \fil{j}H^1(L):=\begin{cases}
\Im(H^1(\sO_L)\to H^1(L)), & \text{if } j=0,\\
H^1(L)\{p'\}\oplus \bigcup_{n\ge 1} \fil{j}H^1_{p^n}(L),&\text{if } j\ge 1,
\end{cases}}
with $H^1(L)\{p'\}= \bigoplus_{\ell\neq p}H^1_{\et}(L, \Q_{\ell}/\Z_{\ell})$ the prime-to-$p$-part of $H^1(L)$.

For $\chi\in H^1(L)$ we define
\eq{para:Art-cond3}{\Art_L(\chi)=\min\{j\ge 0\mid \chi\in \fil{j}H^1(L)\}.}
\end{para}

\begin{proposition}\label{prop:Art-cond}
The collection 
\[\Art=\{\Art_L: H^1(L)\to \N_0\mid L\in\Phi\}\]
 is a semi-continuous conductor on $H^1$, as is its restriction $\Art^{\le 1}$.
\end{proposition}
\begin{proof}
By Proposition \ref{prop:cond-Witt} and Lemma \ref{lem:im-cond}, 
$\Art$ satisfies \ref{c1}-\ref{c6} except possibly for \ref{c4}.
(For \ref{c5} note, that $W_n(Y)\to H^1_{p^n}(Y)$ is surjective for any affine scheme over $k$.)
It remains, to show that $\Art^{\le 1}$ satisfies \ref{c4}. 
Let $X$ be a smooth $k$-scheme and $a\in H^1(\A^1_X)$ with
\eq{prop:Art-level1.1}{\Art^1_{k(x)_\infty}(\rho_x^*a)\le 1,\quad \text{for all closed points } x\in X_{(0)},}
where $\rho_x: \Spec k(x)(t)_\infty=\Spec \Frac(\sO_{\P^1_{x},\infty}^h)\to \A^1_X$ is the natural map.
We want to show : $a\in H^1(X)$.
Since $H^1= H^1\{p'\}\oplus \varinjlim_n H^1_{p^n}$ with $H^1\{p'\}$ the $\A^1$-invariant 
subsheaf of prime-to-$p$-torsion, we can assume $a\in H^1_{p^n}(\A^1_X)$.
Furthermore, the question is local on $X$, hence we can assume $X=\Spec A$ affine.
We consider first the case $n=1$.
Condition \eqref{prop:Art-level1.1} implies
\eq{prop:Art-level1.2}{\rho_x^*a\in \Im\big(H^1_p(\sO_{\P^1_x,\infty}^h) \to H^1_p(k(x)(t)_\infty)\big).}
Denote by $a(x)$ the restriction of $a$ to $\A^1_x$. Since $H^1_p$ is a Nisnevich sheaf we conclude 
\[a(x)\in H^1_p(\P^1_x)=H^1_p(x).\]
Thus we find a polynomial $\tilde{a}=a_0+a_1 t+\ldots+ a_n t^n\in A[t]$
 mapping to $a$ such that for all closed points $x\in X$
there exist  $b_x\in k(x)$ and $g_x\in k(x)[t]$ with
\eq{prop:Art-level1.3}{\tilde{a}(x)= b_x + g_x^p- g_x, \quad \text{in }k(x)[t].}
Assume $n\ge 1$. Then,  $n=p\cdot n_1$, for some $n_1\ge 1$. We claim 
\eq{prop:Art-level1.4}{a_n=c_1^p, \quad \text{some }c_1 \in A^p.}
Indeed, write $n=p^e m$ with $e\ge 1$ and $(p,m)=1$, and for a fixed closed point $x\in X$ write
$g_x= c_0+ c_1t+\ldots + c_{p^{e-1}m} t^{p^{e-1}m}$; then \eqref{prop:Art-level1.3} implies
\[a_n(x)=c_{p^{e-1}m}^p, \quad a_{p^{i}m}(x)= c_{p^{i-1}m}^p-c_{p^{i}m},\, i\in[1,e-1],\quad a_m(x)=- c_m. \]
Hence for all maximal ideals $\fm\subset A$ we have
\[a_n \equiv \sum_{j=0}^{e-1}(-a_{p^jm})^{p^{e-j}} \quad \text{mod }\fm.\]
It follows that $a_n=(\sum_{j=0}^{e-1}(-a_{p^jm})^{p^{e-j-1}})^p\in A^p$, 
which yields \eqref{prop:Art-level1.4}. 

Now $a^{(1)}= \tilde{a}-(c_1t^{n_1})^p+ c_1t^{n_1}$ 
also has property \eqref{prop:Art-level1.3} and its degree is strictly smaller than $n$.
We can replace $a$ by $a^{(1)}$ in the above discussion and go on in this way until 
we reach a polynomial $a^{(r)}\in A[t]$
whose degree is strictly smaller than $p$ in which case \eqref{prop:Art-level1.3} forces it to be constant $=c_r\in A$.
We obtain 
\[\tilde{a}= c_r + \sum_{i=1}^{r-1} (c_it^{n_i})^p- c_i t^{n_i},\]
whence $a\in H^1_p(X)$. 

Let $n\ge 1$.
If $a\in H^1_{p^n}(\A^1_X)$ satisfies \eqref{prop:Art-level1.1}, then so does $p^{n-1}a\in H^1_p(\A^1_X)$.
By the case $n=1$ and the exact sequence ($X$ is affine)
\[0\to H^1_{p^{n-1}}(X)\to H^1_{p^n}(X)\xr{p^{n-1}\cdot } H^1_p(X)\to 0\]
we find an element $b\in H^1_{p^n}(X)$ such that
$p^{n-1}(a-b)=0$. Since $a-b$ also satisfies \eqref{prop:Art-level1.1} we obtain
$a-b\in H^1_{p^{n-1}}(X)$ by induction. This completes the proof.
\end{proof}

\begin{lemma}\label{lem:vanWv}
Let $K$ be a field of positive characteristic, $x$ an indeterminate,  and $g\in W_n(K(x))$.
Assume $F(g)-g=\wV^{n-1}(b x)$ for some $b\in K$.
Then $g\in \Z/p^n\Z$, i.e., $F(g)-g=0$.
\end{lemma}
\begin{proof}
If $n=1$, then $g^p-g=bx$ forces $g$ to be constant and hence
$g^p-g=0$, i.e., $g\in\F_p$. If $n\ge 2$, then $F(g)-g$ is zero when restricted to $W_{n-1}(K(x))$. Hence
 $g= m\cdot [1]+ \wV^{n-1}(f)$ with $f\in K(x)$, $m\in \Z$.
Thus $F(f)-f=bx$, and we conclude with the case $n=1$.
\end{proof}

\begin{proposition}\label{prop:Art-symb}
Let $L$, $t\in \sO_L$, $\sigma_j: K_j\inj \sO_L$, $j=0,1$, be as in Proposition \ref{prop:Wfil-symb}.
We also denote by $\sigma_j: K_j(x)\inj \sO_{L_x}$ the canonical extension.  
Let $r\ge 1$ and $a \in \fil{r}H^1_{p^n}(L)$.
Assume one of the following:
\begin{enumerate}[label=(\arabic*)]
\item\label{prop:Art-symb01} $(r,p)=1$ or $r=p=2$ or $m=0$, and $(a, 1-x t^{r-1})_{L_x,\sigma_0}=0$. 
\item\label{prop:Art-symb02} $r>2$, $p|r$, $m\ge 1$, and $(a, 1-x t^{r-1})_{L_x,\sigma_j}=0$, for $j=0, 1$.
\end{enumerate}
Then $a\in \fil{r-1}H^1_{p^n}(L)$.
\end{proposition}
\begin{proof}
Let $\tilde{a}\in \fil{r}W_n(L)$ be a lift of $a$.
If $(a, 1-xt^{r-1})_{L_x,\sigma_j}=0$, for some $j\in\{0,1\}$, then we find $g_j\in W_n(K_j(x))$ such that 
\eq{prop:Art-symb2}{(\tilde{a}, 1-xt^{r-1})_{L_x, \sigma_j}= F(g_j)- g_j.}
It suffices to show $\tilde{a}\in \filF{r-1}W_n(L)$. Write $r=p^e r_0$ with $e\ge 0$ and $(r_0,p)=1$.

{\em 1st case: $e=0$.} Write $r-1=p^{e_1}r_1$ with $e_1\ge 0$ and $(p,r_1)=0$.
If $e_1\ge n$, then by Corollary \ref{cor:gr-Wfil}\ref{cor:gr-Wfil1} we have $\fil{r}H^1_{p^n}(L)=\fil{r-1}H^1_{p^n}(L)$, 
else we have 
\[\tilde{a}\equiv \wV^{n-1-e_1}([b] [t]^{-r_1})\text{ mod } \fil{r-1}W_n(L),\]
for some $b\in K_0$.
Thus 
\begin{align*}
F(g_0)-g_0 & = (\tilde{a}, 1-xt^{r-1})_{L_x,\sigma_0}, & & \text{by \ref{prop:Art-symb01}},\\
               & = \wV^{n-1-e_1}\bigg([b] ([t]^{-r_1}, 1-xt^{r-1})_{L_x,\sigma_0} \bigg), & &
                                                        \text{by Lem \ref{lem:LS-DRW}},\\
               & = -r_1\wV^{n-1}(b^{p^{e_1}} x), & & \text{by Lem \ref{lem:LSW-form1}\ref{lem:LSW-form1.1}}.
\end{align*}
Lemma \ref{lem:vanWv} implies $F(g_0)-g_0=0$. Hence $\tilde{a}\in \filF{r-1}W_n(L)$ 
by Proposition \ref{prop:Wfil-symb}\ref{prop:Wfil-symb01}.

{\em 2nd case: $r=p=2$.} As in the  proof of Proposition \ref{prop:Wfil-symb} ({\em 2nd case})
we have $\tilde{a}\equiv \wV^{n-1}(b t^{-1}+ c t^{-2})$ mod $W_n(\sO_L)$, with $b,c\in K_0$,
and 
\[g_0^2 - g_0=(\tilde{a}, 1-x t)_{L_x,\sigma_0}= V^{n-1}(bx+cx^2).\]
This implies $c=b^2$; hence $a\in H^1_{p^n}(\sO_L)=\fil{1}H^1_{p^n}(L)$.

{\em 3rd case: $1\le e\le n-1$ and $r>2$.}
By Corollary \ref{cor:gr-Wfil}\ref{cor:gr-Wfil2} we have 
\[\tilde{a}\equiv \wV^{n-1}(b_j [t]^{-(r-1)})+ \wV^{n-e}(c_j [t]^{-r_0p})\quad 
 \text{mod }\filF{r-1}W_n(L),\]
where $b_j\in K_j$ and $c_j\in W_e(K_j)$, $j=0,1$.
By Lemma \ref{lem:LSW-form1}\ref{lem:LSW-form1.1} we have 
\[(\wV^{n-1}(b_j [t]^{-(r-1)}), 1-xt^{r-1})_{L_x,\sigma_j}=-(r-1)\wV^{n-1}(b_j x);\]
by Lemma \ref{lem:LSW-form1}\ref{lem:LSW-form1.2} we have 
\[(\wV^{n-e}(c_j [t]^{-r_0p}), 1-xt^{r-1})_{L_x, \sigma_j}= 0.\]
Thus by \ref{prop:Art-symb02}
\[F(g_j)-g_j=(\tilde{a}, 1-xt^{r-1})_{L_x,\sigma_j}= -(r-1)\wV^{n-1}(b_j x).\]
By Lemma \ref{lem:vanWv} we have $F(g_j)-g_j=0$, for $j=0,1$. 
Hence $\tilde{a}\in \filF{r-1}W_n(L)$ by Proposition \ref{prop:Wfil-symb}.

{\em 4th case: $e\ge n$ and $r>2$.}
By Corollary \ref{cor:gr-Wfil}\ref{cor:gr-Wfil3} we have
\[a\equiv c_j[t]^{-p^{e-n+1}r_0}+ V^{n-1}(b_j[t]^{-(r-1)})\quad \text{mod }\filF{r-1}W_n(L),\]
where $c_j\in W_n(K_j)$ and $b_j\in K_j$, for $j=0,1$.
As in the third case the following equality follows from Lemma \ref{lem:LSW-form1} for $j=0,1$
\[F(g_j)-g_j=(\tilde{a}, 1-xt^{r-1})_{L_x,\sigma_j}= -(r-1)\wV^{n-1}(b_j x).\]
Hence $\tilde{a}\in \filF{r-1}W_n(L)$ as above.  This completes the proof.
\end{proof}

\begin{thm}\label{thm:Art-motivic}
Let $L\in\Phi$ and $r\ge 0$.
Then
\[\fil{r}H^1(L)=\widetilde{H^1}(\sO_L,\fm^{-r}),\]
i.e., the Artin conductor is motivic, $\Art=c^{H^1}$. Furthermore,  $(c^{H^1})^{\le 1}$ is a conductor of level 1.
\end{thm}
\begin{proof}
The last statement follows from the first and Proposition \ref{prop:Art-cond}.
By Corollary \ref{cor:HI} it suffices to show  the corresponding statement on the subsheaf of $p^n$-torsion,
 for all $n\ge 1$.
Here the proof is the same as in Theorem \ref{thm:filF-motivic}
if we replace everywhere $W_n$ by $H^1_{p^n}$, $\filF{}$ by $\fil{}$, 
the reference to Proposition \ref{prop:cond-Witt} by a reference to Proposition \ref{prop:Art-cond},
and the reference to Proposition \ref{prop:Wfil-symb}
by a reference to Proposition \ref{prop:Art-symb}.
\end{proof}

\subsection{Lisse sheaves of rank 1}
In this subsection we fix  a prime number $\ell\neq p$, an algebraic closure $\ol{\Q}_\ell$ of $\Q_\ell$,
and a compatible system of primitive roots of unity $\{\zeta_n\}\subset \ol{\Q}_{\ell}^\times$.

\begin{para}\label{para:rank1}
We denote by ${\rm Lisse}^1(X)$ the group of isomorphism classes
of  lisse $\bar{\Q}_\ell$-sheaves on $X$ of rank 1, with group structure given by  $\otimes$. 
Note that 
\eq{para:rank11}{{\rm Lisse}^{1}(X)\cong\varinjlim_{E/\Q_\ell} H^1_{\et}(X, \sO_E^\times)
                   :=\varinjlim_{E/\Q_\ell}\varprojlim_{n} H^1_{\et}(X, (\sO_E/\fm_{E}^n)^\times), }
where $E$ runs over sub-extensions of $\ol{\Q}_\ell/\Q_\ell$ which are finite over $\Q_\ell$,
and $\sO_E$ and  $\fm_E$ denote the ring of integers and the maximal ideal, respectively. 
Indeed,  a  sheaf $M\in {\rm Lisse}^{1}(X)$ corresponds uniquely to a continuous morphism
$\pi_1^{\ab}(X)\to \ol{\Q}_\ell^\times$, which in particular implies that it 
factors as a continuous morphism $\pi_1^{\ab}(X)\to E^\times$, with some $E$ as above (e.g., \cite[1.1]{Deligne-WeilII}).
Since any representation of a profinite group in a finite dimensional $E$-vector space  has an $\sO_{E}$-lattice, 
we see that such a morphism factors via a continuous map 
\[\pi^{\ab}_1(X)\to {\rm Aut}_{\sO_E}(\fm_E^{-j}\sO_E)=\sO_E^\times.\]
The isomorphism classes of such maps correspond uniquely to elements in 
$\varprojlim_{n} H^1_{\et}(X, (\sO_E/\fm_{E}^n)^\times)$.
By  \ref{para:art} and Lemma \ref{lem:H1Nis} the isomorphism \eqref{para:rank11}
induces the structure of a Nisnevich sheaf with transfers on $X\mapsto {\rm Lisse}^1(X)$, i.e.,
\[{\rm Lisse}^1\in \NST.\]
Write
\[|\sO_E/\fm_E|=\ell^{r_E}, \quad \ell^{r_E}-1= p^{s_E}\cdot h_E, \quad \text{with } (h_E,p)=1, s_E\ge 0.\]
Then $\mu_{\ell^{r_E}-1}(\Qlb)\subset \sO_{E}^\times$ and the roots of unity fixed at the beginning of this
 subsection induce a canonical  isomorphism
\[\sO_{E}^\times \cong \Z/p^{s_E} \times \Z/h_E\times U^{(1)}_E. \]
Since $U^{(1)}_E$ is a pro-$\ell$ group this yields the following decomposition 
\[{\rm Lisse}^1= {\rm Lisse}^{1, p'}\oplus H^1_{p^\infty}\quad \text{in }\NST,\]
where 
\[X\mapsto {\rm Lisse}^{1,p'}(X):= \varinjlim_{E/\Q_\ell} \varprojlim_n H^1_{\et}(X, \Z/h_E\times U^{(1)}_E/U^{(n)}_E),\]
\[X\mapsto H^1_{p^\infty}(X):= \varinjlim_{E/\Q_\ell} H^1_{\et}(X, \Z/p^{s_E})=H^1_{p^\infty}(X).\]

Let $L\in\Phi$. For $j\ge 0$ we define
\eq{para:rank12}{\fil{j}{\rm Lisse}^1(L):=
\begin{cases}
\Im({\rm Lisse}^1(\sO_L)\to {\rm Lisse}^1(L)), & \text{if } j=0,\\
{\rm Lisse}^{1,p'}(L)\oplus \fil{j}H^1_{p^\infty}(L), &\text{if } j\ge 1,
\end{cases}} 
where $\fil{j}H^1_{p^\infty}(L)= \cup_n \fil{j}H^1_{p^n}(L)$ is defined in \ref{para:Art-cond}.
\end{para}

\begin{cor}\label{cor:lisse1}
Let the notation be as in \ref{para:rank1} above.
Then
\begin{enumerate}[label=(\arabic*)]
\item\label{cor:lisse11} ${\rm Lisse}^1\in \RSC_{\Nis}$;
\item \label{cor:lisse12} the motivic conductor is given by 
\[ c^{{\rm Lisse}^1}_L(M)=\min\{j\ge 0\mid M\in \fil{j}{\rm Lisse}^1(L)\};\]
        furthermore it restricts to a level 1 conductor.
\item\label{cor:lisse13} let $X\in \Sm$ be proper over $k$ and $U\subset X$ dense open, then 
\[h^0_{\A^1}({\rm Lisse}^1)(U)={\rm Lisse}^{1,p'}(U)\oplus H^1_{p^\infty}(X),\] 
see \ref{para:HIsub} for notation.
\end{enumerate}
\end{cor}
\begin{proof}
Note ${\rm Lisse}^{1,p'}\in HI_\Nis$. Hence \ref{cor:lisse11} and \ref{cor:lisse12} follow
directly from Theorem \ref{thm:Art-motivic}
together with the Corollaries \ref{cor:HI} and Lemma \ref{lem:cond-sum}.
For \ref{cor:lisse13} observe that by Theorem \ref{thm:Art-motivic} and the definition of the Artin conductor,
we have    $H^1_{p^\infty}(\sO_L,\fm_L^{-1})=H^1_{p^\infty}(\sO_L)$; hence the statement follows from
Corollary \ref{cor:bir-inv-h0}.
 \end{proof}

\begin{remark}\label{rmk:tame-sh}
Let $U\in \Sm$ and denote by $\pi_1^{\rm ab, t}(U/k)$  the abelian tame fundamental group in the sense of
 \cite[7]{KeSch}; it is a quotient of $\pi_1^{\rm ab}(U)$. Denote by ${\rm Tame}^1(U)$ the subgroup of
${\rm Lisse}^1(U)$ consisting of those lisse sheaves of rank one whose corresponding representation factors via
$\pi^{\rm ab, t}_1(U/k)$. Then
\[h^0_{\A^1}({\rm Lisse}^1)(U)={\rm Tame}^1(U).\]
Indeed, we classically have ${\rm Tame}^1(C)={\rm Lisse}^{1,p'}(C)\oplus H^1_{p^\infty}(\ol{C})$, in case
$C\in \Sm$ is a curve over $k$ with smooth compactification $\ol{C}$. Hence
this $\subset$ inclusion follows from Corollary \ref{cor:lisse1}\ref{cor:lisse13} and 
the description of $\pi^{\rm ab, t}_1(U/k)$ via curve-tameness, see \cite{KeSch}.
The other inclusion follows from the $\A^1$-invariance of ${\rm Tame}^1$.
\end{remark}

\section{Torsors under finite group schemes over a perfect field}\label{sec:torsor}
In this section $k$ is a perfect field of positive characteristic $p$.
We fix an algebraic closure $\bar{k}$ of $k$. 
The term {\em $k$-group} is short for {\em commutative group scheme of finite type over $k$}.

\begin{lemma}\label{lem:Hfppf-PST}
Let $G$ be a finite $k$-group. 
Then there exists an exact sequence of sheaves on $({\rm Sch}/k)_{\fppf}$, 
the $\fppf$-site on $k$-schemes,
\eq{lem:Hfppf-PST1}{0\to G\to H_1\to H_2\to 0,}
with $H_i$, $i=1,2$, smooth $k$-groups. Furthermore, if we denote by
$u: ({\rm Sch}/k)_{\fppf}\to ({\rm Sch}/k)_{\et}$ the morphism from the fppf-site to the \'etale site,
then the above sequence induces a canonical isomorphism 
\eq{lem:Hfppf-PST2}{Ru_*G\cong [H_1\to H_2]}
in the derived category of abelian sheaves on $({\rm Sch}/k)_\et$. In particular,
for all $n\ge 0$ the presheaf on $\Sm$
\eq{lem:Hfppf-PST3}{\Sm\ni X\mapsto H^n(X_{\fppf}, G)\cong H^n(X_{\et}, H_1\to H_2),}
admits the structure of a presheaf with transfers. This transfers structure does not depend 
on the choice of the sequence \eqref{lem:Hfppf-PST1} (up to isomorphism).
\end{lemma}
\begin{proof}
By a result of Raynaud (see \cite[3.1.1]{BBM}), there exists a closed immersion $G\inj A$, with $A$ an abelian variety $A$.
By \cite[Exp {${\rm VI}_A$}, Thm 3.2]{SGA3I}, the fppf-quotient sheaf $(A/G)_{\fppf}$ is representable by a 
$k$-group $A/G$ and the quotient map $A\to A/G$ is finite and faithfully flat. Hence $A/G$ is reduced and hence 
a smooth $k$-group.
This shows the existence of a sequence \eqref{lem:Hfppf-PST1}.
By \cite[Thm (11.7)]{Gro-Brauer} a smooth $k$-group is acyclic for the direct image functor
\[u_*: {\rm Shv}(({\rm Sch}/k)_{\fppf})\to {\rm Shv}(({\rm Sch}/k)_{\et}).\]
 Hence \eqref{lem:Hfppf-PST1} is a 
$u_*$-acyclic resolution of the fppf-sheaf $G$, which yields the canonical isomorphism \eqref{lem:Hfppf-PST2}.
Since $H_1\to H_2$ is a complex of \'etale sheaves with transfers,  the presheaf \eqref{lem:Hfppf-PST3} has transfers,
by  \cite[Lem 6.23]{MVW}.
Finally, we have to show that this transfer structure does not depend on the resolution \eqref{lem:Hfppf-PST1}.
Assume $0\to G\to L_1\to L_2\to 0$ is a second such exact sequence. We obtain a commutative diagram
with exact rows in $({\rm Sch}/k)_{\fppf}$
\[\xymatrix{
0\ar[r] & G\ar[r]\ar@{=}[d] & L_1\times H_1\ar[r]\ar[d] & (L_1\times H_1)/G\ar[d]\ar[r] & 0\\
0\ar[r] & G\ar[r]                 &  H_1\ar[r]                     & H_2                               \ar[r] & 0,
}\]
where the vertical arrows are induced by projection and the top horizontal arrow on the left is the diagonal embedding of $G$;
we have also such a sequence with $H$  replaced by $L$ in the lower line.
This yields  the isomorphism $[H_1\to H_2]\cong [L_1\to L_2]$ in the derived category of  \'etale sheaves with transfers,
proving the final statement.
\end{proof}

\begin{nota}\label{nota:H1G}
Let $G$ be a finite $k$-group. Then we denote by $H^1(G)\in \PST$
the presheaf with transfers from Lemma \ref{lem:Hfppf-PST}, 
\[X\mapsto H^1(G)(X):=H^1(X_{\fppf},G).\]
\end{nota}

\begin{lemma}\label{lem:Gal-coh-proper}
Let $\Gal(\bar{k}/k)$ be the absolute Galois group of $k$ and $G$ an \'etale $k$-group.
Then the following functor defined by the Galois cohomology groups
\eq{lem:Gal-coh-proper1}{\Sm\ni X\mapsto H^n(\Gal(\bar{k}/k), G(X_{\bar{k}}))}
is a proper sheaf in $\RSC_\Nis$ in the sense of Definition \ref{defn:properRSC}.
\end{lemma}
\begin{proof}
The composition 
\[\Cor_k(X,Y)\to \Cor_{\bar{k}}(X_{\bar{k}}, Y_{\bar{k}})\to \Hom_{\Ab}(G(Y_{\bar{k}}), G(X_{\bar{k}}))\]
factors through  the homomorphism of Galois-modules; hence $\eqref{lem:Gal-coh-proper1} \in\PST$.
Since $G$ is \'etale we have $G(X_{\bar{k}})= G(\bar{k})^{\pi_0(X_{\bar{k}})}$. It follows that
\eqref{lem:Gal-coh-proper1} is $\A^1$-invariant and restrictions to dense open subsets are isomorphisms. 
Hence it is a Nisnevich sheaf and proper.
\end{proof}

\begin{lemma}\label{lem:HSS-PST}
Let $G$ be an \'etale  $k$-group. 
Then the exact sequence 
\[E(X)\,: \,0\to H^1(\Gal(\bar{k}/k), G(X_{\bar{k}}))\to H^1(G)(X)\to K^1(X)\to 0,\]
with
\[K^1(X):=\Ker(H^1(X_{\bar{k},\et}, G_{\bar{k}})^{\Gal(\bar{k}/k)}\to H^2(\Gal(\bar{k}/k), G(X_{\bar{k}}))),\]
coming from the $E_2$-page of the Hochschild-Serre spectral sequence, defines 
an exact sequence $X\mapsto E(X)$ in $\PST$. 
\end{lemma}
\begin{proof}
First note that by Grothendieck's theorem (see Lemma \ref{lem:Hfppf-PST}) we have $H^1(G)(X)=H^1(X_{\et}, G)$,
so that the sequence $E(X)$, indeed is induced by the Hochschild-Serre spectral sequence.
We show that transfers act on the whole spectral sequence.
By a limit argument it suffices to consider finite Galois extensions $L/k$ and the corresponding spectral sequence.
Let $G\to I^\bullet$ be an injective resolution in ${\rm Sh}_{\et}(\Cor_k)$, the category of \'etale sheaves with transfers.
Then 
\eq{lem:HSS-PST1}{  H^i(X_{\et}, G)= H^i(I^\bullet(X)) = H^i(I^\bullet(X_L)^{\Gal(L/k)}), \quad i\ge 0, }
for all $X\in\Sm$,  see \cite[Lem 6.23]{MVW}. Moreover,
$ H^i(X_{L,\et}, I^n)=0=H^i(X_{\et}, I^n)$, for $i\ge 1$ and $n\ge 0$, see {\em loc. cit.}
Hence
\eq{lem:HSS-PST2}{H^i(\Gal(L/k), I^n(X_L))=0.}
Let $C^\bullet(\Gal(L/k), M)$  be the complex of cochains computing the cohomology of the $\Gal(L/k)$-module $M$.
By \eqref{lem:HSS-PST1}, \eqref{lem:HSS-PST2} the cohomology groups $H^i_{\et}(X, G)$ are the cohomology groups of the
total complex associated to the double complex $C^\bullet(\Gal(L/k), I^\bullet(X_L))$.
The Hochschild-Serre spectral sequence arises from a filtration of this complex.
Furthermore, the canonical map $\Cor_k(X,Y)\times \Gal(L/k)\to \Cor_k(X_L, Y_L)$, 
$(\alpha,\sigma)\mapsto (\alpha\otimes_k L)\circ (\id_{X\times_k Y}\times \sigma)$ induces 
the structure of a complex of presheaves with $\Gal(L/k)$-equivariant transfers on $X\mapsto I^\bullet(X_L)$.
Hence
$X\mapsto C^\bullet(\Gal(L/k), I^\bullet(X_L))$
is a double complex in $\PST$. This proves the Lemma.
\end{proof}

\begin{lemma}\label{lem:et-pp}
Let $G$ be an \'etale $k$-group of order prime to $p$.
Then $H^1(G)\in\HI_\Nis$ (see \ref{nota:H1G} for notation).
\end{lemma}
\begin{proof}
In this case $G_{\bar{k}}$ is a constant finite $k$-group of order prime to $p$.
By \cite[Cor 5.29]{VoPST} the presheaf $X\mapsto K^1(X)$ from Lemma \ref{lem:HSS-PST} is $\A^1$-invariant
and by Lemma \ref{lem:H1Nis} and Lemma \ref{lem:Gal-coh-proper} it is a Nisnevich sheaf.
Thus the claim follows from the Lemmas \ref{lem:HSS-PST}, \ref{lem:Gal-coh-proper}.
\end{proof}

\begin{lemma}\label{lem:et-p}
Let $G$ be an \'etale $k$-group of $p$-primary order. Then $H^1(G)\in \RSC_{\Nis}$
and the motivic conductor $c^{H^1(G)}$ is given by
\[c^{H^1(G)}_L: H^1(G)(L)\to \bigoplus_i H^1_{\et}( \Spec L_i, G_{\bar{k}})
                                          \xr{\max_i\{c^{H^1(G_{\bar{k}})}_{L_i}\}} \N_0,\]
where $L\otimes_k \bar{k}= \prod_i L_i$ and $c^{H^1(G_{\bar{k}})}$ is computed in Theorem \ref{thm:Art-motivic}
(note that $G_{\bar{k}}=\oplus_j \Z/p^{n_j}$). In particular, $(c^{H^1(G)})^{\le 1}$ is a conductor.
Moreover, if $X$ is smooth proper and $U\subset X$ is dense open, then
$h^0_{\A^1}(H^1(G))(U)= H^1(G)(X)$ (see \ref{para:HIsub} for notation).
\end{lemma}
\begin{proof}
Note in this case $H^2(\Gal(\bar{k}/k), G(X_{\bar{k}}))=0$ (e.g. \cite[Exp X, Thm 5.1]{SGA4III}).
Thus the first statement follows from Lemma \ref{lem:HSS-PST}, Lemma \ref{lem:Gal-coh-proper}, 
Lemma \ref{lem:H1Nis}, Lemma \ref{lem:prop-exact}, Proposition \ref{prop:mot-cond-sub-sheaf},
Proposition \ref{prop:mot-cond-restr}, and Theorem \ref{thm:Art-motivic}. For the final statement observe that
\[\widetilde{H^1(G)}(\sO_L,\fm_L^{-1})= H^1(G)(\sO_L).\]
This follows directly from the explicit description of the motivic conductor on $H^1(G_{\bar{k}})$ in Theorem \ref{thm:Art-motivic}.
Hence the final statement follows from Corollary \ref{cor:bir-inv-h0}.
\end{proof}

\begin{lemma}\label{lem:HSS-inf}
Let $G$ be an infinitesimal finite $k$-group. 
Then  
\[H^1(X_{\fppf}, G)\cong H^1(X_{\bar{k}, \fppf}, G_{\bar{k}})^{\Gal(\bar{k}/k)}, \quad \text{for all }X\in \Sm.\]
Furthermore, this isomorphism induces an isomorphism in $\NST$ (cf. Proposition \ref{prop:mot-cond-restr} for notation)
\[H^1(G)\cong (R_{\bar{k}/k}H^1(G_{\bar{k}}))^{\rm \Gal(\bar{k}/k)}.\]
\end{lemma}
\begin{proof}
Since $G$ is infinitesimal, we have $G(Y)=0$ for all reduced schemes $Y$ over $k$.
There is also a Hochschild-Serre spectral sequence for the fppf-cohomology (e.g., \cite[III, Rem 2.21]{MilneEt});
 by the above remark the fppf-version  
of the exact sequence $E(X)$ from Lemma \ref{lem:HSS-PST} yields the first isomorphism.
By Lemma \ref{lem:Hfppf-PST} this isomorphism is compatible with the transfer structure. It remains to
show that $H^1(G)$ is a Nisnevich sheaf.  By the remark from the beginning of this proof 
any sequence \eqref{lem:Hfppf-PST1} yields an injection $H_1\inj H_2$ when restricted to $\Sm$.
Thus the isomorphism \eqref{lem:Hfppf-PST2} implies 
\[Ru_* G\cong (H_2/H_1)_{\et}[-1]\]
in the derived category of \'etale sheaves on $\Sm$, where  $(H_2/H_1)_{\et}$ denotes the \'etale
sheafification of the presheaf $X\mapsto H_2(X)/H_1(X)$. Hence
\[H^1(G)(X)= H^0(X, (H_2/H_1)_{\et}).\]
It follows that $H^1(G)$ is even an \'etale sheaf.
\end{proof}

\begin{lemma}\label{lem:multA1}
Assume $G$ is an infinitesimal  finite $k$-group of multiplicative type.
Then $H^1(G)\in \HI_{\Nis}$.
\end{lemma}
\begin{proof}
By Lemma \ref{lem:HSS-inf} we may assume $k=\bar{k}$.
In this  case $G$ is diagonalizable and we find an exact sequence \eqref{lem:Hfppf-PST1}
with $H_i= \G_m^{n_i}$, some $n_i\ge 1$, see \cite[IV, \S1, 1.5 Cor]{DG}. 
The statement follows from the $\A^1$-invariance of $X\mapsto H^i(X_{\Zar},\G_m)$, $i=0,1$, and Hilbert 90.
\end{proof}

\begin{para}\label{para:ap}
We denote
\[\alpha_p:=\Ker (F: \G_a\to \G_a),\]
where $F$ is the absolute Frobenius on the additive group.
Then $\alpha_p$ is a unipotent infinitesimal finite  $k$-group.
Let $L\in\Phi$ and let $t\in \sO_L$ be a local parameter.
Recall from \ref{para:BMK} that $\fil{j}\G_a(L):= \fil{j}W_1(L)$ is given by
\eq{para:ap3}{\fil{j}\G_a(L)=\begin{cases}
\sO_L, &\text{if }j=0\\
\frac{1}{t^{j-1}}\cdot \sO_L, & \text{if } (j,p)=1\\
\frac{1}{t^j}\cdot \sO_L, & \text{if } p\mid j.
\end{cases}}
We denote by 
\eq{para:ap4}{\fil{j}H^1(\alpha_p)(L)}
the image of $\fil{j}\G_a(L)$ under the connecting homomorphism
\[\delta: \G_a(L)\to H^1(\alpha_p)(L)= H^1(\Spec L_{\fppf}, \alpha_p).\]
Note that $\fil{j}H^{1}(\alpha_p)(L)$ is also equal to the image of the Frobenius saturated filtration $\filF{j} W_1(L)$.
\end{para}

\begin{prop}\label{prop:ap-mot-cond}
We have $H^1(\alpha_p)\in\RSC_{\Nis}$ and the motivic conductor $c^{H^1(\alpha_p)}$ on $H^1(\alpha_p)$ is given by 
\eq{prop:ap-mot-cond1}{c^{H^1(\alpha_p)}_L(b)=\min\{j\ge 0\mid b\in \fil{j}H^1(\alpha_p)(L)\}.}
In particular, either $b\in H^1(\alpha_p)(\sO_L)$ or $c^{H^1(\alpha_p)}_L(b)\ge 2$.
Furthermore, it restricts to a level 2 conductor.
\end{prop}
\begin{proof}
Denote the collection of maps $H^1(\alpha_p)(L)\to \N_0$
defined by the right hand side of \eqref{prop:ap-mot-cond1} by $c$.
By Proposition \ref{prop:cond-Witt} and Lemma \ref{lem:im-cond}, 
$c$ satisfies \ref{c1}-\ref{c6} except possibly for \ref{c4}.
(For \ref{c5} note, that $\G_a(Y)\to H^1(\alpha_p)(Y)$ is surjective for any affine scheme $Y$ over $k$.)
We claim that $c^{\le 2}$ satisfies \ref{c4}.
Let $X$ be a smooth $k$-scheme and $b\in H^1(\alpha_p)(\A^1_X)$ with
\eq{prop:ap-mot-cond2}{c_{k(x)_\infty}(\rho_x^*b)\le 1,\quad \text{for all } x\in X\text{ with }\td(k(x)/k)\le 1,}
where $\rho_x: \Spec k(x)(t)_\infty=\Spec \Frac(\sO_{\P^1_{x},\infty}^h)\to \A^1_X$ is the natural map.
We want to show : $b\in H^1(\alpha_p)(X)$.
This is  equivalent to $b=\pi^*i^*b$ in $H^1(\alpha_p)(\A^1_X)$; 
by the definition of $c$ and Lemma \ref{lem:HSS-inf}, we can therefore assume $k$ is algebraically closed.
Furthermore, the question is local  on $X$, hence we can assume $X=\Spec A$ affine.
Note, for a general $\beta\in H^1(\alpha_p)(L)\setminus H^1(\alpha_p)(\sO_L)$  we have $c_L(\beta)\ge 2$, as
follows directly from \eqref{para:ap3}.
Hence condition \eqref{prop:ap-mot-cond2} implies
\[\rho_x^*b\in \Im(H^1(\alpha_p)(\sO_{\P^1_x,\infty}^h) \to H^1(\alpha_p)(k(x)(t)_\infty)).\]
Denote by $b(x)$ the restriction of $b$ to $\A^1_x$. Since $H^1(\alpha_p)$ is a Nisnevich sheaf we conclude 
\[b(x)\in H^1(\alpha_p)(\P^1_x)=H^1(\alpha_p)(x).\]
Thus we find a polynomial $\tilde{b}=b_0+b_1 t+\ldots+ b_n t^n\in A[t]$
 mapping to $b$ such that for all points $x\in X$ with $\td(k(x)/k)\le 1$
there exist  $c_x\in k(x)$ and $g_x\in k(x)[t]$ with
\eq{prop:ap-mot-cond3}{\tilde{b}(x)= c_x + g_x^p, \quad \text{in }k(x)[t].}
It follows immediately that $\tilde{b}\in A[t^p]$ and it remains to show $b_i\in A^p$, for all $i\ge 1$,
since then $b=b_0$ in $H^1(\alpha_p)(\A^1_X)$. Thus we are reduced to show the following:
Let $X=\Spec A\to \A^d=\Spec k[x_1,\ldots, x_d]$ be an \'etale map and $a\in A\setminus A^p$. 
Then there exists a smooth connected curve $i: C\inj X$ such that  $i^*a\in \sO(C)\setminus \sO(C)^p$.
If $a\not\in A^p$ we find a variable - say $x_1$ - such that 
$a= a_0+ a_1 x_1+\ldots +a_n x_1^n$, where $a_i\in A^p[x_2,\ldots, x_d]:=B$ and $a\not\in B[x_1^p]$.
A tuple $\lambda=(\lambda_2,\ldots, \lambda_d)\in k^{d-1}$ induces a closed immersion
$i_\lambda : \A^1\to \A^d$ given by $x_1\mapsto x_1$, $x_i\mapsto \lambda_i$, $i=2,\ldots, d$.
Denote by $C_\lambda$ the pullback of $X$ along $i_\lambda$. Since $k$ is algebraically closed we  find  a tuple
$\lambda$ such that $a_{|C_\lambda}\not\in\sO(C_\lambda)^p$. This proves the above claim; hence $c^{\le 2}$ satisfies \ref{c4}.

Corollary \ref{cor:mot-cond-mini} yields $c^{H^1(\alpha_p)}\le c$.
To show the other inequality it suffices by Corollary \ref{cor:localLS3}
to show the following:
Let $L\in\Phi$, $t\in \sO_L$ a local parameter, and
let $\sigma: K\inj \sO_L$ be some coefficient field; extend it in the canonical way to $\sigma: K(x)\inj \sO_{L_x}$,
where $L_x=\Frac(\sO_L[x]_{(t)}^h)$. Assume $b\in \fil{r}H^1(\alpha_p)(L)$, $r\ge 1$.
Then the following implication holds
\eq{prop:ap-mot-cond4}{(b, 1-xt^{r-1})_{L_x,\sigma}=0 \quad \text{for all }\sigma\,\Rightarrow\, b\in \fil{r-1}H^1(\alpha_p)(L),}
where the local symbol on the left hand side is the one from \ref{para:symbol-hdvf} for $H^1(\alpha_p)$, and $\sigma$
runs through all coefficient fields of $\sO_L$.
By \ref{LS6} the local symbol on $H^1(\alpha_p)$ is given by
\[(b, 1-xt^{r-1})_{L_x,\sigma} =\delta(\Res_{t,\sigma}(\tilde{b}\dlog (1-xt^{r-1}))),\]
where $\tilde{b}\in \fil{r}\G_a(L)$ is a lift of $b$, $\delta: \G_a(K(x))\to H^1(\alpha_p)(K(x))$ 
is the connecting homomorphism, and we use the isomorphism $L_x=K(x)((t))$ defined by $\sigma$ and $t$
to compute the residue symbol on the right.
To prove the implication \eqref{prop:ap-mot-cond4} it suffices to consider $b$ modulo $\fil{r}$.
Fix $\sigma: K\inj \sO_L$.

{\em 1st case: $(r,p)=1=(r-1,p)$.} In this case $\tilde{b}\equiv c/t^{r-1}$ mod $\fil{r-1}\G_a(L)$, for some $c\in K$.
Hence 
\[\Res_{t,\sigma}(\tilde{b}\dlog (1-xt^{r-1}))=-(r-1)cx.\]
Since $\delta(-(r-1)cx)=0$ iff $cx\in K(x)^p$, this is only possible if $c=0$.

{\em 2nd case:  $p\mid r-1$.} In this case $\fil{r}H^1(\alpha_p)(L)=\fil{r-1}H^1(\alpha_p)(L)$, 
and there is nothing to show.

{\em 3rd case: $p\mid r$.}  In this case $\tilde{b}\equiv  c/t^{r-1}+e/t^r$ mod $\fil{r-1}\G_a(L)$, for some $c, e\in K$.
By the same argument as in the 1st case we obtain the following implication
\[(b, 1-xt^{r-1})_{L_x,\sigma}=0 \, \Rightarrow \, 
(\tilde{b}, 1-xt^{r-1})_{L_x,\sigma}=0 \quad \text{in } \G_a(K(x)).\]
Since this hold for all $\sigma$, Proposition \ref{prop:Wfil-symb} (in the case $n=1$) yields
$\tilde{b}\in \filF{r-1}\G_a(L)$, hence $b\in \fil{r-1}H^1(\alpha_p)(L)$.
This completes the proof.
\end{proof}

\begin{proposition}\label{prop:infu-RSC}
Let $G$ be a finite unipotent infinitesimal $k$-group.
\begin{enumerate}[label= (\arabic*)]
\item\label{prop:infu-RSC1} $H^1(G)\in \RSC_{\Nis}$;
\item\label{prop:infu-RSC2} the motivic conductor $c^{H^1(G)}$ restricts to a level 2 conductor;
\item\label{prop:infu-RSC3} if $X$ is a proper smooth $k$-scheme and $U\subset X$ is open dense,
then $h^0_{\A^1}(H^1(G))(U)= H^1(G)(X)$ (see \ref{para:HIsub} for notation).
\end{enumerate}
\end{proposition}
\begin{proof}
\ref{prop:infu-RSC1}. We find an exact sequence in the category of $k$-groups 
\[0\to G\to H_1\to H_2\to 0\]
with $H_i$ smooth unipotent $k$-groups. Indeed, by \cite[V, \S 1, 4.2, 4.7]{DG} we find a closed immersion  
$G\inj W_n^N:=H_1$, for some $n, N$, and by \cite[IV, \S2, 2.3]{DG} the quotient $H_2:=H_1/G$ is again unipotent,
and it is automatically reduced, hence is smooth. As in the proof of Lemma \ref{lem:HSS-inf} we find 
$H^1(G)\cong (H_2/H_1)_{\et}$, where $(H_2/H_1)_{\et}$ is the \'etale sheaf associated to the presheaf
$\Sm\ni X\mapsto H_2(X)/H_1(X)$. Let $v: \Sm_{\et}\to \Sm_{\Nis}$ be the natural morphism of sites.
Since $H_1$ is smooth unipotent, it  is a successive extension of $\G_a$'s, hence $R^1v_*H_1=0$.
We obtain an isomorphism in $\NST$
\[H^1(G)\cong (H_2/H_1)_{\Nis},\]
where $(H_2/H_1)_{\Nis}$ is the Nisnevich sheaf associated to the presheaf $X\mapsto H_2(X)/H_1(X)$.
Thus $H^1(G)\in \RSC_{\Nis}$ follows from  $H_i\in \RSC_{\Nis}$ and \cite[Thm 0.1]{Saito-Purity-Rec},
which states that Nisnevich sheafification preserves SC-reciprocity.

\ref{prop:infu-RSC2}. By \cite[IV, 5.8]{DG} $G$ admits a descending sequence 
\eq{prop:infu-RSC4}{0=G_n\subset G_{n-1}\subset \ldots\subset G_0=G }
with successive quotients $G_{r-1}/G_r\cong \alpha_p$.
In particular, $H^2(X_{\fppf}, G)=0$, for all {\em affine} smooth $k$-schemes $X$.
Note that this induces for all $r\in [1,n]$  an exact sequence in $\NST$
\eq{prop:infu-RSC5}{0\to H^1(G_r)\to H^1(G_{r-1})\to H^1(\alpha_p)\to 0.}
Indeed, by Lemma \ref{lem:HSS-inf} this sequence is in $\NST$; hence it suffices to check
its exactness on any smooth affine $k$-scheme $X$, in which case it follows from 
$H^0(X_{\fppf}, \alpha_p)=0=H^2(X_{\fppf}, G_r)$. By Proposition \ref{prop:ap-mot-cond}
the motivic conductor of $H^1(\alpha_p)$ restricts to a level 2 conductor and by induction we may assume
that so does the motivic conductor of $H^1(G_{r-1})$. We deduce that the motivic conductor of
$H^1(G_r)$ restricts to a level 2 conductor from \eqref{prop:infu-RSC5} and a similar argument as at the end of the proof of
Proposition \ref{prop:Art-cond}.

\ref{prop:infu-RSC3}. We claim
\eq{prop:infu-RSC6}{\widetilde{H^1(G)}(\sO_L, \fm^{-1}_L)=H^1(G)(\sO_L).}
The claim is true for $G=\alpha_p$, by the explicit formula of the motivic conductor in Proposition \ref{prop:ap-mot-cond}.
Consider the sequence $\eqref{prop:infu-RSC4}$ and assume the claim is proven for $G_r$.
Let $b\in \widetilde{H^1(G_{r-1})}(\sO_L, \fm^{-1})$. By the exact sequence \eqref{prop:infu-RSC5} and the claim
for $\alpha_p$ we find a $c\in H^1(G_{r-1})(\sO_L)$ such that $b-c$ is in the image of $H^1(G_r)(L)$. 
By Proposition \ref{prop:mot-cond-sub-sheaf} we find
\[b-c\in \widetilde{H^1(G_{r})}(\sO_L, \fm^{-1})= H^1(G_r)(\sO_L),\]
which proves \eqref{prop:infu-RSC6}. Hence \ref{prop:infu-RSC3} follows from Corollary \ref{cor:bir-inv-h0}.
\end{proof}

In summary:
\begin{thm}\label{thm-H1G-RSC}
Let $G$ be a finite $k$-group. Then:
\begin{enumerate}[label=(\arabic*)]
\item\label{thm-H1G-RSC1} $H^1(G)\in \RSC_{\Nis}$;
\item\label{thm-H1G-RSC2} 
the motivic conductor of $H^1(G)$ restricts to conductor  of level 2, and if $G$ has no infinitesimal unipotent factor,
       to a conductor of level 1;
\item\label{thm-H1G-RSC3}
 write $G=G'\times G_{\rm unip}$ with $G_{\rm unip}$ unipotent and $G'$ without any unipotent subgroup, and let
         $X$ be smooth proper over $k$ and $U\subset X$ dense open. Then 
\[h^0_{\A^1}(H^1(G))(U)= H^1(G')(U)\oplus H^1(G_{\rm unip})(X).\]
\end{enumerate}
\end{thm}
\begin{proof}
By \cite[IV, \S 3, 5.9]{DG} we can decompose $G$ uniquely into a product
\[G= G_{em}\times G_{eu}\times G_{im}\times G_{iu},\]
where $G_{em}$ is \'etale multiplicative, i.e., it is an \'etale $k$-group without $p$-torsion,
$G_{eu}$ is \'etale unipotent, i.e., it is an  \'etale $k$-group with $p$-primary torsion,
$G_{im}$ is infinitesimal and of multiplicative type, and $G_{iu}$ is an infinitesimal unipotent $k$-group.
Hence the statement follows from Lemma \ref{lem:et-pp}, Lemma \ref{lem:et-p}, Lemma \ref{lem:multA1},
and  Proposition \ref{prop:infu-RSC}.
\end{proof}

\begin{remark}\label{rmk:H1unip-birinv}
Let $G$ be a finite unipotent $k$-group.
Note that by  Theorem \ref{thm-H1G-RSC}\ref{thm-H1G-RSC3} above, the functor
$X\mapsto H^1(X_{\fppf}, G)$ is a birational invariant for smooth proper $k$-schemes.
This gives a new proof of this (probably) well-known result (it follows, e.g., also from \cite{CR11}).
\end{remark}

\medskip
\medskip


\providecommand{\bysame}{\leavevmode\hbox to3em{\hrulefill}\thinspace}
\providecommand{\MR}{\relax\ifhmode\unskip\space\fi MR }
\providecommand{\MRhref}[2]{%
  \href{http://www.ams.org/mathscinet-getitem?mr=#1}{#2}
}
\providecommand{\href}[2]{#2}

\end{document}